\documentclass[a4paper,12pt,dvips]{amsart}    
\usepackage{amsmath,supertabular} 
\usepackage[sectionbib,square]{natbib}
\usepackage[all]{xy}
\usepackage{amssymb,stmaryrd}
\usepackage[margin=1in]{geometry}
\usepackage{appendix}
\usepackage[sectionbib,square]{natbib}
\usepackage[dvips,pagebackref,colorlinks=true,pdftitle={Moduli Spaces of Semistable Sheaves on Projective Deligne-Mumford Stacks},
pdfauthor={Aeshma},pdfpagemode=None]{hyperref}

\usepackage{fancyhdr}
\newcommand{\coloneqq}{:=}
\newcommand{\eqqcolon}{=:}
\newcommand{\gr}{gr}

\newcommand{\EE}{\mathcal E}
\newcommand{\FF}{\mathcal F}
\newcommand{\OO}{\mathcal O}
\newcommand{\QQ}{\mathcal Q}
\newcommand{\GG}{\mathcal G}
\newcommand{\XX}{\mathcal X}
\newcommand{\UU}{\mathcal U}

\newcommand{\HH}{\mathcal H}

\newcommand{\Gm}{{\mathbb G}_m}

\newcommand{\fs}{\!\!\!\fatslash}
\newcommand{\leftsub}[2]{\vphantom{#2}_{#1}{#2}}
\DeclareMathOperator{\Grass}{Grass}
\DeclareMathOperator{\Quot}{Quot}

\DeclareMathOperator{\GL}{GL}
\DeclareMathOperator{\PGL}{PGL}
\DeclareMathOperator{\SL}{SL}
\DeclareMathOperator{\reg}{reg}
\DeclareMathOperator{\Proj}{Proj}
\DeclareMathOperator{\rk}{rk}

\newcommand{\pp}{\mathbb{P}} 

\DeclareMathOperator{\Tor}{Tor}%
\DeclareMathOperator{\supp}{supp}%
\DeclareMathOperator{\Hom}{Hom}%
\DeclareMathOperator{\Ext}{Ext}%
\DeclareMathOperator{\coker}{coker}%
\DeclareMathOperator{\ev}{ev}%
\DeclareMathOperator{\Id}{id}%
\DeclareMathOperator{\et}{\text{\'et}}%
\DeclareMathOperator{\Spec}{Spec}%
\DeclareMathOperator{\spec}{Spec}
\DeclareMathOperator{\Aut}{Aut}

\DeclareMathOperator{\ch}{ch}
\DeclareMathOperator{\Td}{Td}
\DeclareMathOperator{\Ass}{Ass}
\DeclareMathAlphabet{\mathpzc}{OT1}{pzc}{m}{it}

\DeclareMathOperator{\HOM}{\mathcal{H}\!\mathpzc{om}}
\DeclareMathOperator{\ISO}{\mathcal{I}\!\mathpzc{so}}
\DeclareMathOperator{\END}{\mathcal{E}\!\mathpzc{nd}}
\DeclareMathOperator{\EXT}{\mathcal{E}\!\mathpzc{xt}}
\DeclareMathOperator{\AUT}{\mathcal{A}\!\mathpzc{ut}}

\DeclareMathOperator{\stab}{stab}

\theoremstyle{plain}
\newtheorem{thm}{Theorem}[section]
\newtheorem*{thmn}{Theorem}
\newtheorem{cor}[thm]{Corollary}
\newtheorem{lem}[thm]{Lemma}
\newtheorem{prop}[thm]{Proposition}
\theoremstyle{definition} 
\newtheorem{defn}[thm]{Definition}

\newtheorem{exmp}[thm]{Example}%
\newtheorem*{ass}{Assumption}

\theoremstyle{remark}
\newtheorem{rem}[thm]{Remark}

\def\cf{\textit{cf.}\kern.3em}

\def\ie{\textit{i.e.},\ }

\begin{document}

  \title{Moduli Spaces of Semistable Sheaves \\ on Projective Deligne-Mumford Stacks}
  \author{Fabio Nironi}
 \email{fabio.nironi@gmail.com}
  \address{Columbia University, 2990 Broadway,  New York, NY 10027}

\thanks{The author has been supported by a PhD grant of the Scuola Internazionale Superiore di Studi Avanazati.}

 \begin{abstract}
  In this paper we introduce a notion of Gieseker stability for coherent sheaves on tame Deligne-Mumford stacks with projective moduli scheme and some chosen generating sheaf on the stack in the sense of Olsson and Starr \cite{MR2007396}. We prove that this stability condition is open, and pure dimensional semistable sheaves form a bounded family. We explicitly construct the moduli stack of semistable sheaves as a finite type global quotient, and study the moduli scheme of  stable sheaves and its natural compactification in the same spirit as the seminal paper of Simpson \cite{MR1307297}. With this general machinery we are able to retrieve, as special cases, results of Lieblich \cite{MR2309155} and Yoshioka \cite{MR2306170} about moduli of twisted sheaves and results of Maruyama-Yokogawa \cite{MR1162674} about moduli of parabolic bundles.
\end{abstract}

\maketitle

\section*{Overview}

We define a notion of stability for coherent sheaves on stacks, and construct a moduli stack of semistable sheaves. The class of stacks that is suitable to approach this problem is the class of projective stacks:  tame stacks (for instance Deligne-Mumford stacks in characteristic zero) with projective moduli scheme and a  locally free sheaf that is ``very ample'' with respect to the map to the moduli scheme (a \textit{generating sheaf} in the sense of \cite{MR2007396}). The hypothesis of tameness lets us reproduce useful scheme-theoretic results such as a \textit{cohomology and base change} theorem, \textit{semicontinuity} for cohomology and $\Ext$ functors and other results related to flatness.  The class of projective stacks includes for instance every DM toric stack with projective moduli scheme and more generally every  smooth DM stack proper over an algebraically closed field  with projective moduli scheme.
 We also introduce a notion of \textit{family of projective stacks} parameterized by a noetherian finite-type scheme: it is a separated tame global quotient whose geometric fibers are projective stacks. These objects will play the role of projective morphisms.

In the first section we recall the notion of tame stack and projective stack and some results about their geometry taken from \cite{MR2427954}, \cite{geomDM} and \cite{MR2007396}. Moreover we collect all the results about flatness, like cohomology and base change and semicontinuity for cohomology, that we are going to use in the following. In the second section we prove that the stack of coherent sheaves on a projective Deligne-Mumford stack is algebraic and we provide it with an explicit smooth atlas . The algebraicity of the stack is an already known result and it is stated in more generality (no projectivity is required) in \cite{MR2233719}; however, we have decided to include this proof since it is elementary explicit and  it is also a first example of a practical usage of generating sheaves on a stack. 

Let $\XX\to S$ be a family of projective Deligne-Mumford stacks with moduli scheme $X$, a chosen polarization $\OO_X(1)$ and $\EE$ a generating sheaf of $\XX$; we will call this collection of data a projective stack with a chosen polarization $(\OO_X(1),\EE)$. We denote with $\QQ_{N,m}\coloneqq\Quot_{\XX/S}(\EE^{\oplus N}\otimes\pi^\ast\OO_X(-m))$ the functor of flat quotient sheaves on $\XX$ of the locally free sheaf $\EE^{\oplus N}\otimes\pi^\ast\OO_X(-m))$ ($N,m$ are integers). It is proven in \cite{MR2007396} that this functor is representable and a disjoint union of projective schemes on $S$. 
\begin{thmn}[\ref{prop:alg-stack-coherent}]
Let $\mathcal{U}_{N,m}$ be the universal quotient sheaf of $\QQ_{N,m}$.
  For every couple of  integers $N,m$ there is an open subscheme $\QQ^0_{N,m}\subseteq\QQ_{N,m}$ (possibly empty) such that:
  \begin{displaymath}
\xymatrix{
    \coprod_{N,m}\QQ^0_{N,m}\subseteq\coprod_{N,m}\QQ_{N,m}\ar[rr]^-{\coprod \mathcal{U}_{N,m}} && \mathfrak{Coh}_{\XX/S} \\
} 
\end{displaymath}
is a smooth atlas.

The stack $\mathfrak{Coh}_{\XX/S}$ is a locally finite-type Artin stack.
\end{thmn}

Sections $3-6$ contain the definition of stability and the study of the moduli space of semistable sheaves.
The motivation for the kind of stability we propose  comes from studies of stability in two well known examples  of \textit{decorated sheaves} on projective schemes that can be interpreted as sheaves on algebraic stacks: \textit{twisted sheaves} and  \textit{parabolic bundles}. In the case of twisted sheaves it is possible to associate to a projective scheme $X$ and a chosen twisting cocycle $\alpha\in H^2_{\et}(X,\Gm)$, an abelian $\Gm$-gerbe $\mathcal{G}$ on $X$ such that the category of coherent $\alpha$-twisted sheaves on $X$ is equivalent to the category of coherent sheaves on $\mathcal{G}$ (Donagi-Pantev \cite{donagi-2003}, C\u{a}ld\u{a}raru \cite{cualduararu-dcotsocm2000},  Lieblich \cite{MR2309155}). In the case of a parabolic bundle on $X$, with parabolic structure defined by an effective Cartier divisor $D$ and some rational weights, it is possible to construct an algebraic stack whose moduli scheme is $X$ by a root construction. It was proven (Biswas \cite{MR1455522}, Borne \cite{borne-2006}, \cite{borne-2007}) that the category of parabolic bundles on $X$ with parabolic structure on $D$ and fixed rational weights is equivalent to the category of vector bundles on the associated root stack.

Since intersection theory on algebraic stacks was established in \cite{Vitas-1989} and \cite{MR1719823} it is possible to define $\mu$-stability for a stack in the usual way. It is proven in \cite{borne-2006} that the degree of a sheaf on a root stack is the same as the parabolic-degree defined in \cite{MR1162674} and used there (and also in \cite{MR575939}) to study stability. It is also well known that the degree of a sheaf on a gerbe, banded by a cyclic group, can be used to study stability \cite{MR2309155}, or equivalently it can be defined a modified degree for the corresponding twisted sheaf on the moduli scheme of the gerbe \cite{MR2306170}. From these examples it looks reasonable that the degree of a sheaf on a stack could be a good tool to study stability, and we are lead to think  that the na\"{\i}ve definition of $\mu$-stability should work in a broad generality. 

However it is already well known that a Gieseker stability defined in the na\"{\i}ve way doesn't work. Let $\XX$ be a projective Deligne-Mumford stack with moduli scheme $\pi\colon\XX\to X$ and $\OO_X(1)$ a polarization of the moduli scheme and $\FF$ a coherent sheaf on $\XX$. Since $\pi_\ast$ is exact and preserves coherence of sheaves and cohomology groups,  we can define a Hilbert polynomial:
\begin{displaymath}
  P(\FF,m)=\chi(\XX,\FF\otimes\pi^\ast\OO_X(m))=\chi(X,\pi_\ast\FF(m))
\end{displaymath} 
We could use this polynomial to define Gieseker stability in the usual way. We observe immediately that in the case of gerbes banded by a cyclic group this definition is not reasonable at all. A quasicoherent sheaf on such a gerbe splits in a direct sum of eigensheaves of the characters of the cyclic group, however every eigensheaf with non trivial character does not contribute to the Hilbert polynomial and eventually semistable sheaves on the gerbe, according to this definition, are the same as semistable sheaves on the moduli scheme of the gerbe.
In the case of root stacks there is a definition of a parabolic Hilbert polynomial and a parabolic Gieseker stability (see \cite{MR1162674}) which is not the na\"{\i}ve Hilbert polynomial or equivalent to a na\"{\i}ve Gieseker stability; moreover it is proven in \cite{MR1162674} and in \cite{borne-2006} that the parabolic degree can be retrieved from the parabolic Hilbert polynomial, while it is quite unrelated to the na\"{\i}ve Hilbert polynomial.
We introduce a new notion of Hilbert polynomial and Gieseker stability which depends not only on the polarization of the moduli scheme, but also on a chosen generating sheaf on the stack (see Def \ref{def:generating-sheaf}). If $\EE$ is a generating sheaf on $\XX$  we define a functor from $\mathfrak{Coh}_{\XX/S}$ to $\mathfrak{Coh}_{X/S}$:
\begin{displaymath}
  F_\EE\colon \FF \mapsto F_\EE(\FF)=\pi_\ast\HOM_{\OO_\XX}(\EE,\FF) 
\end{displaymath}
and the \textit{modified Hilbert polynomial}:
\begin{displaymath}
   P_{\EE}(\FF,m)=\chi(\XX,\HOM_{\OO_\XX}(\EE,\FF)\otimes\pi^\ast\OO_X(m))=\chi(X,F_{\EE}(\FF)(m))
\end{displaymath}
which is a polynomial if $\XX$ is tame and the moduli space of $\XX$ is a projective scheme. Using this polynomial we can define a Gieseker stability in the usual way. It is also easy to prove that given $\XX$ with orbifold structure along an effective Cartier divisor, there is a choice of $\EE$ such that this is the parabolic stability, and if $\XX$ is a gerbe banded by a cyclic group this is the same stability condition defined in \cite{MR2309155} and \cite{MR2306170} (the twisted case is developed with some detail in the appendix). There is also a wider class of examples where the degree of a sheaf can be retrieved from this modified Hilbert polynomial (see proposition \ref{prop:degree-mu}).

 In order to prove that semistable sheaves form an algebraic stack we need to prove that Gieseker stability is an open condition. To prove that the moduli stack of semistable sheaves is a finite type global quotient we need to prove that semistable sheaves form a bounded family. 
To achieve these results  we first prove a version of the well known Kleiman criterion,  suitable for sheaves on stacks, that is Theorem \ref{thm:kleiman-criterion}.  In particular we prove that a set-theoretic family $\mathfrak{F}$ of sheaves on a projective stack $\XX$ is bounded if and only if the family $F_{\EE}(\mathfrak{F})$ on the moduli scheme $X$ is bounded. From this result follows that the stack of semistable sheaves with fixed modified Hilbert polynomial is open in the stack of coherent sheaves, in particular it is algebraic. 
We are then left with the task of proving that the family of semistable sheaves is mapped by the functor $F_{\EE}$ to a bounded family.
First we prove that the functor $F_{\EE}$ maps pure dimensional sheaves to pure dimensional sheaves of the same dimension (Proposition \ref{prop:pure-sheaves}): it preserves the torsion filtration. However it doesn't map semistable sheaves on $\XX$ to semistable sheaves on $X$: it preservers neither the Harder-Narasimhan nor the Jordan-H\"older filtration. For this reason the boundedness of the family $F_{\EE}(\mathfrak{F})$ is not granted for free. 

Given $\FF$ a semistable sheaf on $\XX$ with chosen modified Hilbert polynomial, we study the maximal destabilizing subsheaf of $F_{\EE}(\FF)$ and prove that its slope has an upper bound which doesn't depend on the sheaf $\FF$. This numerical estimate, together with the Kleiman criterion for stacks and results of Langer \cite{MR2051393} and \cite{MR2085175} (applied on the moduli scheme), is enough to prove that semistable sheaves on a projective  stack with fixed modified Hilbert polynomial form a bounded family. The theorem of Langer we use here, replaces the traditional Le Potier-Simpson's result \cite[Thm 3.3.1]{MR1450870} in characteristic zero.  

The result of boundedness leads in section $5$ to the following  explicit construction of the moduli stack of semistable sheaves as a global quotient of a quasiprojective scheme  by the action of a reductive group. 
Let $\XX$ be a projective Deligne-Mumford stack over an algebraically closed field $k$ with a chosen polarization $(\EE,\OO_X(1))$ where $\EE$ is a generating sheaf and $\OO_X(1)$ is a very ample line bundle on the moduli scheme $X$. Fix an integer $m$, such that semistable sheaves on $\XX$ with chosen modified Hilbert polynomial $P$ are $m$-regular. Denote with $V$ the linear space $k^{\oplus N}\cong H^0(X,F_{\EE}(\FF)(m))$ where $N=h^0(X,F_{\EE}(\FF)(m))=P(m)$ for every semistable sheaf $\FF$.
\begin{thmn}[\ref{prop:stack-dei-moduli}]
  There is an open subscheme $\QQ$ in $\Quot_{\XX/k}(V\otimes\EE\otimes\pi^\ast\OO_X(-m),P)$, such that the algebraic stack of pure dimensional semistable sheaves on $\XX$ with modified Hilbert polynomial $P$ is the global quotient:
  \begin{displaymath}
    [\QQ/GL_{N,k}]\subseteq [\Quot_{\XX/k}(V\otimes\EE\otimes\pi^\ast\OO_X(-m),P)/GL_{N,k}]
  \end{displaymath}
where the group $\GL_{N,k}$ acts in the evident way on $V$.
\end{thmn}

In the sixth section we study the quotient $\QQ/GL_{N,k}$ using GIT techniques. We prove that the open substack of pure stable sheaves has a moduli scheme which is a quasiprojective scheme, while the whole GIT quotient provides a natural compactification of this moduli scheme,  and  parameterizes classes of $S$-equivalent semistable sheaves. As in the case of sheaves on a projective scheme the GIT quotient is a moduli scheme of semistable sheaves if and only if there are no strictly semistable sheaves. 

In the last section we make explicit the equivalence between our approach to stability for sheaves and stability of certain decorated sheaves. We first compare  stability of sheaves on a gerbe banded by a diagonalizable group scheme $G$ and stability of twisted sheaves according to  Yoshioka and Lieblich. Chosen a generating sheaf $\EE$, we prove that our moduli space of semistable sheaves is a disjoint union of moduli spaces of twisted sheaves according to Yoshioka, where the disjoint union runs over all the irreducible representations of $G$. It is a disjoint union of moduli spaces of twisted sheaves according to Lieblich for opportune choices of $\EE$. 
The second case  we study is the equivalence between our stability condition for torsion free sheaves on a root stack and parabolic stability in the sense of Maruyama and Yokogawa. In particular we prove that for a specific choice of the generating sheaf the two stability conditions are equivalent and the corresponding moduli spaces of semistable objects are isomorphic in a natural way.

\section*{Limits and future developments}

This article leaves some interesting questions open. The first and probably most obvious one is the dependence of the stability condition from the choice of a polarization.
In the case of schemes a change of polarization modifies the geometry of the moduli space of sheaves, in the case of stacks we expect both a change of $\OO_X(1)$ and a change of generating sheaf to produce modifications to the moduli space. For the moment we have not investigated this kind of problem but it is known that ``tensoring by a vector bundle'' affects Gieseker stability even in the case of schemes (every vector bundle on a scheme is a generating sheaf). See for instance [Thm 1.2]\cite{MR2019443} that is actually an immediate consequence of the general construction in \cite{MR1433203} or \cite{MR1355920}.

The second problem, that is quite complementary to the first one, is the connectedness of the $\Quot$-scheme with fixed modified Hilbert polynomial. The modified Hilbert polynomial is good enough to provide us with a $\Quot$-scheme that is  projective, however the projective scheme we obtain can be not connected in a very unnatural way. In the case of connected integral schemes, the Hilbert scheme of points of length $n$ turns out to be connected, that basically means that we can deform every point of length $n$ to any other of the same length. Beside, the Hilbert scheme of points does not depend on a choice of polarization. In the case of stacks there is not an evident way to fix numerical data for  the Hilbert scheme of points independently from the choice of a generating sheaf, moreover a Hilbert scheme with fixed modified Hilbert polynomial cannot be connected in general because orbifold points are ``more rigid'' than ordinary points. We can see this in a very simple integral example: consider the weighted projective stack $\XX=\mathbb{P}(3,3,2)$. We can try to classify points of $\XX$ whose coarse moduli space is a reduced point. We have only four cases: the generic non orbifold point $\mathbb{P}(1)$, the reduced point $\mathbb{P}(2)$ with stabilizer $\mu_2$, the reduced point $\mathbb{P}(3)$ with stabilizer $\mu_3$ and a double point $2\mathbb{P}(3)$ with stabilizer $\mu_3$ and non reduced structure orthogonal to the orbifold divisor. We can make a minimal choice for the generating sheaf $\EE = \OO_\XX\oplus \OO_\XX(1)\oplus\OO_\XX(2)$. Denote with $n$ the modified Hilbert polynomial $P_\EE(\OO_Z)=\chi(\XX ,\OO_Z\otimes\EE^\vee)$ where $\OO_Z$ is the structure sheaf of a point $Z$. Denote also with $n_i$ the integers $\chi(\XX,\OO_Z\otimes\OO_\XX(-i))$ so that $n=\sum_{i=0}^2 n_i$. We can produce the following table:
\begin{center}
\tablefirsthead{%
\hline
\multicolumn{1}{|l}{$Z$} &
\multicolumn{1}{|c}{$n_0$} &
\multicolumn{1}{|c}{$n_1$} &
\multicolumn{1}{|c}{$n_2$} &
\multicolumn{1}{|c|}{$n$} \\
\hline
}
\begin{supertabular}{|l|c|c|c|c|}
$\mathbb{P}(1)$ & 1 & 1 & 1 & 3 \\
$\mathbb{P}(2)$ & 1 & 0 & 1 & 2 \\
$\mathbb{P}(3)$ & 1 & 0 & 0 & 1\\
$2\mathbb{P}(3)$ & 1 & 0 & 1 & 2\\
\hline
\end{supertabular} 
\end{center}         
We can observe at least two undesired phenomena: first of all the generic point $\mathbb{P}(1)$ has the same modified Hilbert polynomial $n$ as $\mathbb{P}(2)\coprod\mathbb{P}(3)$, but the first correspond to a single point in the coarse moduli space and the second to a double point, so that they belong to two distinct connected components that $n$ cannot distinguish. We can easily overcome this problem considering the whole collection of numbers $n_i$ instead of $n$ alone. The second problem is that both $\mathbb{P}(2)$ and $2\mathbb{P}(3)$ have the same $n$ and the same collection of $n_i$'s but they cannot be deformed one into the other. We can solve this problem making a better a choice of the generating sheaf. Let $\EE=\OO_\XX\oplus\OO_\XX(2)\oplus\OO_\XX(4)\oplus\OO_\XX(3)$. We rewrite the table according to the new choice of $\EE$:
\begin{center}
\tablefirsthead{%
\hline
\multicolumn{1}{|l}{$Z$} &
\multicolumn{1}{|c}{$n_0$} &
\multicolumn{1}{|c}{$n_2$} &
\multicolumn{1}{|c}{$n_4$} &
\multicolumn{1}{|c}{$n_3$} &
\multicolumn{1}{|c|}{$n$} \\
\hline
}
\begin{supertabular}{|l|c|c|c|c|c|}
$\mathbb{P}(1)$ & 1 & 1 & 1 & 1 & 4 \\
$\mathbb{P}(2)$ & 1 & 1 & 1 & 0 & 3 \\
$\mathbb{P}(3)$ & 1 & 0 & 0 & 1 & 2 \\
$2\mathbb{P}(3)$ & 1 & 1 & 0 & 1 & 3 \\
\hline
\end{supertabular} 
\end{center}         
After this new choice the collection of the $n_i$'s is adequate to distinguish points that cannot be deformed into one another. We want also to remark that there is a natural bijective correspondence between the summands of $\EE$ and the connected components of the inertia stack of $\XX$, the correspondence being provided by the decomposition in eigensheaves of $\EE$ restricted to the inertia stack. 

Given $\XX$ a connected integral tame Deligne-Mumford stack that is a global quotient, we believe that it is always possible to associate to each connected component of the inertia stack a vector bundle $\EE_i$ such that the sum $\bigoplus_i \EE_i$ is a generating sheaf and the collection of integers $n_i=\chi(\XX,\OO_Z\otimes\EE_i^\vee)$ distinguishes each connected component of the Hilbert scheme of points with fixed modified Hilbert polynomial $n=\sum_i n_i$. The choice of the generating sheaf becomes unique, and the Hilbert scheme of points intrinsic once we have normalized the rank of each $\EE_i$. This kind of approach has evident applications to the study of Donaldson-Thomas invariants and Pandharipande-Thomas invariants and it will be the content of a work in preparation.

\section*{Acknowledgements}
First of all I want to express my gratitude to Barbara Fantechi who turned me from a physicist into mathematician. She introduced me to algebraic geometry, a completely mysterious discipline to me till the age of 25, and to the joy of rigorous thinking. Last but not least I am grateful for the beautiful problem she found for me and for all the assistance and encouragements during 4 years at SISSA;  I have to acknowledge her for the brilliant intuition that there should have been some analogy between moduli of twisted sheaves and moduli of parabolic bundles, that is actually the foundation of this article.
I am  grateful to Angelo Vistoli who assisted, encouraged  and technically helped me when this work was still at an embryonal stage, I have to acknowledge him for the idea that the use of a generating sheaf should have been the correct choice to study stability of sheaves on stacks and for Lemma \ref{lem:pure-submodules}.
I am also thankful to Scuola Normale Superiore (Pisa) for ospitality during May 2007.
I also say thanks to all the people in Trieste who spent some of their time talking to me about mathematics, among these deserve a special mention Etienne Mann who introduced me to toric stuff and for a whole year spent working together, and Elena Andreini for uncountable discussions during the last three years.



\section*{Conventions and notations}\label{sec:conv-notat}

Every scheme is assumed to be noetherian and also every tame stack (Def \ref{def:tame-stack}) is assumed noetherian if not differently stated. 
Unless differently stated every scheme, stack is defined over an algebraically closed field.
With $S$ we will denote a generic base scheme of finite type over the base field; occasionally it could be an algebraic space but in that case it will be explicitly stated. We will just say \textit{moduli space} for the \textit{coarse moduli space} of an algebraic stack and we will call it \textit{moduli scheme} if it is known to be a scheme. We will always denote with $\pi\colon \XX\to X$ the map from an $S$-stack $\XX$ to its moduli space $X$, with $p\colon \XX\to S$ the structure morphism of $\XX$. With the name \textit{orbifold} we will always mean a smooth Deligne-Mumford stack of finite type over a field and with generically trivial stabilizers. 

We will call a \textit{root stack} an orbifold whose only orbifold structure is along a simple normal crossing divisor. To be more specific let $X$ be a scheme over a field $k$ of characteristic zero. Let $\mathcal{D}=\sum_{i=1}^n\mathcal{D}_i$  be a simple normal crossing  divisor. Let $\mathbf{a}=(a_1,\ldots, a_n)$ a collection of positive integers . We associate to this collection of data a  stack: 
\begin{displaymath}
  \sqrt[\mathbf{a}]{D/X}\coloneqq \sqrt[a_1]{\mathcal{D}_1/X}\times_{X}\ldots \times_{X}\sqrt[a_n]{\mathcal{D}_n/X}
\end{displaymath}
that we will call a root stack. See \cite{Cstc-2007} and \cite{MR2450211} for a comprehensive treatment of the subject.

 A projective morphism of schemes $f:X\to Y$ will be projective in the sense of Grothendieck, that is $f$ is projective if there exists a coherent sheaf $E$ on $Y$ such that $f$ factorizes as a closed immersion of $X$ in $\mathbb{P}(E)$ followed by the structure morphism $\mathbb{P}(E)\to Y$.





\section{Cohomology and base change}
The natural generality to state a \textit{Cohomology and base change} result for algebraic stacks is provided by the concept of \textit{tame stack}. We recall the definition of tame stack from \cite{MR2427954}. Let $S$ be a scheme and $\XX\to S$ an algebraic stack locally of finite type over $S$. Assume that the stack has finite stabilizer, that is the natural morphism $I_{\XX}\to\XX$ is finite. Under this hypothesis $\XX$ has a moduli space $\pi:\XX\to X$ and the morphism $\pi$ is proper \cite{MR1432041}.

\begin{defn}\label{def:tame-stack}
Let $\XX$ be an algebraic stack with finite stabilizer as above and moduli space $\pi:\XX\to X$.  The stack $\XX$ is \textit{tame} if the functor $\pi_\ast:\text{QCoh}(\XX)\to\text{QCoh}(X)$ is exact where $\text{QCoh}$ is the category of quasicoherent sheaves. 
\end{defn}

We recall also the main result in \cite[Thm 3.2]{MR2427954}:
\begin{thm}\label{thm:tame-stacks-1}
  The following conditions are equivalent:  
  \begin{enumerate}
  \item $\XX$ is tame\label{item:4}.
  \item  For every $k$  algebraically closed field with a morphism $\Spec k\to S$ and every $\xi\in\XX(\Spec k)$ an object,  the stabilizer at the point $\xi$ (which is the group scheme $\AUT_k(\xi)\to\Spec k$) is linearly reductive.\label{item:5}
  \item There exists an $\text{fppf}$ cover $X'\to X$, a linearly reductive group scheme $G\to X'$ acting on a finite and finitely presented scheme $U\to X'$, together with an isomorphism
    \begin{displaymath}
      \XX\times_X X'\cong [U/G]
    \end{displaymath}
of algebraic stacks over $X'$. \label{item:6}
\item The same as the previous statement but $X'\to X$ is an \'etale cover. \label{item:7}
  \end{enumerate}
\end{thm}
For the definition of a \textit{linearly reductive} group scheme see in the same paper the second section and in particular definition $2.4$.  

We recall also the results in \cite[Cor 3.3]{MR2427954}
\begin{cor} \label{cor:tame-stack-2}
  Let $\XX$ be a tame stack over a scheme $S$ and let $\XX\to X$ be its moduli space:
  \begin{enumerate}
  \item If $X'\to X$ is a morphism of algebraic spaces, then $X'$ is the moduli space of $X'\times_X \XX$. \label{item:1}
  \item If $\XX$ is flat over $S$ then $X$ is flat over $S$. \label{item:2}
  \item Let $\FF\in\text{QCoh}(\XX)$ be a flat sheaf over $S$, then $\pi_\ast\FF$ is flat over $S$. \label{item:3}
  \end{enumerate}
\end{cor}

Actually the third point is not proven  in \cite[Cor 3.3]{MR2427954}. The result in the tame Deligne-Mumford case is an immediate consequence of \cite[Lem 2.3.4]{MR1862797}. It can be extended to tame stacks using Lemma \ref{lem:pure-submodules}.
\begin{rem} \label{item:18}
  For the convenience of the reader we recall also the following properties:
\begin{enumerate} 
\item the functor $\pi_\ast$ maps coherent sheaves to coherent sheaves. A proof can be found in \cite[Lem 2.3.4]{MR1862797}
\item the natural map $\OO_X\to\pi_\ast\OO_\XX$ is an isomorphism
\item since $\pi_\ast$ is an exact functor on $\text{QCoh}(\XX)$ and maps injective sheaves to flasque sheaves (Lem \ref{lem:cohomo-pushf}), we have that $H^\bullet(\XX,\FF)\cong H^\bullet(X,\pi_\ast\FF)$ for every quasicoherent sheaf $\FF$. 
\end{enumerate}
\end{rem}

In order to reproduce the \textit{Cohomology and base change} theorem as in \cite{Hag} or \cite{MR0282985} for an algebraic stack we need the following statement about tame  stacks:
\begin{prop}\label{prop:cohom-base-change-spcoarse}
  Let $\XX$ be a tame stack with moduli space $\pi:\XX\to X$ and $\rho:X'\to X$ a morphism of algebraic spaces. Consider the $2$-cartesian diagram:
  \begin{displaymath}
    \xymatrix@1{
      \XX\times_X X' \ar[d]^{\pi'}\ar[r]^-{\sigma} & \XX \ar[d]^{\pi} \\
      X'\ar[r]^{\rho} & X \\
    }
  \end{displaymath}
For every quasicoherent sheaf $\FF$ on $\XX$ the natural morphism $\rho^{\ast}\pi_{\ast}\FF\to\pi'_{\ast}\sigma^{\ast}\FF$ is an isomorphism.
\end{prop}

In order to prove this we first need the following lemma that has been suggested to us by Angelo Vistoli:
\begin{lem}\label{lem:pure-submodules}
  Let $P$ be a free $A$-module acted on by a linearly reductive group scheme $G\to \spec{A}$. The $G$-invariant part $P^G$ is a direct summand of $P$ and it is possible to choose a $G$-equivariant splitting of $P^G\to P$.
\end{lem}
\begin{proof}
  We start observing that the  natural morphism $P^G\otimes_A R\to (P\otimes_A R)^G$ is an isomorphism for every free $A$-module $R$ with a trivial coaction of $G$. With the same argument used in \cite[Cor 3.3]{MR2427954} we can deduce that it is an isomorphism for every finitely presented $A$-module $R$. As a consequence of this, the morphism $P^G\otimes_A R\to P\otimes_A R$ is always injective and in different words the submodule $P^G$ is pure. Using the theorem on pure submodules \cite[Thm 7.14]{MR1011461} we deduce that $P^G$ is a direct summand of $P$. Now we prove that there is an equivariant splitting. Using that $P^G$ is a free $A$-module and $G$ is linearly reductive we have the following exact sequence:
  \begin{displaymath}
   \xymatrix{
     0 \ar[r] & \Hom_A(P/P^G,P^G)^G \ar[r] & \Hom_A(P,P^G)^G \ar[r] & \Hom_A(P^G,P^G) \ar[r] & 0 \\
}
  \end{displaymath}
The identity homomorphism in $\Hom_A(P^G,P^G)$ lifts to an equivariant splitting in \\ $\Hom_A(P,P^G)^G=\Hom^G_A(P,P^G)$.
\end{proof}

\begin{proof}[Proof of proposition \ref{prop:cohom-base-change-spcoarse}.]
Since the problem is local in both $X$ and $X'$ we can assume that $X=\Spec A$ and $X'=\Spec A'$ are affine schemes and the base scheme $S$ is $X$. Applying theorem \ref{thm:tame-stacks-1}.\ref{item:6} we may assume that $\XX=[\Spec B/G]$ where $\rho\colon G\to \spec{A}$ is a linearly reductive group scheme acting on $\Spec B$, the map $\spec{B}\to\spec{A}$ is finite and of finite presentation and $A=B^G$. By the same theorem we obtain that the fibered product $\XX'\times_{X}X'$ is $[Spec{(B\otimes_A A')}/G]$ where the action of $G$ is induced by the action of $G$ on $B$ and $A'=(B\otimes_A A')^G$. In this setup a coherent sheaf $\FF$ is a finitely generated $B$-module $M$ which is equivariant for the groupoid $\xymatrix{G\times_{\spec{A}}\spec{B} \ar@<1ex>[r]^-{p} \ar@<0ex>[r]_-{a} & \spec{B}}$ where the two arrows $p,a$ are respectively the projection and the action. 
To prove the proposition is the same as proving that the natural morphism:
\begin{equation}\label{eq:affine-basechange}
  A'\otimes_A(\leftsub{A}{M})^G\xrightarrow{\psi_M} (M\otimes_A A')^G
\end{equation}
is an isomorphism.
The equivariant structure of the $B\otimes_A A'$-module $M\otimes_A A'$ is the obvious one; the $G$-invariant part of a module can be computed as follows: take the coaction $\leftsub{A}{M}\xrightarrow{\alpha} M\otimes_A \rho_\ast\OO_G$ and the trivial coaction $\leftsub{A}{M}\xrightarrow{\iota} M\otimes_A \rho_\ast\OO_G$ mapping $m\mapsto m\otimes 1$; the $G$-invariant part $\leftsub{A}{M}^G$ is the equalizer $\ker{\alpha-\iota}$. 
Since $B$ is finitely generated  as an $A$-module, the $A$-module $\leftsub{A}{M}$ is finitely generated  (the push forward of a coherent sheaf to the moduli space is coherent). Moreover $\leftsub{A}{M}$ admits a finite free presentation $P_2\to P_1\to \leftsub{A}{M}\to 0$. Since the tensor product $P_i\otimes \rho_\ast\OO_G$ is a flat resolution of $ M\otimes_A\rho_\ast\OO_G$ the  resolution $P_i$ inherits an equivariant structure from $\leftsub{A}{M}$. 

First we prove the statement for $P$ a projective $A$-module. 
To construct the natural map $\psi_P$ we start from the following exact diagram of $A$-modules:
\begin{displaymath}
  \xymatrix@1{
    0 \ar[r] & P^G \ar[r]\ar[d] & P \ar[rr]^-{\alpha-\iota}\ar[d] & & P\otimes_A\rho_\ast\OO_G \ar[d] \\
    0 \ar[r] &  (P\otimes_A A')^G \ar[r] & P\otimes_A A'\ar[rr]^-{(\alpha-\iota)\otimes\Id} & &  P\otimes_A\rho_\ast\OO_G\otimes_A A'
  }
\end{displaymath}
where the vertical map $P\to P\otimes_A A'$ is induced by $A\to A'$. We apply now the functor $\otimes_A A'$ and obtain th following diagram of $A'$-modules:
\begin{displaymath}
  \xymatrix@1{
     &  A'\otimes_A P^G \ar[r]\ar[d]^{\psi_P} & A'\otimes_A P \ar[r]\ar[d]^{\wr} & A'\otimes_A P\otimes_A\rho_\ast\OO_G \ar[d]^{\wr} \\
    0 \ar[r] &  (P\otimes_A A')^G \ar[r] & P\otimes_A A'\ar[r] & P\otimes_A\rho_\ast\OO_G\otimes_A A'
  }
\end{displaymath}
Since $G$ is linearly reductive and  $P$ is a free $A$-module $P^G$ is a direct summand of $P$ (as an equivariant module) according to Lemma \ref{lem:pure-submodules}, and the $A'$-module $(P\otimes_A A')^G$ is also a direct summand of $P\otimes_A A'$ of the same rank. Since the morphism $\psi_P$ is a surjective morphism  between two free $A'$-modules of the same rank it is an isomorphism.  


Since the formation of $\psi_M$ is functorial and the free resolution of $M$ is compatible with the coaction we obtain:
\begin{displaymath}
  \xymatrix@1{
    A'\otimes_A P_2^G \ar[r]\ar[d]_{\wr}^{\psi_2} & A'\otimes_A P_1^G \ar[r]\ar[d]_{\wr}^{\psi_1} & A'\otimes_A (\leftsub{A}{M})^G\ar[r]\ar[d]^{\psi_M} & 0\\
    (A'\otimes_A P_2)^G\ar[r] & (A'\otimes_A P_1)^G\ar[r] & (A'\otimes_A M)^G\ar[r] & 0\\
  }
\end{displaymath}
We have exactness on the right since $G$ is linearly reductive. Eventually $\psi_M$ is an isomorphism since the other two columns are isomorphisms. To extend the proof to quasicoherent sheaves we first observe that a quasi coherent sheaf is just a  $B$-module $N$ with a coaction. Quasi coherent sheaves on algebraic stacks are filtered limits of coherent sheaves, so we can assume that $N$ and the coaction are a filtered limit of coherent equivariant $B$-modules $M_\lambda$. We first observe that the tensor product commutes with filtered limits because it has a right adjoint. The functor $()^G$ commutes with filtered limits because it involves a tensor product and a kernel (which is a finite limit). The result follows now from the statement in the coherent case.  
\end{proof}
\begin{thm}[Cohomology and base change]\label{thm:cohom-base-change-stack}
 Let $p:\XX\to S$ be a tame  stack over $S$ with  moduli scheme $\pi:\XX\to X$ and such that $q:X\to S$ is proper. Let $\Spec{k(y)}\to S$ be a point. Let $\FF$ be a quasicoherent sheaf on $\XX$ flat over $S$. Then:
 \begin{enumerate}
 \item if the natural map
   \begin{displaymath}
     \phi^i(y):R^ip_{\ast}\FF\otimes k(y)\to H^i(\XX_y,\FF_y)
   \end{displaymath}
   is surjective, then it is an isomorphism, and the same is true for all $y'$ in a suitable neighborhood of $y$;
 \item Assume that $\phi^{i}(y)$ is surjective. Then the following conditions are equivalent:
   \begin{enumerate}
   \item $\phi^{i-1}(y)$ is also surjective;
   \item $R^ip_\ast\FF$ is locally free in a neighborhood of $y$.
   \end{enumerate}
 \end{enumerate}
\end{thm}
\begin{proof}
It follows from \ref{cor:tame-stack-2}.\ref{item:3} that $\pi_\ast\FF$ is flat over $S$ and according to \cite[Thm 12.11]{Hag}  the statement is true for the quasicoherent sheaf $\pi_\ast\FF$ and the natural map $\psi^i(y):R^iq_{\ast}(\pi_\ast(\FF))\otimes k(y)\to H^i(X_y,(\pi_\ast\FF)_y)$. Since $\pi_\ast$ is exact we have $R^iq_\ast\circ \pi_\ast\cong R^i(q_\ast\circ\pi_\ast)$. Applying   \ref{prop:cohom-base-change-spcoarse} we deduce that $(\pi_\ast\FF)_y$ is isomorphic to ${\pi_y}_\ast(\FF_y)$. According to \ref{cor:tame-stack-2}.\ref{item:1} the morphism $\pi_y:\XX_y\to X_y$ is the moduli scheme of $\XX_y$ so that $\pi_y$ is exact and we can conclude that $H^i(X_y,{\pi_y}_\ast(\FF_y))\cong H^i(\XX_y,\FF_y)$.
\end{proof}
Repeating exactly the same proof we can reproduce the \textit{Semicontinuity} theorem and a standard result of \textit{Flat base change}.
\begin{thm}[Semicontinuity]\label{thm:semicontinuity-stack}
  Let $p:\XX\to S$ be a tame  stack over $S$ with moduli scheme $\pi:\XX\to X$ and $q:X\to S$ is proper. Let $\FF$ be a quasicoherent sheaf on $\XX$ flat over $S$. Denote with $y$ a point of $S$. For every $i\geq 0$ the function $y\mapsto h^i(\XX_y,\FF_y)$ is upper semicontinuous on $S$.
\end{thm}
\begin{thm}\label{thm:flat-base-change}
  Let $p:\XX\to S$ be a separated tame  stack over $S$; let $u:S'\to S$ be a flat morphism and $\FF$ a quasicoherent sheaf on $\XX$.
  \begin{displaymath}
    \xymatrix@1{
      \XX\times_S S' \ar[d]^{p'}\ar[r]^{v} & \XX\ar[d]^p \\
      S' \ar[r]^u & S \\
    }
  \end{displaymath}
For all $i\geq 0$ the natural morphisms $u^*R^ip_\ast\FF\to R^ip'_\ast(v^\ast\FF)$ are isomorphisms.
\end{thm}

We conclude the section  with the following lemma, proving  that $\pi_\ast$ maps injectives to flasque sheaves (as anticipated in Remark \ref{item:18}). We guess it is well known to experts since many years, nevertheless we prefer to write a proof for lack of references. 
\begin{lem}\label{lem:cohomo-pushf}
  Assume $\pi\colon\XX\to X$ is a tame stack  and $\FF$ is an abelian sheaf on $\XX$. 
  If  $\mathcal{I}$ is an injective sheaf on $\XX$, the pushforward $\pi_\ast\mathcal{I}$ is \textit{flasque}\footnotemark \footnotetext{A sheaf on a site is flasque if it is acyclic on every object of the site (in agreement with Milne)} on $X$. 
\end{lem}
\begin{proof}
  We choose a smooth presentation  $X_0\to\XX$ and we associate to it  the simplicial nerve $X^\bullet$. Let $f^i\colon X_i\to X$ be the obvious composition. For every sheaf $\mathcal{I}$ on $\XX$ represented by $\mathcal{I}^\bullet$ on $X^\bullet$ we have a resolution (see \cite[Lem 2.5]{Osas05}):
\begin{equation}\label{eq:6}
  0\to\pi_\ast\mathcal{I}\to f^0_\ast\mathcal{I}_0\to f^1_\ast\mathcal{I}_1\to \ldots
\end{equation}
Assume now that $\mathcal{I}$ is injective, according to \cite[Cor 2.5]{Osas05} the sheaves $\mathcal{I}_i$ are injective for every $i$ so that $H^p(X_q,\mathcal{I}_q)$ is zero for every $p>0$ and every $q$ and $H^p(\XX,\mathcal{I})$ is zero for every $p>0$. 
Using \cite[Cor 2.7]{Osas05} and \cite[Th 4.7]{Osas05} we have a spectral sequence $E_1^{p,q}=H^p(X_q,\mathcal{I}_q)$ abutting to $H^{p+q}(\XX,\mathcal{I})$; for our previous observation this sequence reduces to the complex:
\begin{equation}\label{eq:7}
  H^0(X_0,\mathcal{I}_0)\to H^0(X_1,\mathcal{I}_1)\to H^0(X_2,\mathcal{I}_2)\to\ldots
\end{equation}
Now we observe that, being $\mathcal{I}_q$ injective, $R^pf^q_\ast\mathcal{I}_q=0$ for every $p>0$ \cite[III 1.14]{Met}. Using the Leray spectral sequence \cite[III 1.18]{Met} we have that $H^0(X_i,\mathcal{I}_q)=H^0(X,f^q_\ast\mathcal{I}_q)$. The resolution (\ref{eq:6}) is actually a flasque resolution of $\pi_\ast\mathcal{I}$ (apply Lemma \cite[III 1.19]{Met}) and applying the functor $\Gamma(X,\cdot)$ it becomes resolution (\ref{eq:7}). This proves that $H^i(X,\pi_\ast\mathcal{I})=H^i(\XX,\mathcal{I})$ and eventually zero for $i>0$. With the same argument (and actually applying Proposition \ref{prop:cohom-base-change-spcoarse}) we can prove that $\pi_\ast\mathcal{I}$ is acyclic on every open of the \'etale site of $X$ and conclude that $\pi_\ast\mathcal{I}$ is flasque using \cite[III 2.12.c]{Met}. 
\end{proof}
\begin{rem}[psychological]
  We don't know if $\pi_\ast$ maps injectives to injectives. If $\pi$ is flat (gerbes and root stacks) the answer is trivially yes, but in non flat cases we guess it could be false.
\end{rem}

\subsection{More results related to flatness}

In this part of the work we collect a few technical results taken from EGA which are related to flatness. We put these in a separate section since they are technically necessary but not so interesting in their own right. First of all we can reproduce for algebraic stacks a result of generic flatness \cite[6.9.1]{MR0199181}.
\begin{lem}[a result of hard algebra {\cite[6.9.2]{MR0199181}}]
  Let $A$ be a noetherian and integral ring and $B$ a finite type $A$-algebra; $M$ a finitely generated $B$-module. There is a principal open subscheme $\spec{A_f}$ such that $M_f$ is a free $A_f$-module. 
\end{lem}
\begin{prop}\label{prop:generic-flatness}
  Let $\XX\to S$ be a finite type noetherian algebraic stack. Let $\FF$ be a coherent $\OO_\XX$-module. There is  a finite stratification $\coprod S_i\to S$ where $S_i$ are locally closed in $S$  such that, denoted with $\XX_{S_i}$ the fibered product $\XX\times_S S_i$ the $\OO_{\XX_i}$-module $\FF\otimes_{\OO_S}\OO_{S_i}$ is flat on $\OO_{S_i}$.  
\end{prop}
\begin{proof}
As in the case of a finite type noetherian schemes, the proof is just an implementation of the previous lemma.
\end{proof}
\begin{rem}
The previous result is obviously weaker then a flattening-stratification result. In the case of a projective scheme it is possible to prove the existence of the flattening-stratification using generic-flatness with some cohomology and base change and some extra feature coming from the projective structure.  In \cite{MR2007396} Olsson and Starr proved a deeper result for stacks, that is the existence of \textit{the} flattening stratification\footnotemark \footnotetext{The flattening stratification is the one coming with a nice universal property.}; with no  assumption of  noetherianity they can  produce a surjective quasi-affine morphism to $S$ (which seems to be the optimal result in such a generality). They conjectured also that the flattening stratification is labeled by ``generalized'' Hilbert polynomial (as defined in the same paper). 
\end{rem}
We state  a stack theoretic version of \cite[6.9.9.2]{MR0163911} which is similar to \ref{thm:flat-base-change} but it can be used in the case of an arbitrary base change.
\begin{prop}\label{prop:arbitrary-base-change}
    Let $p:\XX\to S$ be a separated tame   stack over $S$ with $S$-projective moduli scheme $\pi:\XX\to X$; let $u:S'\to S$ be a  morphism of schemes and $\FF$ a coherent sheaf on $\XX$ which is flat on $\OO_S$ and such that $R^ip_\ast\FF$ are locally free for every $i\geq 0$, then for all $i\geq 0$ the natural morphisms $u^*R^ip_\ast\FF\to R^ip'_\ast(v^\ast\FF)$ are isomorphisms.
\end{prop}
\begin{proof}
  It follows from \cite[6.9.9.2]{MR0163911} applying proposition \ref{prop:cohom-base-change-spcoarse}. 
\end{proof}
We put here also a classical criterion about flatness of fibers which is theorem \cite[11.3.10]{MR0217086}. This will be used to fix a detail in the proof of the Kleiman criterion \ref{thm:kleiman-criterion}. This kind of result cannot be deduced from the analogous result for the moduli space, since flatness of $\XX$ on $S$ implies flatness of $X$ on $S$, but the contrary is not true. First we recall the statement in the affine case:
\begin{lem}[Lemme {\cite[11.3.10.1]{MR0217086}}]\label{lem:flatness-fibers-affine}
  Let $A\to B$ be a local homomorphism of noetherian local rings. Let $k$ be the residue field of $A$ and $M$ be a finitely generated non zero $B$-module. The following two conditions are equivalent:
  \begin{enumerate}
  \item $M$ is $A$ flat and $M\otimes_Ak$ is a flat $B\otimes_Ak$-module.
  \item $B$ is a flat $A$-module and $M$ is $B$-flat.
  \end{enumerate}
\end{lem}
\begin{prop}[flatness for fibers]\label{prop:flatness-for-fibers}
  Let $p\colon\XX\to S$ be a tame stack locally of finite type with moduli space $\pi\colon\XX\to X$. Let $\FF$ be a coherent $\OO_\XX$-module flat on $\OO_S$. Let $x$ be a point of $\XX$ and $s=p(x)$.The following statements are equivalent:
  \begin{enumerate}
  \item $\FF$ is flat at the point $x$ and the fiber $\FF_s$ is flat at $x$.
  \item The morphism $\pi$ is flat at the point $x$ and $\FF$ is flat at $x$.
  \end{enumerate}
If one of the two conditions is satisfied for a point $x$ then there is an open substack of $\XX$  such that for every point in it, the condition is satisfied.
\end{prop}
\begin{proof}
As in the case of schemes we can reduce the problem to the previous lemma.
\end{proof}

\section{The algebraic stack of coherent sheaves on a projective stack}
In this section we prove that $\mathfrak{Coh}_{\XX /S}$, the stack of coherent sheaves over an algebraic stack $\XX\to S$, is algebraic if the stack $\XX$ is tame   and satisfies some additional conditions. For every $S$-scheme  $U$, the objects in $\mathfrak{Coh}_{\XX /S}$ are all the coherent sheaves on $\XX_U=\XX\times_S U$ which are $\OO_U$-flat. Morphisms are isomorphisms of $\OO_{\XX_U}$-modules. We can also define a functor of flat quotients of a given coherent sheaf $\FF$, and we will denote it by $\Quot_{\XX /S}(\FF)$ in the usual way.
We have seen in the previous section that, if $\XX$ is tame, we have the same results of cohomology and base change and semicontinuity we have on schemes. To prove that $\mathfrak{Coh}_{\XX/S}$ is  algebraic we need some more structure. We need a polarization on the moduli scheme of $\XX$ and a \textit{very ample} sheaf on $\XX$. It is known that there are no very ample invertible sheaves on a stack unless it is an algebraic space, however it was proven in \cite{MR2007396}  that, under certain hypothesis, there exist locally free sheaves, called \textit{generating sheaves}, which behave like ``very ample sheaves''. Moreover in \cite{MR1844577} is introduced another class of locally free sheaves that could be interpreted as ``ample'' sheaves on stacks. Relations between these two classes of sheaves and the ordinary concept of ampleness are explained with some details in \cite{geomDM}. We briefly recall these notions. Let $\pi\colon\XX\to X$ be a Deligne-Mumford  $S$-stack with moduli space $X$:
\begin{defn} 
 A locally free sheaf $\mathcal{V}$ on $\XX$ is $\pi$-\textit{ample} if for every geometric point of $\XX$ the representation of the stabilizer group at that point on the fiber is faithful.
\end{defn}
\begin{defn}\label{def:generating-sheaf}
 A locally free sheaf $\mathcal{\EE}$ on $\XX$ is $\pi$-\textit{very ample} if for every geometric point of $\XX$ the representation of the stabilizer group at that point contains every irreducible representation.
\end{defn}
The following proposition is the reason why we have decided to use the word ``ample'' for the first class of sheaves.
\begin{prop}[{\cite[5.2]{geomDM}}]
  Let $\mathcal{V}$ be a  $\pi$-ample sheaf on $\XX$, there is a positive integer $r$ such that the locally free sheaf $\bigoplus_{i=0}^r\mathcal{V}^{\otimes i}$ is $\pi$-very ample.
\end{prop}
We recall here the notion of generating sheaf together with the existence result in \cite{MR2007396}.
Let $\XX$ be a tame Deligne-Mumford $S$-stack.
\begin{defn}\label{def:FE-GE}
 Let $\EE$ be a locally free sheaf on $\XX$. We define a functor $F_{\EE}:\mathfrak{QCoh}_{\XX /S}\to\mathfrak{QCoh}_{X/S}$ mapping $\FF\mapsto\pi_{\ast}\HOM_{\OO_\XX}(\EE,\FF)$ and a second functor $G_{\EE}:\mathfrak{QCoh}_{X/S}\to\mathfrak{QCoh}_{\XX /S}$ mapping $F\mapsto\pi^\ast F\otimes\EE$.
\end{defn}
\begin{rem}\label{rem:FE-exact}
\begin{enumerate}
\item  The functor $F_{\EE}$ is exact since the dual $\EE^\vee$ is locally free and the pushforward $\pi_\ast$ is exact. The functor $G_{\EE}$ is not exact unless the morphism $\pi$ is flat. This happens for instance if the stack is a flat gerbe over a scheme or in the case of root stacks.
\item (Warning) The notation $F_\EE$ is the same as in \cite{MR2007396} but $G_\EE$ is not. What they called $G_\EE$ there, is actually our $G_\EE\circ F_\EE$.
\end{enumerate}
\end{rem}
\begin{defn}
  A locally free sheaf $\EE$ is said to be a \textit{generator} for the quasi coherent sheaf $\FF$ if the adjunction morphism (left adjoint of the identity $\pi_\ast \FF\otimes\EE^\vee \xrightarrow{\Id} \pi_\ast \FF\otimes\EE^\vee $): 
  \begin{equation}\label{eq:4}
    \theta_{\EE}(\FF):\pi^\ast\pi_\ast\HOM_{\OO_\XX}(\EE,\FF)\otimes\EE\to \FF
  \end{equation}
is surjective.
It is a \textit{generating} sheaf of $\XX$ if it is a generator for every quasicoherent sheaf on $\XX$.
\end{defn}
\begin{prop}[{\cite[5.2]{MR2007396}}]
  A locally free sheaf on a tame Deligne-Mumford stack $\XX$ is a generating sheaf if and only if it is $\pi$-very ample.
\end{prop}
In the following we will use the word generating sheaf or $\pi$-very ample (or just very ample) sheaf interchangeably. The property expressed by (\ref{eq:4}) suggests that a generating sheaf should be considered as a very ample sheaf relatively to the morphism $\pi\colon\XX\to X$. Indeed the fundamental theorem of Serre \cite[Thm 2.2.1]{MR0217085}  says that: if $f\colon Y\to Z$ is a proper morphism and $\OO_Y(1)$ is a very ample invertible sheaf on $Y$ with respect to $f$, then there is a positive integer $n$ such that  the adjunction morphism $f^\ast f_\ast\HOM(\OO_Y(-n),\FF)\otimes\OO_Y(-n)\to \FF$ is surjective for every coherent sheaf $\FF$ on $Y$.

As we have defined  $\theta_{\EE}$ as the left adjoint of the identity we can define $\varphi_{\EE}$ the right adjoint of the identity. In order to do this we recall the following lemma from \cite{MR2007396}:
\begin{lem}\label{lem:coherent-projection}
  Let $\FF$ be a quasicoherent  $\OO_\XX$-module and $G$ a coherent $O_X$-module. A projection formula holds: 
  \begin{displaymath}
    \pi_\ast(\pi^\ast G\otimes\FF)=G\otimes\pi_\ast\FF
  \end{displaymath}
Moreover it is functorial in the sense that if $\alpha\colon \FF\to\FF'$ is a morphism of quasicoherent sheaves and $b\colon G\to G'$ is a morphism of coherent sheaves we have 
\begin{displaymath}
  \pi_\ast(\pi^\ast b \otimes \alpha)=b\otimes\pi_\ast\alpha
\end{displaymath}
\end{lem}
\begin{proof}
 We can prove the statement working locally. If we assume that  $G$ is coherent it has a finite free presentation and we conclude using exactness of $\pi_\ast$, right exactness of $\pi^\ast$ and $\otimes_{\OO_\XX}\FF$ and the projection formula in the free case. Functoriality follows with a similar argument. We can extend the result to quasicoherent sheaves with a standard limit argument.
\end{proof}

Let $F$ be a quasicoherent $\OO_X$-module:
\begin{displaymath}
\xymatrix{
  F \ar[rr]^-{\varphi_{\EE}(F)} && \pi_\ast\HOM_{\OO_\XX}(\EE,\pi^\ast F\otimes\EE)=F_{\EE}(G_{\EE}(F))
}
\end{displaymath}
According to lemma \ref{lem:coherent-projection} it can be rewritten as:
\begin{equation}\label{eq:2}
\xymatrix{
  F \ar[rr]^-{\varphi_{\EE}(F)}  && F\otimes\pi_\ast\END_{\OO_\XX}(\EE) }
\end{equation}
and it is the map given by tensoring a section with the identity endomorphism and in particular it is injective.

\begin{lem}\label{lem:idenity-theta-phi}
  Let $\FF$ be a quasicoherent sheaf on $\XX$. The following composition is the identity:
  \begin{displaymath}
\xymatrix{
    F_{\EE}(\FF) \ar[rr]^-{\varphi_{\EE}(F_{\EE}(\FF))} && F_{\EE}\circ G_{\EE}\circ F_{\EE}(\FF)\ar[rr]^-{F_{\EE}(\theta_{\EE}(\FF))} && F_{\EE}(\FF) }
  \end{displaymath}
Let $H$ be a coherent sheaf on $X$ then the following is the identity
 \begin{displaymath}
\xymatrix{
    G_{\EE}(H) \ar[rr]^-{G_{\EE}(\varphi_{\EE}(H))} && G_{\EE}\circ F_{\EE}\circ G_{\EE}(H)\ar[rr]^-{(\theta_{\EE}(G_\EE(H)))} && G_{\EE}(H) }
  \end{displaymath}
\end{lem}
\begin{proof}
This statement is precisely \cite[IV Thm 1]{MR1712872}. It's also easy to explicitly compute the composition because in the first statement  the second map is the composition of $\END_{\OO_\XX}(\EE)$ with $\HOM_{\OO_\XX}(\EE,\FF)$ while the first one is tensoring with the identity; in the second  the first map is tensoring with the identity endomorphism of $\EE$ while the second is $\Id_{\pi^\ast H}\otimes\theta_\EE(\EE)$.
\end{proof}

 As we have said before there are no very ample invertible sheaves on a stack which is not an algebraic space, however there can be ample invertible sheaves. 
 \begin{exmp}
 Let $\XX$ be a global quotient $[U/G]$ where $U$ is a scheme and $G$ a linearly reductive finite group. We have a natural morphism $\iota\colon [U/G]\to BG$. Let $V$ be the sheaf on $BG$ given by the left regular representation; the sheaf $\iota^\ast V$ is a generating sheaf of $\XX$.
 \end{exmp}
\begin{exmp}
A tame root stack $\XX\coloneqq\sqrt[r]{\mathcal{D}/X}$ over a scheme $X$ has an obvious ample invertible sheaf which is the tautological bundle $\OO_{\XX}(\mathcal{D}^{\frac{1}{r}})$ associated to the orbifold divisor. If the orbifold divisor has order $r$ the locally free sheaf $\bigoplus_{i=0}^{r-1} \OO_{\XX}(\mathcal{D}^{\frac{i}{r}})$ is obviously very ample and it has minimal rank. 
 \end{exmp}
 \begin{exmp}
A gerbe over a scheme banded by a cyclic group $\mu_r$ has an obvious class of ample locally free sheaves which are the twisted bundles, and there is an ample invertible sheaf if and only if the gerbe is essentially trivial (see \cite[Lem 2.3.4.2]{MR2309155}). As in the previous example if  $\mathcal{T}$ is a twisted locally free sheaf, $\bigoplus_{i=0}^{r-1} \mathcal{T}^{\otimes i}$ is very ample.
 \end{exmp}
 \begin{exmp}
  Let $\XX$ be a weighted projective space, the invertible sheaf $\OO_\XX(1)$ is ample, and denoted with $m$ the least common multiple of the weights, $\bigoplus_{i=1}^m\OO_\XX(i)$ is very ample. Usually it is not of minimal rank.
 \end{exmp}
 \begin{exmp}
   If $\XX$ is a toric orbifold with $\mathcal{D}_i,\,1\leq i\leq n$ the $T$-divisors associated to the coordinate hyperplanes, the locally free sheaf $\bigoplus_{i=1}^n\OO_\XX(\mathcal{D}_i)$ is  ample. Indeed if $\XX=[Z/G]$ where $Z$ is quasi affine in $\mathbb{A}^n$ and $G$ is a diagonalizable group scheme and the action of $G$ on $Z$ is given by  irreducible representations $\chi_i$ for $i=1,\ldots,n$ then the map 
   \begin{displaymath}
\xymatrix{
     1\ar[r] & G \ar[r]^-{\chi} & (\mathbb{C}^\ast)^n \\
}
   \end{displaymath}
is injective. To complete the argument we just notice that $\OO_\XX(\mathcal{D}_i)$ is the invertible sheaf given by the character $\chi_i$ and the structure sheaf of $Z$. 
\end{exmp}
With the following theorem  Olsson and Starr proved the existence of generating sheaves, and proved also that the notion of generating sheaf is stable for arbitrary base change on the moduli space.
\begin{defn}[{\cite[Def 2.9]{MR1844577}}]
  An $S$-stack $\XX$ is a global quotient if it is isomorphic to a stack $[Z/G]$ where $Z$ is an algebraic space of finite type over $S$ and $G\to S$ is a flat group scheme which is a subgroup scheme (a locally closed subscheme which is a subgroup) of $\GL_{N,S}$ for some integer $N$.
\end{defn}
\begin{thm}[{\cite[Thm. 5.7]{MR2007396}}]\label{thm:existence-generating-sheaf}
\begin{enumerate}
\item  Let   $\XX$  be a Deligne-Mumford tame  stack which is a separated global quotient over $S$, then there is a locally free sheaf $\EE$ over $\XX$ which is a generating sheaf for $\XX$. \label{item:15}
\item Let $\pi:\XX\to X$ be the moduli space of $\XX$ and $f:X'\to X$ a morphism of algebraic spaces. Moreover let $p:\XX':=\XX\times_X X'\to \XX$ be the natural projection from the fibered product, then $p^\ast\EE$ is a generating sheaf for $\XX'$.  \label{item:16}
\end{enumerate}
\end{thm}

In order to produce a smooth atlas of $\mathfrak{Coh}_{\XX /S}$ we need to study the representability of the $\Quot$ functor. Fortunately this kind of difficult study\footnotemark \footnotetext{We are not interested in quasicoherent sheaves so we don't state \cite[Thm 4.4]{MR2007396} in its full generality } can be found in \cite{MR2007396}.
\begin{thm}[{\cite[Thm. 4.4]{MR2007396}}]\label{thm:quot-closed}
  Let $S$ be a noetherian scheme of finite type over a field. Let $p\colon\XX\to S$ be a Deligne-Mumford tame stack  which is a separated global quotient and $\pi:\XX\to X$ the moduli space which is a scheme with a projective morphism  $\rho\colon X\to S$ with $p=\rho\circ\pi$. Suppose $\FF$ is a coherent sheaf on $\XX$ and $P$ a generalized Hilbert polynomial in the sense of Olsson and Starr \cite[Def 4.1]{MR2007396} then the functor $\Quot_{\XX /S}(\FF,P)$ is represented by a projective $S$-scheme. 
\end{thm}
The theorem we have stated here is slightly different from the theorem in the paper of Olsson and Starr. They have no noetherian assumption but they ask the scheme $S$ to be affine. Actually the proof doesn't change. 

First they prove this statement: 
\begin{prop}{\cite[Prop 6.2]{MR2007396}}
  Let $S$ be an algebraic space and $\XX$ a tame Deligne-Mumford stack over $S$ which is a separated global quotient. Let $\EE$ be a generating sheaf on $\XX$ and $P$ a generalized Hilbert polynomial: the natural transformation $F_{\EE}:\Quot_{\XX /S}(\FF,P)\to \Quot_{X/S}(F_{\EE}(\FF),P_V)$ is relatively representable by schemes and  a closed immersion (see the original paper for the definition of $P_V$; we are not going to use it). 
\end{prop}
We obtain theorem \ref{thm:quot-closed} from this proposition and using the classical result of Grothendieck about the representability of $\Quot_{X/S}(F_{\EE}(\FF))$ when $S$ is a noetherian scheme. 
\begin{rem}
  As in the case of schemes the functor $\Quot_{\XX/S}(\FF)$ is the disjoint union of projective schemes $\Quot_{\XX/S}(\FF,P)$ where $P$ ranges through all generalized Hilbert polynomial.
\end{rem}
Assume now that $\XX$ is defined over a field; it is known that $\XX$ has a generating sheaf and projective moduli scheme if and only if $\XX $  is a global quotient and has a projective moduli scheme. In characteristic zero this is also equivalent to the stack $\XX$ to be a closed embedding in a smooth proper Deligne-Mumford stack with projective moduli scheme \cite[Thm 5.3]{geomDM}; in general this third property implies the first twos. This motivates the definition of projective stack:
\begin{defn}
  Let $k$ be a field. We will say $\XX \to \Spec{k}$ is a \textit{projective stack}  (quasi projective) over $k$ if it is a  tame separated global quotient with moduli space which is a projective scheme (quasi projective). 
\end{defn}
For the reader convenience we summarize here  equivalent definitions in characteristic zero:
\begin{thm}[{\cite[Thm 5.3]{geomDM}}]
  Let $\XX\to\spec{k}$ be a Deligne-Mumford stack over a field $k$ of characteristic zero. The following are equivalent:
  \begin{enumerate}
  \item the stack $\XX$ is projective (quasi projective)
  \item the stack $\XX$ has a projective (quasi projective) moduli scheme and there exists a generating sheaf
  \item the stack $\XX$ has a closed embedding (locally closed) in a smooth Deligne-Mumford stack over $k$ which is proper over $k$ and has projective moduli scheme.
  \end{enumerate}

\end{thm}
\begin{rem}
  It could seem more natural to define a projective stack via the third statement in the previous theorem since it has an immediate geometric meaning; however we prefer to use a definition that is well behaved in families, also in mixed  characteristic.
\end{rem}
We give a relative version of the definition of projective stack. We first observe that if $\XX=[Z/G]$ is a global quotient over a scheme $S$, for every geometric point $s$ of $S$ the fiber $\XX_s$ is the global quotient $[Z_s/G_s]$ where $Z_s$ and $G_s$ are the fibers of $Z$ and $G$. Moreover if $X\to S$ is a projective morphism the fibers $X_s$ are projective schemes and according to Corollary \ref{cor:tame-stack-2}~\ref{item:1} they are the moduli schemes of $\XX_s$. This consideration leads us to the definition:
\begin{defn}\label{def:projective-family}
  Let $p\colon\XX\to S$ be a  tame stack on $S$ which is a separated global quotient with moduli scheme $X$ such that $p$ factorizes as $\pi\colon\XX\to X$ followed by $\rho\colon X\to S$ which is a projective morphism. We will call $p\colon\XX \to S$ a \textit{family of projective stacks}. 
\end{defn}
\begin{rem}
  \begin{enumerate}
  \item we don't say it is a \textit{projective morphism} from $\XX$ to $S$ since this is already defined and means something else. There are no projective morphisms in the sense of \cite[14.3.4]{LMBca} from $\XX$ to $S$, indeed such a morphism cannot be representable unless $\XX$ is a scheme.   
  \item Each fiber over a geometric point of $S$ is actually a projective stack, which motivates the definition.
  \item A family of projective stacks $\XX\to S$ has a generating sheaf $\EE$ according to Theorem \ref{thm:existence-generating-sheaf}, and according to the same theorem the fibers of $\EE$ over geometric points of $S$ are generating sheaves for the fibers of $\XX$.
  \end{enumerate}
\end{rem}

\subsection{A smooth atlas for the stack of coherent sheaves}

  Let $\pi :\XX\to X$ be a family of projective Deligne-Mumford stacks. Choose a polarization $\OO_X(1)$ and a generating sheaf $\EE$ on $\XX$. Consider the disjoint union of projective schemes $Q_{N,m}:=\Quot_{\XX /S}(\EE^{\oplus N}\otimes\pi^\ast\OO_X(-m))$ where $N$ is a non negative integer and $m$ is an integer and let $\EE_Q^{\oplus N}\otimes\pi_Q^\ast\OO_{X_Q}(-m))\xrightarrow{u_{N,m}}\mathcal{U}_{N,m}$ be the universal quotient sheaf. We can define the morphism:
  \begin{displaymath}
    \xymatrix@1{
    \mathcal{U}_{N,m} \colon Q_{N,m}\; \ar[rr] & &  \;\mathfrak{Coh}_{\XX /S}
  } 
 \end{displaymath}
Denote with $\rho\colon X\to S$, with $p$ the composition $\rho\circ\pi$ and with $\pi_U,\rho_U,p_U$ every map obtained by base change from a scheme $U$ with a map to $S$. We assume that for every base change $U\to S$ it is satisfied ${\rho_U}_\ast\OO_{X_U}=\OO_U$ so that we have also ${\rho_U}_\ast \rho_U^\ast =\Id $.
   We define an open subscheme $Q_{N,m}^0\to Q_{N,m}$.
Let $U$ be an $S$ scheme with a map to $Q_{N,m}$ given by a quotient $\EE_U^{\oplus N}\otimes \pi_U^\ast \OO_{X_U}(-m)\xrightarrow{\mu} \mathcal{M}$. In order for the map to factor through $Q^0_{N,m}$ it must satisfy the following conditions:
\begin{enumerate}
\item The higher derived functors $R^i{\rho_U}_\ast(F_{\EE_U}(\mathcal{M})(m))$ vanish for every positive $i$, and for $i=0$ it is a free sheaf. This condition is open because of proposition \ref{thm:cohom-base-change-stack}.
\item The $\OO_U$-module ${\rho_U}_\ast(F_{\EE_U}(\mathcal{M})(m))$ is locally free and has constant rank $N$. This is an open condition because of \ref{thm:semicontinuity-stack}.
\item Consider the morphism:
  \begin{displaymath}
    E_{N,m}(\mu)\colon \OO_{X_U}^{\oplus N}\xrightarrow{\varphi_{\EE}(\OO_{X_U}^{\oplus N})} F_{\EE_U}\circ G_{\EE_U}(\OO_{X_U}^{\oplus N}) \xrightarrow{F_{\EE}(\mu)} F_{\EE_U}(\mathcal{M})(m)
  \end{displaymath}
The pushforward ${\rho_U}_\ast E_{N,m}$ is a morphism of locally free $\OO_{U}$-modules of the same rank because of the previous point. We ask this map to be an isomorphism which is an open condition since it is a map of locally free modules.  
\end{enumerate}

\begin{prop}\label{prop:alg-stack-coherent}

The following composite morphism:
  
\begin{displaymath}
    \xymatrix{
       Q_{N,m}^0\subseteq Q_{N,m}\; \ar[rr]^-{\mathcal{U}_{N,m}} & & \;\mathfrak{Coh}_{\XX /S} 
    }
  \end{displaymath}

denoted with $\mathcal{U}^0_{N,m}$ is representable locally of finite type and smooth for every couple of integers $m,N$. 
\end{prop}

\begin{proof}
This proof follows the analogous one for schemes in \cite[Thm 4.6.2.1]{LMBca} with quite a  number of necessary modifications.

Let $V$ be an $S$-scheme with a map $\mathcal{N}$ to $\mathfrak{Coh}_{\XX /S}$. In order to study the representability and smoothness of $\mathcal{U}^0_{N,m}$ we compute the fibered product $Q^0_{N,m}\!\!\!\!\!\!\!\!\!\!\underset{\phantom{AAA}\mathfrak{Coh}_{\XX /S},\mathcal{N}}{\times}\!\!\!\!\!\!\! V$. Denote with $QV$ the fibered product $Q^0_{N,m}\times_S V$, with $\sigma_Q,\sigma_V$ its two projections and with $\tau_Q\colon X_{QV}\to X_{Q^0_{N,m}},$ $\tau_V\colon X_{QV}\to X_V$ the two projections induced by base change and with $\eta_Q$ and $\eta_V$ the two analogous projections from $\XX_{QV}$. It follows almost from the definition that the fibered product  is given by:
\begin{displaymath}
  \ISO_{\XX_{QV}}\bigl(\eta_Q^\ast \mathcal{U}^0_{N,m},\eta_V^\ast \mathcal{N}\bigr)\xrightarrow{p_1} V
\end{displaymath}
As in \cite{LMBca} we observe that there is a maximal open subscheme $V_{N,m}\subseteq V$ such that the following conditions are satisfied (here and in the following we write $V$ instead of $V_{N,m}$ since it is open in $V$ and in particular smooth):
\begin{enumerate}
\item The higher derived functors $R^i{\rho_{V}}_\ast(F_{\EE_V}(\mathcal{N})(m))$ vanish for all $i>0$.
\item The coherent sheaf ${\rho_V}_\ast(F_{\EE_V}(\mathcal{N})(m))$ is locally free of rank $N$ (not free as we have assumed before).
\item The following adjunction morphism is surjective: 
\begin{displaymath}
\rho_V^\ast {\rho_V}_\ast F_{\EE_{V}}(\mathcal{N})(m) \xrightarrow{\psi_V} F_{\EE_{V}}(\mathcal{N})(m)
\end{displaymath}
\end{enumerate}
The last condition is a consequence of Serre's fundamental theorem about projective morphisms \cite[2.2.1]{MR0217085} applied to the moduli scheme. 
Keeping in mind the conditions we have written we can define a natural transformation:
\begin{displaymath}
  \ISO_{\XX_{QV}}\bigl(\eta_Q^\ast \mathcal{U}^0_{N,m},\eta_V^\ast \mathcal{N}\bigr)\xrightarrow{I_{N,m}} \ISO_{V}(\OO_V^{\oplus N},{\rho_V}_\ast(F_{\EE_V}(\mathcal{N})(m)))
\end{displaymath}
factorizing the projection $p_1$ to $V$.
An object of the first functor over a scheme $W$ is a morphism $f\colon W\to V$, a morphism $g\colon W\to Q^0_{N,m}$ and an isomorphism:
\begin{displaymath}
  \alpha\colon g^\ast\mathcal{U}^0_{N,m}\longrightarrow f^\ast\mathcal{N}
\end{displaymath}
The transformation $I_{N,m}(W)$  associates to these data the morphism $f$ and the isomorphism:
\begin{displaymath}
  {\rho_W}_\ast (F_{\EE_W}(\alpha)\circ (g^\ast E_{N,m}(u^0_{N,m})))\colon \OO_W^{\oplus N}\to {\rho_W}_\ast (F_{\EE_W}(f^\ast\mathcal{N})(m)) 
\end{displaymath}
This transformation is \textit{relatively representable}\footnotemark \footnotetext{This notion appears in some notes of Grothendieck, it just means that for every natural transformation from a scheme to the second functor, the fibered product is a scheme.}, moreover we can prove that it is an isomorphism of sets for every $f\colon W\to V$. To do this we construct an explicit inverse of $I_{N,m}$, call it $L_{N,m}$. The map $L_{N,m}$ is defined in this way: to an isomorphism $\beta:\OO_W^{\oplus N}\to {\rho_W}_\ast (F_{\EE_W}(f^\ast\mathcal{N})(m))$ associate  the following surjective map:
\begin{displaymath}
  \tilde{g}\coloneqq\theta_{\EE}(f^\ast\mathcal{N})\circ G_{\EE_W}(\psi_W\circ\rho_W^\ast\beta)\colon \EE_W^{\oplus N} \longrightarrow f^\ast(\mathcal{N}\otimes\pi_V^\ast\OO_{X_V}(m))
\end{displaymath}
To give an object in $\ISO_{\XX_{QV}}$ we need to verify that this quotient is a map to $Q^0_{N,m}$: we have to check that ${\rho_W}_\ast E_{N,m}(\tilde{g})$ is an isomorphism. To achieve this we analyze the morphism with the following diagram:
\begin{equation}\label{eq:10}
  \xymatrix{
  &&& \OO_{X_W}^{\oplus N}  \ar[ddddlll]_-{\psi_W\circ\rho_W^\ast\beta} \ar[dd]^-{\varphi_{\EE_W}(\OO_{X_W}^{\oplus N})}\\
  &&&  \\
  &&& F_{\EE_W}\circ G_{\EE_W}(\OO_{X_W}^{\oplus N})\ar[dd]^-{F_{\EE}\circ G_{\EE}(\psi_W\circ\rho_W^\ast\beta)} \\
  &&& \\
  F_{\EE_W}(f^\ast(\mathcal{N}\otimes\pi_V^\ast\OO_{X_V}(m)))\ar[rrr]_-{\varphi_{\EE_W}} \ar@{=}[ddrrr] &&&   F_{\EE_W}\circ G_{\EE_W}\circ F_{\EE_W}(f^\ast(\mathcal{N}\otimes\pi_V^\ast\OO_{X_V}(m)))\ar[dd]^-{F_{\EE_W}(\theta_{\EE_W}(f^\ast(\mathcal{N}\otimes\pi_V^\ast\OO_{X_V}(m))))} \\
  &&& \\
  &&& F_{\EE_W}(f^\ast(\mathcal{N}\otimes\pi_V^\ast\OO_{X_V}(m))) \\
}
\end{equation}
where the upper triangle is commutative because $\varphi_{\EE}$ is a natural transformation, the lower triangle is commutative according to Lemma \ref{lem:idenity-theta-phi}. We can conclude that $E_{N,m}(\tilde{g})=\psi_W\circ \rho_W^\ast\beta$; then we have to apply ${\rho_W}_\ast$ and we obtain exactly $\beta$ (recall that ${\rho_W}_\ast\rho_W^\ast=\Id $).
It is now immediate to verify that  $\tilde{g}^\ast\mathcal{U}^0_{N,m}$ is isomorphic to $f^\ast\mathcal{N}$, to explicitly obtain the isomorphism we must compare the universal quotient $\tilde{g}^\ast u^0_{N,m}$ and $\theta_{\EE}(f^\ast\mathcal{N})\circ G_{\EE_W}(\psi_W\circ\rho_W^\ast\beta)$. The identity $I_{N,m}(W)\circ L_{N,m}(W)(\beta)=\beta$ is implicit in the construction. 
To prove that $L_{N,m}(W)\circ I_{N,m}(W)(\alpha)=\alpha$ we use the following diagram:
\begin{displaymath}
  \xymatrix{
 &&& G_{\EE_W}(\OO_{X_W}^{\oplus N})  \ar@{=}[ddlll] \ar[dd]^-{G_{\EE_W}(\varphi_{\EE_W}(\OO_{X_W}^{\oplus N}))}\\
  &&&  \\
  G_{\EE_W}(\OO_{X_W}^{\oplus N}) \ar[dd]^-{g^\ast u^0_{N,m}} &&& G_{\EE_W}\circ F_{\EE_W}\circ G_{\EE_W}(\OO_{X_W}^{\oplus N}) \ar[lll]^-{\theta_{\EE_W}} \ar[dd]^-{G_{\EE_W}\circ F_{\EE_W}(g^\ast u^0_{N,m})} \\
  &&& \\
 g^\ast(\mathcal{U}^0_{N,m}\otimes\pi_Q^\ast\OO_{X_{Q}}(m)) \ar[dd]^-{\alpha} &&&   G_{\EE_W}\circ F_{\EE_W}(g^\ast(\mathcal{U}^0_{N,m}\otimes\pi_Q^\ast\OO_{X_{Q}}(m))) \ar[lll]^-{\theta_{\EE_W}} \ar[dd]_-{G_{\EE_W}\circ F_{\EE_W}(\alpha)} \\
  &&& \\
 f^\ast(\mathcal{N}\otimes\pi_V^\ast\OO_{X_V}(m)) &&& G_{\EE_W}\circ F_{\EE_W}(f^\ast(\mathcal{N}\otimes\pi_V^\ast\OO_{X_V}(m))) \ar[lll]^-{\theta_{\EE_W}} \\
}
\end{displaymath}
where the first triangle is commutative because of Lemma \ref{lem:idenity-theta-phi} and the two squares are commutative because $\theta_{\EE}$ is natural.

Since the functor $\ISO_{V}(\OO_V^{\oplus N},{\rho_V}_\ast F_{\EE_V}(\mathcal{N})(m))$ is represented by a scheme of finite type according to \cite[Thm. 4.6.2.1]{LMBca} and $I_{N,m}$ is a relatively representable isomorphism we deduce that the functor $ \ISO_{\XX_{QV}}\bigl(\eta_Q^\ast \mathcal{U}^0_{N,m},\eta_V^\ast \mathcal{N}\bigr)$ is represented by a scheme of finite type and it is a $\GL(N,\OO_{V})$-torsor over $V$ so that it is represented by a scheme\footnotemark \footnotetext{It is an application of \cite[7.7.8-9]{MR0163911} as explained in detail in the proof of \cite[Thm 4.6.2.1]{LMBca}} smooth over $V$.
\end{proof}
\begin{prop}
  The morphism:
\begin{displaymath}
    \xymatrix{
      \underset{N,m}{\coprod} Q_{N,m}^0\subseteq\underset{N,m}{\coprod} Q_{N,m}\; \ar[rr]^-{\coprod \mathcal{U}^0_{N,m}} & & \;\mathfrak{Coh}_{\XX /S} 
    }
  \end{displaymath}
is surjective.
\end{prop}
\begin{proof}
 To prove surjectivity of the map $\coprod \mathcal{U}^0_{N,m}$ we observe that given an $S$-scheme $U$ and an object $\mathcal{N}\in\mathfrak{Coh}_{\XX /S}(U)$, we can construct the coherent $\OO_{X_U}$-module $F_{\EE_U}(\mathcal{N})$, and according to Serre \cite[Thm 2.2.1]{MR0217085} there is $m$ big enough such that the adjunction morphism is surjective:
\begin{displaymath}
  H^0(F_{\EE_U}(\mathcal{N})(m))\otimes \OO_{X_U}(-m)\longrightarrow F_{\EE_U}(\mathcal{N})
\end{displaymath}
Now we apply the functor $G_{\EE_U}$ and the adjunction morphism $\theta_{\EE_u}$ and we obtain the surjection:
\begin{displaymath}
  H^0(F_{\EE_U}(\mathcal{N})(m))\otimes \EE_{U}\otimes \pi_U^\ast \OO_{X_U}(-m)\longrightarrow \mathcal{N}
\end{displaymath}
We will denote this composition with $\widetilde{\ev}(\mathcal{N},m)$.
Now let $N$ be the dimension of $H^0(F_{\EE_U}(\mathcal{N})(m))$; the point $\mathcal{N}$ in the stack of coherent sheaves is represented on the chart $Q^0_{N,m}$. To prove it is in $Q_{N,m}$ we use the fundamental theorem of projective morphisms of Serre; to prove it is in $Q^0_{N,m}$ we use the same argument depicted in diagram (\ref{eq:10}). 
\end{proof}
\begin{cor}\label{cor:alg-stack-of-coh-sh}
The stack $\mathfrak{Coh}_{\XX/S}$ is an Artin stack locally of finite type with atlas $\underset{N,m}{\coprod} Q_{N,m}^0$.
\end{cor}

\section{Gieseker stability}

In the first part we have developed the setup we use here to define a good notion of stability for coherent sheaves. We define a concept of  \textit{Gieseker stability} that relies on  a modified Hilbert polynomial.
\begin{ass}
In this section $p\colon\XX\to \spec{k}$ is a projective Deligne-Mumford stack over an algebraically closed field $k$ with moduli scheme $\pi\colon\XX\to X$; a very ample invertible sheaf $\OO_X(1)$ and a generating sheaf $\EE$ are chosen. We will call this couple of sheaves a \textit{polarization} of $\XX$
\end{ass}

\subsection{Pure sheaves}
As in the case of sheaves on schemes we can define the support of a sheaf in the following way \cite[1.1.1]{MR1450870}
\begin{defn}
  Let $\FF$ be a coherent sheaf on $\XX$, we define the support of $\FF$ to be the closed substack associated to the ideal:
  \begin{displaymath}
    0\to\mathcal{I}\to\OO_\XX\to\END_{\OO_\XX}(\FF)
  \end{displaymath}
\end{defn}
Since the stack $\XX$ has finite stabilizers we can deal with the dimension of the support of a sheaf as we do with schemes and define \cite[1.1.2]{MR1450870}:
\begin{defn}
  A pure sheaf of dimension $d$ is a coherent sheaf $\FF$ such that for every non zero subsheaf $\GG$ the support of $\GG$ is of pure dimension $d$. 
\end{defn}
\begin{rem}\label{rem:pure-sheaves-supp}
\begin{enumerate}
\item  Let $\xymatrix{X_1\ar@<1ex>[r]^-s \ar@<-1ex>[r]^-t & X_0 \ar[r]^-{\phi} & \XX}$ be an \'etale presentation of $\XX$ and $\FF$ a pure sheaf on $\XX$ of dimension $d$. If the ideal sheaf $\mathcal{I}$ defines the support of $\FF$, the ideal $\phi^\ast\mathcal{I}$ defines the support of $\phi^\ast\FF$ in $X_0$. This follows from the flatness of $\phi$. For the same reason we can produce an \'etale presentation of $\supp{(\FF)}$ as $\xymatrix{\supp{(t^\ast\phi^\ast\FF)}\ar@<1ex>[r]^s\ar@<-1ex>[r]^t & \supp{(\phi^\ast\FF)}}$. It is then clear that $\dim{(\FF)}=2\dim{(\phi^\ast\FF)}-\dim{(t^\ast\phi^\ast\FF)}=\dim{\phi^\ast\FF}$ in every point of $\supp{(\FF)}$.  
\item \label{item:17} The notion of associated point in the case of a Deligne-Mumford stack is the usual one \cite[2.2.6.5]{MR2309155}. A geometric point $x$ of $\XX$ is associated for the coherent sheaf $\FF$ on $\XX$ if $x$ is associated for the stalk $\FF_x$ as on $\OO_{\XX,x}$-module. A sheaf is pure if and only if every associated prime has the same dimension (\ie the support has pure dimension and there are no embedded primes in the sheaf). Moreover $\phi(\Ass{(\phi^\ast\FF)})=\Ass{(\FF)}$ \cite[2.2.6.6]{MR2309155}. It is now clear that $\FF$ is pure of dimension $d$ if and only if $\phi^\ast\FF$ is pure of the same dimension.  
\end{enumerate}
\end{rem}
As in \cite[1.1.4]{MR1450870} we have the torsion filtration:
\begin{displaymath}
  0\subset T_0(\FF)\subset\ldots\subset T_d(\FF)=\FF
\end{displaymath}
where every factor $T_i(\FF)/T_{i-1}(\FF)$ is pure of dimension $i$ or zero.

In the case $\XX$ is an Artin stack  we don't know what is the meaning of the dimension of the support, but we can use the notion of associated point as defined in \cite[2.2.6.4]{MR2309155} and of torsion subsheaf \cite[2.2.6.10]{MR2309155}. Lieblich proves that the sum of torsion subsheaves is torsion so that there is a maximal torsion subsheaf \cite[2.2.6.11]{MR2309155}. We will denote it with $T(\FF)$. Maybe there is a more general notion of a torsion filtration for sheaves on Artin stacks but for the moment we prefer to limit the study to the case of torsion free sheaves.

In the following we  prove that the functor $F_{\EE}$ preserves the pureness the support and the dimension of a sheaf. This will be of great help to prove Corollary \ref{cor:boundedness-eventually}.
The proof goes in two parts. First we observe that the morphism $\pi$ is an omeomorphism so that it preserves the dimension of points and we prove that $\pi \Ass{\FF}=\Ass(\pi_\ast\FF)$ for every coherent sheaf $\FF$, unless the push-forward $\pi_\ast\FF$ vanishes. Second, we prove that $F_{\EE}(\FF)$ is non zero unless $\FF$ itself is zero. 
To clarify the situation we show the following example. Let $\pi\colon\XX\to X$ be an abelian $G$-banded gerbe over a scheme. Every sheaf $\FF$ on $\XX$ decomposes into a direct sum on the characters of the banding group $\FF=\bigoplus_{\chi\in C(G)}\FF_\chi$. The push-forward $\pi_\ast\FF$ is just $\pi_\ast\FF_0$ where $\FF_0$ corresponds to the trivial character. 
This example explains that a sheaf supported on a gerbe can be sent to zero, however we shall prove that this is the only pathology that occurs, and tensoring with the generating sheaf removes this pathology.

\begin{lem}\label{lem:pure-sheaves-supp}
 Let $\XX\to\spec{k}$ be a  projective Deligne-Mumford stack and $\pi\colon\XX\to X$ the moduli scheme. Let $\FF$ be a coherent sheaf on $\XX$, then we have:
 \begin{equation}
   \label{eq:supp-fe}
   \pi\supp{\FF}=\pi \supp{\FF\otimes\EE^\vee}\supseteq\supp{F_{\EE}(\FF)}
 \end{equation}
moreover $F_{\EE}(\FF)$ is the zero sheaf if and only if $\FF=0$.
\end{lem}
\begin{proof}
 Since the sheaf $\EE$ is locally free and supported everywhere we have the first claimed equality (the tensor product intersects the supports). 
This is made evident by the following diagram:
\begin{displaymath}
  \xymatrix{
    0 \ar[r] & \mathcal{I} \ar[r]\ar[d]^-{\wr} & \OO_\XX \ar[r]\ar@{=}[d] & \END_{\OO_\XX}(\FF)\ar@{>->}[d]^-{\otimes\Id_{\EE}} \\
    0 \ar[r] & \mathcal{I}_\EE \ar[r] & \OO_\XX \ar[r] & \END_{\OO_\XX}(\FF\otimes\EE^\vee) \\
  }
\end{displaymath}
where $\mathcal{I}$ defines $\supp{(\FF)}$ and $\mathcal{I}_\EE$ the support of $\FF\otimes\EE^\vee$.
With the following commutative and exact diagram we verify the second inclusion:
\begin{displaymath}
   \xymatrix{
     0 \ar[r] & \pi_\ast\mathcal{I}_\EE \ar[r]\ar@{>->}[d] & \OO_X \ar[r]\ar@{=}[d] &  \pi_\ast\END_{\XX}(\FF\otimes\EE^\vee)\ar[d] \\
  0 \ar[r]  & \mathcal{I}_{F_\EE(\FF)}\ar[r] & \OO_X \ar[r] & \END_{X}(F_\EE(\FF)) \\ 
}
 \end{displaymath}
The vertical arrow on the left is easily verified to be injective. 

The non vanishing of $F_{\EE}(\FF)$ can be checked on points, in particular we can assume that $\XX$ is irreducible and prove that the generic fiber of $F_{\EE}(\FF)$ doesn't vanish. Let $\eta$ be an opportune field extension of the generic point of $\XX$. We can represent the residual gerbe at the generic point with this diagram:
\begin{displaymath}
  \xymatrix{
B\stab{(\eta)} \ar[r]^-{h} \ar[d]_-{\rho} &  \mathcal{G}_\eta \ar[d] \ar[r]^-{f} & \XX \ar[d]^-{\pi} \\ 
 \spec{\eta} \ar[r]^-{i} & \spec{\overline{\eta}} \ar[r]^-{g} & X \\
}
\end{displaymath}
where $\mathcal{G}_\eta,\spec{\overline{\eta}}$ are the residual gerbe and its residue field, $f,g$ are open morphisms, $h,i$ are \'etale (the residual gerbe is an \'etale gerbe) and the two squares are cartesian (the one on the right wouldn't be cartesian for a non open point). 
Using base change on the moduli space we have $i^\ast g^\ast\pi_\ast(\FF\otimes\EE^\vee)=\rho_\ast h^\ast f^\ast(\FF\otimes\EE^\vee)$. Since  a linearly reductive group scheme on a field is completely reducible we can decompose $h^\ast f^\ast(\FF\otimes\EE^\vee)$ in eigensheaves. Now we use that $h^\ast f^\ast\EE$ is again a generating sheaf and it contains every irreducible representation so that the sheaf $h^\ast f^\ast(\FF\otimes\EE^\vee)$ must contain as a direct summand the trivial representation. As a consequence $\rho_\ast h^\ast f^\ast(\FF\otimes\EE^\vee)$ is non zero.  
\end{proof}
We still don't know if the morphism $\pi$ ``respects'' associated points of $\FF$.
\begin{lem}\label{lem:algebra-pure-sheaves}
  Let $f\colon\spec{B}\to\spec{A}$ be a surjective flat morphism of noetherian schemes, $E$ an $A$-module. We have the following:
  \begin{equation}
    \label{eq:etale-pullback}
    f\Ass{(E\otimes_A B)}=\Ass{(E)}
  \end{equation}
\end{lem}
\begin{proof}
  It is a special case of  \cite[Thm 12]{MR575344}.
\end{proof}
\begin{prop}\label{prop:pure-sheaves}
  Let $\XX\to\spec{k}$ be a  projective Deligne-Mumford stack and $\pi\colon\XX\to X$ the moduli scheme. Let $\FF$ be a coherent sheaf on $\XX$. If the sheaf $\FF$ is pure of dimension $d$ the sheaf $F_{\EE}(\FF)$ is pure of the same dimension.
\end{prop}
\begin{proof}
  We can use Theorem \ref{thm:tame-stacks-1}~\ref{item:7} to produce the usual local picture of a tame stack with a presentation:
  \begin{displaymath}
    \xymatrix@1{
      \coprod_i \spec{B_i} \ar@{->>}[r]^-{\chi} &  \XX \ar[r]^-{\pi} & X \\
       & \coprod_i[\spec{B_i}/G_i] \ar[r]^-{\rho}\ar@{->>}[u]^-{\phi} & \coprod_i\spec{A_i}\ar@{->>}[u]^-{\psi} \\
}
  \end{displaymath}
where vertical arrows are \'etale, the obvious map $\coprod_i\spec{B_i}\to[\spec{B_i}/G_i]$ composed with $\rho$ gives a finite morphism $h\colon\coprod\spec{B_i}\to\coprod_i\spec{A_i}$ and the square in the picture is cartesian. The sheaf $\chi^\ast(\FF\otimes\EE^\vee)$ is given by finitely generated $B_i$-modules $M_i$. It is clear that:
\begin{displaymath}
  \Ass{(\leftsub{A_i}{M_i})^{G_i}}\subseteq\Ass{(\leftsub{A_i}{M_i})}=h\Ass{M_i}
\end{displaymath}
and  $\Ass{(\leftsub{A_i}{M_i})^{G_i}}$ cannot be empty because of Lemma \ref{lem:pure-sheaves-supp}.
We can rewrite the formula as:
\begin{equation}\label{eq:5}
  \Ass{\rho_\ast\phi^\ast(\FF\otimes\EE^\vee)}\subseteq\Ass{h_\ast\chi^\ast(\FF\otimes\EE^\vee)}=h\Ass{\chi^\ast(\FF\otimes\EE^\vee)}=\rho\Ass{\phi^\ast(\FF\otimes\EE^\vee)}
\end{equation}
where the second equality follows from \ref{rem:pure-sheaves-supp}~\ref{item:17}. For the same reason if $\FF$ is pure of dimension $d$ the module $M_i$ is pure of the same dimension, moreover $h$ is finite and preserves the dimension of points so that $(\leftsub{A_i}{M_i})^G$ is pure of the same dimension. Since there are no embedded primes in $M_i$ and using Lemma \ref{lem:pure-sheaves-supp} we deduce that $M_i^G$ is pure of the same dimension. Using \ref{prop:cohom-base-change-spcoarse} we have $\rho_\ast\phi^\ast(\FF\otimes\EE^\vee)=\psi^\ast\pi_\ast(\FF\otimes\EE^\vee)$ that implies:
\begin{displaymath}
  \Ass{\psi^\ast\pi_\ast(\FF\otimes\EE^\vee)}\subseteq\rho\Ass{\phi^\ast(\FF\otimes\EE^\vee)}
\end{displaymath}
Using Lemma \ref{lem:algebra-pure-sheaves} we obtain:
\begin{displaymath}
  \Ass{\pi_\ast(\FF\otimes\EE^\vee)}=\psi\Ass{\psi^\ast\pi_\ast(\FF\otimes\EE^\vee)}\subseteq\psi\circ\rho\Ass{\phi^\ast(\FF\otimes\EE^\vee)}=\pi\Ass{\FF\otimes\EE^\vee}
\end{displaymath}
Since $\phi$ and $\psi$ are \'etale and preserve the dimension of points \cite[I Prop 3.14]{Met}, and  $h$ is finite and preserves the dimension of points,  we obtain that $F_\EE(\FF)$ is pure of dimension $d$. 
\end{proof}
With this result the following corollary is immediate:
\begin{cor}\label{cor:torsion-filtration-preserved}
  Let $\XX\to \Spec{k}$ be a projective DM stack and $\FF$ a pure coherent sheaf on $\XX$ of dimension $d$. Consider the torsion filtration $0\subset T_0(\FF)\subset\ldots \subset T_d(\FF)=\FF$. The functor $F_{\EE}$ sends the torsion filtration to the torsion filtration of $F_{\EE}(\FF)$ that is $F_{\EE}(T_i(\FF))=T_i(F_{\EE}(\FF))$.  
\end{cor}
We can sharpen a little bit the result in Proposition \ref{prop:pure-sheaves} and prove that the support of $F_{\EE}(\FF)$ is actually $\pi\supp{\FF}$.
\begin{cor}
  Let $\XX$ and $\FF$ be as in the previous theorem then $\supp{F_{\EE}(\FF)}=\pi\supp{\FF}$
\end{cor}
\begin{proof}
 According to Proposition \ref{prop:pure-sheaves} this could be false if and only if the generic point $p$ of an irreducible component of the support of $\FF$ is mapped to a point $\pi(p)$ that is not associated to the pure sheaf $F_{\EE}(\FF)$. Let $\Pi$ be the closure in $X$ of $\pi(p)$ and $\XX_{\Pi}$ the base change of $\XX$ to $\Pi$ with map $\pi_p\colon \XX_{\Pi}\to \Pi$. According to Proposition \ref{prop:cohom-base-change-spcoarse} we have ${\pi_P}_\ast (\FF\otimes\EE^\vee)\vert_{\XX_{\Pi}}\cong F_{\EE}(\FF)\vert_{\Pi}$. We observe that $\EE\vert_{\XX_{\Pi}}$ is again a generating sheaf and $\FF\vert_{\XX_{\Pi}}$ is pure of the dimension $d$, we can apply Lemma \ref{lem:pure-sheaves-supp} and conclude that  ${\pi_P}_\ast (\FF\otimes\EE^\vee)\vert_{\XX_{\Pi}}$ is a non zero sheaf pure of dimension $d$. However $\Pi$ is irreducible so that $\pi(p)$ must be associated.   
\end{proof}
\begin{rem}
In order to classify coherent sheaves on a scheme or a stack we consider three filtrations which let us split the problem in simpler pieces. The first one is the torsion filtration that reduces the problem to the study of pure sheaves, the second is the Harder-Narasimhan filtration that reduces the problem to the study of pure dimensional semistable sheaves, and the last one is the Jordan-H\"older filtration that reduces the problem to the study of stable sheaves. We have these three filtrations both on a projective stack $\XX$ and on its projective moduli scheme $X$; while the torsion filtration on $\XX$ is sent to the torsion filtration on $X$, the functor $F_{\EE}$ doesn't respect the other two filtrations as it will be clear in the following of this work.
\end{rem}

\subsection{Stability condition}
\begin{defn}
Let $\FF$ be a coherent sheaf on $\XX$, we define the following \textit{modified Hilbert polynomial}:
\begin{displaymath}
  P_{\EE}(\FF,m)=\chi(\XX,\FF\otimes\EE^\vee\otimes\pi^\ast\OO_X(m))=P(F_{\EE}(\FF)(m))=\chi(X,F_{\EE}(\FF)(m))
\end{displaymath}
\end{defn}
\begin{rem}\label{rem:additive-hilbert-polynomial}
\begin{enumerate}
\item If the sheaf $\FF$ is pure of dimension $d$, the function $m\mapsto  P_{\EE}(\FF,m)$ is a polynomial and we will denote it with:
\begin{equation}\label{item:14}
  P_{\EE}(\FF,m)=\sum_{i=0}^d \alpha_{\EE,i}(\FF)\frac{m^i}{i!}
\end{equation}
This is true since  the functor $F_{\EE}$ preserves the pureness and dimension of sheaves \ref{prop:pure-sheaves}, so that we can conclude as in the case of schemes. 
\item The modified Hilbert polynomial is additive on short exact sequences since the functor $F_{\EE}$ is exact \ref{rem:FE-exact} and the Euler characteristic is additive on short exact sequences.
\item The modified Hilbert polynomial is not a generalized Hilbert polynomial in the sense of Olsson and Starr \cite[Def 4.1]{MR2007396}.
\end{enumerate}
\end{rem}
\begin{defn}
As usual we define the \textit{reduced Hilbert polynomial} for pure sheaves, and we will denote it with $p_{\EE}(\FF)$; it is the monic polynomial with rational coefficients $ \frac{P_{\EE}(\FF)}{\alpha_{\EE,d}(\FF)}$.
\end{defn}
\begin{defn}
We define also the \textit{slope} of a sheaf of dimension $d$ (not necessarily pure):
\begin{displaymath}
  \hat{\mu}_{\EE}(\FF)= \frac{\alpha_{\EE,d-1}(\FF)}{\alpha_{\EE,d}(\FF)}
\end{displaymath}
\end{defn} 
We will also use the ordinary slope of a  sheaf $F$ on a scheme, and we will denote it with $\hat{\mu}(F)$ as usual (see \cite[Def 1.6.8]{MR1450870}). 

And here the definition of stability:
\begin{defn}
  Let $\FF$ be a pure coherent sheaf, it is \textit{semistable} if for every proper subsheaf $\FF'\subset\FF$ it is verified $p_{\EE}(\FF')\leq p_{\EE}(\FF)$ and it is stable if the same is true with a strict inequality. 
\end{defn}
\begin{rem}
\begin{enumerate}
\item The notion of $\mu$-stability and semistability  are defined in the evident way and they are related to Gieseker stability as they are in the case of schemes.
\item When $\XX$ is a scheme this notion of stability is not necessarily the same as ordinary Gieseker stability; it is clear that every locally free sheaf in this case is a generating sheaf. As a special example we can consider parabolic sheaves on a scheme $X$ provided with the special parabolic structure, that is a filtration $F(-D)\subset F$. Parabolic sheaves on $X$ with the special parabolic structure are obviously equivalent to sheaves on $X$, however parabolic stability requires $\EE$ to be $\OO_X(D)\oplus \OO_X$. In general this is not equivalent to Gieseker stability as observed in \cite{MR1162674}.
\item  The functor $F_{\EE}$ doesn't map semistable sheaves on $\XX$ to semistable sheaves on $X$. Indeed it induces a closed immersion of the $\Quot$-scheme of $\XX$ in the $\Quot$-scheme of $X$; this means that in general we have ``more quotients'' on $X$ then on $\XX$.
\end{enumerate}
\end{rem}
The \textit{multiplicity} or rank of the sheaf $F_{\EE}(\FF)$ is the usual thing: if the Hilbert polynomial of $\OO_X$ has coefficients $a_d(\OO_X),\ldots,a_0(\OO_X)$ it is given by:
\begin{displaymath}
  \rk{F_{\EE}(\FF)}=\frac{\alpha_{\EE,d}(\FF)}{a_d(\OO_X)}
\end{displaymath}
We can also try to relate the rank of the sheaf $\FF$ to the Hilbert polynomial. Let $P(\FF,m)=\chi(\FF\otimes\pi^\ast\OO_X(m))=\sum_{i=0}^d \alpha_i(\FF)\frac{m^i}{i!}$ be the Hilbert polynomial of $\FF$ with respect to the polarization $\pi^\ast\OO_X(1)$ alone. We could be tempted to define the rank of $\FF$ using this polynomial. Assume that $\XX$ is an orbifold, we can put $\rk{\FF}\coloneqq \frac{\alpha_d(\FF)}{\alpha_{d}(\OO_X)}$. This is a reasonable definition. Indeed if $\FF$ is locally free this is the rank of the free module. This is because the only contribution to the rank from the To\"en-Riemann-Roch formula is from the piece $\int_\XX \ch{(\FF\otimes\pi^\ast\OO_X(m))}\Td{(T_\XX)}$ (see the next subsection for some recall about the To\"en-Riemann-Roch formula). Assume that $\XX$ is a smooth Deligne-Mumford stack with non generically trivial stabilizer and $\FF$ is locally free. In this case this is not the rank of the locally free sheaf but the rank of a direct summand\footnotemark of $\FF$. \footnotetext{A quasicoherent sheaf on a $G$-banded gerbe, where $G$ is a diagonalizable group scheme decomposes according to the irreducible representations of the group (This is written in many papers). The direct summand is the one corresponding to the trivial representation.} 

To conclude the section we write a technical lemma. It states that given a flat family of sheaves the modified Hilbert polynomial is locally constant on the fibers. It replaces the analogous one for generalized Hilbert polynomials \cite[Lem 4.3]{MR2007396}.
\begin{lem}\label{lem:hilbert-poly-fibers}
  Let $\XX\to S$ be a family of projective stacks with chosen $\EE,\OO_X(1)$ and $\FF$ an $\OO_S$-flat sheaf on $\XX$. Assume $S$ is connected. There is an integral polynomial $P$ such that for every point $\Spec{K}\xrightarrow{s}S$ the modified Hilbert polynomial of the fiber $\chi(\XX_s,\FF\otimes\EE^\vee\otimes\pi^\ast\OO_X(m)\vert_{\XX_s})=P(m)$.     
\end{lem}
\begin{proof}
  Since $\pi$ preserves flatness and using \ref{prop:cohom-base-change-spcoarse} together with \ref{thm:existence-generating-sheaf}~\ref{item:16} we reduce the problem to the moduli scheme $X$. We have to prove that the integral polynomial $\chi(X_s,F_\EE(\FF)(m)\vert_{X_s})$ doesn't depend from $s$, but this is the statement of \cite[Thm 7.9.4]{MR0163911}.
\end{proof}
\begin{rem}
Using this lemma and generic-flatness (Prop \ref{prop:generic-flatness}) we can produce a stratification $\coprod S_i\to S$ such that on each $S_i$ the sheaf $\FF$ is flat and its modified Hilbert polynomial is constant. Again this is not the same as a flattening stratification since the universal property of a flattening stratification described by Mumford in \cite{MR0209285} is missing. Let $p\colon X\to S$ be a projective morphism of schemes, Mumford constructed the flattening stratification for such a morphism relying on the couple of functors $\Gamma_\ast,\widetilde{}\;$ where $\Gamma_\ast(F)=\bigoplus_{m\geq 0} p_\ast F(m)$ and $\;\widetilde{}\;$ is its inverse. In particular he was able to prove that $F$ is $S$-flat if and only if for all sufficiently large $m$ the sheaves $p_\ast F(m)$ are locally free. In the case of projective stacks $\XX\xrightarrow{\pi}X\xrightarrow{p}S$ we would need an analogous couple of functors. We could think of using $\Gamma_\ast\circ F_\EE$. As we will see in Lemma \ref{lem:left-inverse} it has a left inverse $\eta\circ\widetilde{}\;$; however it's evident that the statement $\FF$ is $S$-flat if and only if for all sufficiently large $m$ the sheaves $p_\ast F_{\EE}(\FF)(m)$ are locally free is quite false.
\end{rem}

\subsection{To\"en-Riemann-Roch and geometric motivations}

%

It is natural to ask if the degree of the sheaf $\FF$, computed with respect to $\pi^\ast\OO_X(1)$ is related to the slope $\widehat{\mu}_{\EE}$. It is, in a wide class of examples, but in general it is not. To explain this kind of relation we need some machinery from the paper of To\"en \cite{MR1710187}. We recall a couple of ideas from the Riemann-Roch theorem for smooth tame Deligne-Mumford stacks. For the sake of simplicity we restrict to the case where $\XX$ is a smooth tame Deligne-Mumford stack over an algebraically closed field $k$ and it is a global quotient of a scheme $Z$ by the action of a diagonalizable group scheme $G$. For the general case we redirect the reader to \cite[Thm 3.15]{MR1710187}, \cite[Def 4.5-4.7]{MR1710187} and \cite[Thm 4.10]{MR1710187}. Denote with $\sigma\colon I_{\XX}\to \XX$  the inertia stack. Let $\FF$ be a coherent sheaf on $\XX$ and consider the sheaf $\sigma^\ast\FF$. The inertia stack can be written as a disjoint union of Deligne-Mumford stacks $\XX_{g}=[Z^g/G]$ where $g\in G$ stabilizes the closed subscheme $Z^g\subset Z$. The coherent sheaf $\sigma^\ast\FF$ is a disjoint union of sheaves on the $\XX_g$'s. Each of these components carry an action of the cyclic group $\langle g \rangle$, whose order is coprime with the characteristic of $k$ by the tameness assumption. This implies that we can choose an isomorphism between $\langle g \rangle$ and $\mu_{a,k}$,  $a$ the order of $\langle g \rangle$, that sends $g$ to $\xi$, a generator of $\mu_{a,k}$. Moreover we can decompose the sheaf $\sigma^\ast\FF$ according to the irreducible representations of $\langle g \rangle$ in a direct sum of eigensheaves $\FF^{(z)}$ where $z\in\mu_{\infty}$. On each $\FF^{(z)}$ the element $g$ acts by multiplication by $z$. We will denote with  $\rho_{\XX}([\sigma^\ast\FF])$ the element $\sum_{z\in\mu_{\infty}} z[\FF^{(z)}]$ in $K_0(I_{\XX})\otimes_{\mathbb{Z}}\mathbb{Q}(\mu_{\infty})$. Denote also with $I^1_{\XX}$ the substack of the inertia stack made of connected components of codimension $1$ and by $\sigma_1\colon I^1_{\XX}\to \XX$ the composition of the inclusion of $I^1_\XX$ in $I_\XX$ and the morphism $\sigma$.  
\begin{prop}\label{prop:degree-mu}
  Assume that $\XX$ is a projective (connected) orbifold that is a  global quotient $[Z/G]$ where $G$ is diagonalizable. The generating sheaf $\EE$ is chosen so that $\rho_{\XX}([\sigma_1^\ast\EE])$ is a sum of locally free  sheaves of the same rank on each connected component of $I_\XX^1$ (the rank can change from a component to the other); let $\FF$ be a locally free sheaf of rank $r$  and $e$ be the rank of $\EE$ then we have:
  \begin{equation}\label{eq:9}
 \frac{\deg{(\FF)}}{r}=\frac{\alpha_{\EE,d-1}(\FF)}{re}-\frac{\alpha_{\EE,d-1}(\OO_\XX)}{e}   
  \end{equation}
\end{prop}
\begin{proof}
  This is a computation with the To\"en-Riemann-Roch formula. The degree receives  contributions only from pieces of codimension zero and codimension one in the inertia stack. Since we have assumed that $\XX$ is an orbifold the only piece in codimension zero contributing to $P_{\EE}(\FF,m)$ is:
  \begin{displaymath}
    \int_{\XX}\ch{(\FF\otimes\EE^\vee)}\ch{(\pi^\ast\OO_X(m))}\Td{(T_{\XX})}
  \end{displaymath}
Let $\sigma_g\colon\XX_g\to \XX$ be the connected components of $I^1_{\XX}$ and $\langle g \rangle\cong\mu_{n_g,k}$. Each component contributes with:
\begin{displaymath}
  \sum_{i,j=1}^{n_g} \xi^{i-j}\int_{\XX_g}\frac{\ch{(\FF_i\otimes\EE_j^\vee)}\ch{(\sigma_g^\ast\pi^\ast\OO_X(m))}\Td{(T_{\XX_g})}}{Q(\xi,c(N_{\XX_g\vert\XX}))}
\end{displaymath}
The coherent sheaves $\FF_i$ and $\EE_j$ are summands in the decomposition in eigensheaves of $\sigma_g^\ast\FF$ and $\sigma_g^\ast\EE$ with respect to the group $\mu_{n_g,k}$. The function $Q(\xi,c(N_{\XX_g\vert\XX}))$ is an opportune rational polynomial in $\xi$ and the Chern classes of the normal bundle $N_{\XX_g\vert\XX}$; the complex number   $\xi\in\mu_{n_g}$ is different from $1$. The sum $\sum_{j=1}^{n_g} \xi^{-j}$ vanishes, if we can also assume that $\rk{\EE_j}$ does not depend from  $j$ we retrieve the claimed identity with a little algebra.  
\end{proof}
\begin{rem}
\begin{enumerate}
\item If the stack has a non generically trivial stabilizer we have to take care of contributions to the To\"en-Riemann-Roch formula coming from twisted sectors of codimension zero. In order to retrieve formula  (\ref{eq:9}) we have to take care of the vanishing of expressions like:
  \begin{displaymath}
    \sum_{j=1}^{n_g} z^{-j}(\rk{(F_i)}c_1(\EE_j^\vee)+c_1(F)\rk{(\EE_j^\vee)})
  \end{displaymath}
where $z\neq 1$ and $n_g$ is again the order of the cyclic group generated by $g$.
We can achieve this requiring that $\rk{\EE_j^\vee}$ does not depend from $j$ and that the determinant of $\EE^\vee$ is some fixed invertible sheaf. This second request doesn't sound very reasonable in general.
\item There are cases where $\EE$ can be chosen so that $\rk{\EE_j^\vee}$ in Proposition \ref{prop:degree-mu} and in the previous point are equal to one. The component $\XX_g$ of the inertia stack can be considered as a $\langle g \rangle$-gerbe over a Deligne-Mumford stack and as a matter of fact $\langle g \rangle$-banded when $G$ is diagonalizable. The existence of an invertible sheaf over a $\langle g \rangle$-banded gerbe which is an eigensheaf with respect to the representation associated to a generator of the group (an invertible twisted sheaf) is not a trivial fact. If $\XX_g$ is a $\langle g \rangle$-gerbe over a scheme, such an invertible sheaf exists if and only if the gerbe is essentially trivial \cite[Lem 2.3.4.2]{MR2309155}. Despite this being a stringent assumption there are significant cases where this is satisfied. If $\XX$ is a toric stack the inertia stack is again toric and every toric gerbe is abelian and essentially trivial. In the toric case it is always possible to find a generating sheaf $\EE$ satisfying the condition in \ref{prop:degree-mu} or even in the case of a smooth  Deligne-Mumford curve with abelian stabilizers.   
\end{enumerate}
\end{rem}

\subsection{Harder-Narasimhan and Jordan-H\"older filtrations}

The last part of this section is devoted to the definition of the Harder-Narasimhan filtration and the Jordan-H\"older filtrations. The construction of these two filtrations doesn't differ from the case of sheaves on schemes which can be found in great detail in \cite[1.3]{MR1450870} and \cite[1.5]{MR1450870}; their existence in the case of stacks is granted by the fact that the functor $F_{\EE}$ is exact (Remark \ref{rem:FE-exact}) and that the modified Hilbert polynomial $P_{\EE}$ is additive for short exact sequences (Remark \ref{rem:additive-hilbert-polynomial}).  
\begin{defn}
  Let $\FF$ be a pure sheaf on $\XX$ . A strictly ascending filtration:
  \begin{displaymath}
    0=HN_0(\FF)\subset HN_1(\FF)\ldots \subset HN_l(\FF)=\FF
  \end{displaymath}
is a Harder-Narasimhan filtration if it satisfies the following:
\begin{enumerate}
\item the $i$-th graded piece $gr_i^{HN}(\FF)\coloneqq\frac{HN_i(\FF)}{HN_{i-1}(\FF)}$ is a semistable sheaf for every $i=1,\ldots,l$. 
\item denoted with $p_i=p_\EE(gr_i^{HN}(\FF))$, the reduced Hilbert polynomial are ordered in a strictly decreasing way:
\begin{displaymath}
  p_{\max{}}(\FF)\coloneqq p_1 > \ldots > p_l\eqqcolon p_{\min{}}(\FF)
\end{displaymath}
\end{enumerate}
\end{defn}
\begin{defn}
   Let $\FF$ be a semistable  sheaf on $\XX$  with reduced Hilbert polynomial $p_{\EE}(\FF)$. A strictly ascending filtration:
  \begin{displaymath}
    0=JH_0(\FF)\subset JH_1(\FF)\ldots \subset JH_l(\FF)=\FF
  \end{displaymath}
is a Jordan-H\"older filtration if $gr_i^{JH}(\FF)=\frac{JH_i(\FF)}{JH_{i-1}(\FF)}$ is stable with reduced Hilbert polynomial $p_{\EE}(\FF)$ for every $i=1,\ldots,l$. 
\end{defn}
\begin{thm}[{\cite[Thm 1.3.4]{MR1450870}}]\label{thm:harder-narasimhan}
  For every pure sheaf $\FF$ on $\XX$ there is a unique Harder-Narasimhan filtration.
\end{thm}
\begin{thm}[{\cite[Prop 1.5.2]{MR1450870}}]\label{thm:jordan-holder}
  For every semistable sheaf $\FF$ on $\XX$ there is a Jordan-H\"older filtration and the  sheaf $gr^{JH}(\FF)\coloneqq\bigoplus_i gr_i^{JH}(\FF) $ doesn't depend on the particular chosen filtration. 
\end{thm}
\begin{rem}\label{rem:harder-ecc}
  \begin{enumerate}
  \item All the summands $gr_i^{JH}(\FF)$ of the Jordan-H\"older filtration are semistable with reduced Hilbert polynomial $p_\EE(\FF)$; also the graded object \\ $\bigoplus_i gr_i^{JH}(\FF) $ is semistable with polynomial $p_\EE(\FF)$ \cite[1.5.1]{MR1450870}.
  \item If $\FF$ is  pure with Harder-Narasimhan filtration $HN_i(\FF)$ the sheaf $F_{\EE}(\FF)$ is again pure, $F_{\EE}(HN_i(\FF))$ is again a filtration but it is not the Harder-Narasimhan filtration in general. This is clear in the trivial case where the sheaf $\FF$ is already semistable,  and the sheaf $F_{\EE}(\FF)$ is not semistable and has a non trivial filtration. To fix the ideas we can think of the structure sheaf on $\XX$ which is semistable (stable) and $\pi_\ast\EE^\vee $ which is not semistable in most situations. 
  \item If $\FF$ is semistable the sheaf $F_{\EE}(\FF)$ is not; in general there is no hope to send a Jordan-H\"older filtration to a Jordan-H\"older filtration using the functor $F_{\EE}$. Again consider the simple case of an invertible sheaf $L$ on $X$. The pullback $\pi^\ast L$ is always stable on $\XX$ (we have a trivial filtration), however $F_{\EE}(\pi^\ast L)=L\otimes\pi_\ast\EE^\vee $ is not stable in general since $\pi_\ast\EE^\vee $ is not, and usually it is not even semistable.
  \end{enumerate}
\end{rem}
\begin{defn}
  As usual \cite[1.5.3-1.5.4]{MR1450870} two semistable sheaves $\FF_1,\FF_2$ with the same reduced modified Hilbert polynomial are called $S$-\textit{equivalent} if they satisfy \\ $gr^{JH}(\FF_1)\cong gr^{JH}(\FF_2)$. A semistable sheaf $\FF$ is \textit{polystable} if it is the direct sum of stable sheaves or equivalently $\FF\cong gr^{JH}(\FF)$.  
\end{defn}

\section{Boundedness}

In order to construct the stack of semistable sheaves as a finite type stack and a global quotient we first need to know if the family of semistable sheaves is bounded. In the previous section we have defined the Mumford regularity of a sheaf $\FF$ on a projective stack to be the Mumford regularity of $F_\EE(\FF)$, however it is not of great help to know that the family  $F_\EE(\FF)$ is bounded by a family of sheaves on the moduli scheme, since this family cannot be ``lifted'' to a bounding family on the stack. To obtain a boundedness result we need to study a \textit{Kleiman criterion} on the stack; the fact that we are using  Mumford regularity of $F_\EE(\FF)$  means that we are just going to consider regular sequences of hyperplane sections of the polarization $\OO_X(1)$ pulled back to the stack. A priori we could decide to study a more general class of sections, for instance the global sections of the generating sheaf $\EE$, however the generating sheaf is not suitable to produce the standard induction arguments that are commonly used with Mumford regularity.   
\subsection{Kleiman criterion for stacks}
\begin{ass}
In this section the morphism $p\colon\XX\to S$ will denote a family of projective stacks (Def \ref{def:projective-family}) on $S$ with a fixed polarization $\OO_X(1),\EE$.
\end{ass}
We prove here that general enough sequences of global sections of $\OO_X(1)$ are enough to establish a result of boundedness for semistable sheaves.
We recall a couple of results from Kleiman's \textit{expos\'e}  (\cite[XIII]{MR0354655}) about Mumford regularity and the definition and properties of $(b)$-sheaves.
Let $k$ be an algebraically closed field and $X$ a projective $k$-scheme with a very ample line bundle $\OO_X(1)$.
\begin{defn}[{\cite[1.7.1]{MR1450870}}]
Let $F$ be a coherent sheaf on $X$. The sheaf $F$ is $m$-regular (Mumford-Castelnuovo regular) if for every $i>0 $ we have $H^i(X,F(m-i))=0$. The regularity of $F$ denoted with $\reg(F)$ is the least integer $m$ such that $F$ is $m$-regular. 
\end{defn}
\begin{defn}
  We define the Mumford regularity of a coherent sheaf on $\XX$ to be the Mumford regularity of $F_{\EE}(\FF)$ on $X$ and we will denote it by $\reg_{\EE}(\FF)$. 
\end{defn}

\begin{prop}{\cite[XIII-1.2]{MR0354655}}\label{pro:m-regular}
  Let $F$ be a coherent $m$-regular sheaf on $X$. For $n\geq m$ the following results hold:
  \begin{enumerate}
    \item $F$ is $n$-regular\label{item:8}
    \item $H^0(F(n))\otimes H^0(\OO_X(1))\to H^0(F(n+1))$ is surjective \label{item:9}
    \item $F(n)$ is generated by its global sections.\label{item:10}
  \end{enumerate}
\end{prop}
  \begin{defn}
    Let $F$ be a coherent sheaf on $X$ and $r$ an integer $\geq\dim(\supp{F})$. Let $(b)=(b_0,\ldots,b_r)$ a collection of $r+1$ non negative integers. The sheaf $F$ is a $(b)$-sheaf if there is an $F$-regular sequence $\sigma_1,\ldots,\sigma_r$ of $r$ global sections of $\OO_X(1)$ such that, denoted with $F_i$ the restriction of $F$ to the intersection $\cap_{j\leq i}Z(\sigma_j)$ of the zero schemes of the sections ($i=0,\ldots,r$), the dimension of the global sections of $F_i$ is estimated by $h^0(F_i)\leq b_i$.   
  \end{defn}
Now let $\XX$ be a family of projective stacks on $S$ with moduli scheme $X$ and generating sheaf $\EE$. 
\begin{defn}
  Let $\FF$ be a coherent sheaf on $\XX$; it is defined to be a $(b)$-sheaf if $F_{\EE}(\FF)$ is a $(b)$-sheaf on $X$.
\end{defn}
\begin{rem}
  Let $\sigma_1,\ldots,\sigma_r$ an $F_{\EE}(\FF)$-regular sequence of sections of $\OO_X(1)$ making $\FF$ into a $(b)$-sheaf; let $Z(\sigma_i)$ the associated zero-scheme. The closed substack  $\pi^{-1}Z(\sigma_i)=Z(\pi^\ast\sigma_i)$ is the zero-stack of $\pi^\ast\sigma_i\in H^0(\XX,\pi^\ast\OO_X(1))$. Denote by $\FF_i$ the restriction to $\cap_{j\leq i}Z(\pi^\ast\sigma_j)$. An obvious application of \ref{prop:cohom-base-change-spcoarse} and exactness of $\pi_\ast$ imply that the following holds: $h^0(\FF_i\otimes\EE^\vee)\leq b_i$ for $i=0,\ldots,r$.
\end{rem}
As in \cite[XIII-1.9]{MR0354655} we define inductively a class of polynomials with rational coefficients $P_i\in\mathbb{Q}[x_0,\ldots,x_i]$:
\begin{displaymath}
  \begin{cases}
  & P_{-1}=0 \\
  & P_i(x_0,\ldots,x_i)=P_{i-1}(x_1,\ldots,x_i)+\sum_{j=0}^i x_j\binom{P_{i-1}(x_1,\ldots,x_i)-1+j}{j} \\ 
  \end{cases}
\end{displaymath}
\begin{prop}\label{prop:sub-b-sheaf}
   Let $\FF$ be a coherent $(b)$-sheaf on $\XX$. Every $\FF'$ subsheaf of $\FF$ is a $(b)$-sheaf.
\end{prop}
\begin{proof}
  It follows from \cite[XIII-1.6]{MR0354655} that every subsheaf $\FF''$ of $F_\EE(\FF)$ is a $(b)$-sheaf, and in particular $F_{\EE}(\FF')$ is such a subsheaf.  
\end{proof}
\begin{prop}\label{prop:b-sheaves-m-reg}
    Let $\FF$ be a coherent $(b)$-sheaf on $\XX$ with Hilbert polynomial $P_{\EE}(\FF,m)=\sum_{i=0}^r a_i\,\binom{m+i}{i}$ and $(b)=(b_0,\ldots,b_r)$; let $(c)=(c_0,\ldots,c_r)$ be integers such that $c_i\geq b_i-a_i$ for $i=0,\ldots,r$ then $n\coloneqq P_r(c_0,\ldots,c_r)$ is non negative and $\FF$ is $n$-regular.
\end{prop}
\begin{proof}
  This has nothing to do with the stack $\XX$ so that the proof in \cite[XIII-1.11]{MR0354655} is enough.
\end{proof}
To state the Kleiman criterion we first need to recall the notion of family of sheaves and bounded family.
Given $s$ a point of $S$ with residue field $k(s)$ and given also $K$ a field extension of $k(s)$ a sheaf on a fiber is a coherent sheaf $\FF_K$ on $\XX\times_S\Spec{K}$. If we are given two field extensions $K,K'$ and two sheaves, respectively $\FF_K,\FF_{K'}$, they are equivalent if there are $k(s)$-homomorphisms of $K,K'$ to a third extension $K''$ of $k(s)$ such that $\FF_{K}\otimes_{k(s)}K''$ and $\FF_{K'}\otimes_{k(s)}K''$ are isomorphic. 
\begin{defn}
  A \textit{set-theoretic family} of sheaves on $p\colon\XX\to S$ is a set of sheaves defined on the fibers of $p$. 
\end{defn}
\begin{defn}
  A set theoretic family $\mathfrak{F}$ of sheaves on $\XX$ is \textit{bounded} if there is an $S$-scheme $T$ of finite type and a coherent sheaf $\mathcal{G}$ on $\XX_T\coloneqq\XX\times_S T$ such that every sheaf in $\mathfrak{F}$ is contained in the fibers of $\mathcal{G}$. 
\end{defn}
\begin{rem}
  It is quite standard in the literature to define a family of sheaves to be a flat sheaf. However according to the generic flatness result every coherent sheaf is flat on some finite stratification of the base scheme. For this reason it really doesn't matter if we bound a set theoretic family with a sheaf or a flat sheaf. We have decided to keep Kleiman's definition.  
\end{rem}
The following theorem is the stacky version of \cite[XIII-1.13]{MR0354655}.
\begin{thm}[Kleiman criterion for stacks]\label{thm:kleiman-criterion}
 Let $p\colon\XX\to S$ be a family of projective stacks with moduli scheme $\pi:\XX\to X$. Assume $\OO_X(1)$ is chosen so that for every point $s$ of $S$ the line bundle restricted to the fiber $\OO_{X_s}(1)$ is generated by the global sections (for instance $\OO_X(1)$ is very ample relatively to $X\to S$). Let $\mathfrak{F}$ be a set-theoretic family of coherent sheaves on the fibers of $\XX\to S$. The following statements are equivalent:
 \begin{enumerate}
 \item The family $\mathfrak{F}$ is bounded by a coherent sheaf $\mathcal{G}$ on $\XX_T\coloneqq\XX\times_S T$. Moreover if every $\FF_K\in\mathfrak{F}$ is locally free of rank $r$ a bounding family can be chosen locally free of rank $r$ ($\FF_K$ is a sheaf on a $K$-fiber of $\XX\to S$).
 \item The set of the Hilbert polynomials $P_{\mathcal{E}_K}(\FF_K)$ for $\FF_K\in\mathfrak{F}$ is finite and there is a sequence of integers $(b)$ such that every $\FF_K$ is a $(b)$-sheaf ($\EE_K$ is the $K$-fiber of the generating sheaf $\EE$).
\item The set of Hilbert polynomials $P_{\mathcal{E}_K}(\FF_K)$ for $\FF_K\in\mathfrak{F}$ is finite and there is a non negative integer $m$ such that every $\FF_K$ is $m$-regular.
\item The set of Hilbert polynomials $P_{\mathcal{E}_K}(\FF_K),\,\FF_K\in\mathfrak{F}$ is finite  and there is a coherent sheaf $\mathcal{H}$ on $\XX_T$ such that every $\FF_K$ is a quotient of $\mathcal{H}_K$ for some point $K$-point in $T$. We can assume that $T=S$ and $\mathcal{H}=\EE^{\oplus N}\otimes\pi^\ast\OO_X(-m)$ for some integers $N,m$.
\item There are two coherent sheaves $\mathcal{H},\mathcal{H}'$ on $\XX_T$ such that every $\FF_K$ is the cokernel of a morphism $\mathcal{H}'_K\to\mathcal{H}_K$ for some $K$-point of $T$. Moreover we can assume that $T=S$ and  $\mathcal{H}=\EE^{\oplus N}\otimes\pi^\ast\OO_X(-m),\mathcal{H}'=\EE^{\oplus N'}\otimes\pi^\ast\OO_X(-m')$. 
 \end{enumerate}
\end{thm}
\begin{proof}
  Part of the proof, that is $1\Rightarrow 2,\, 2\Rightarrow 3$ is just the proof in the expos\'e of Kleiman. 



$(3)\Rightarrow (4)\;$ Take $m$ such that every $\FF_K$ is $m$-regular. Let $N$ be the maximal $P_{\EE_K}(\FF_K,m)=h^0(F_{\EE_K}(\FF_K)(m))$. According to Proposition \ref{pro:m-regular}~\ref{item:10} there is a surjective map $\OO_{X_T}^{\oplus N}\otimes \OO_{X_T}(-m)\to F_{\EE_T}(\mathcal{G})$. Applying $G_{\EE_T}$ and composing with $\theta_{\EE_T}(\mathcal{G})$ we obtain the surjection we wanted.

$(4)\Rightarrow (5)\;$ Assume that there is a coherent sheaf $\mathcal{H}$ on $\XX_T$ satisfying $(4)$. For every $\FF_K$ there is a point $t$ such that:
  \begin{displaymath}
    0\to F_{\EE_T}(\FF'_t) \to F_{\EE_T}(\mathcal{H}_t) \to F_{\EE_K}(\FF_K)\to 0
  \end{displaymath}
Since $\mathcal{H}$ is bounded, the number of Hilbert polynomials $P_{\EE_T}(\mathcal{H}_t)$ is finite and there is $(b)$ such that every $\mathcal{H}_t$ is a $(b)$-sheaf. The number of Hilbert Polynomials of $\FF_K$ is finite by hypothesis so that the number of Hilbert polynomials of $\FF'_t$ is finite too. Moreover according to \ref{prop:sub-b-sheaf} $\FF'_t $ are $(b)$-sheaves. We can apply $(2)\Rightarrow (4)$ to $\FF'_t$ and deduce $(5)$.

$(5)\Rightarrow (1)\;$ First we prove that $\mathcal{H}$ and $\mathcal{H}'$ can be chosen of the kind $\EE_T^{\oplus N}\otimes\pi^\ast\OO_{X_T}(-m)$. 
Given a point $t\in T$ consider a cokernel $\mathcal{H}'_t\xrightarrow{\beta}\mathcal{H}_t\to\coker{\beta}\to 0$. We apply the functor $F_{\EE_T}$ and  observe that it commutes with $\otimes k(t)$ so that the family $F_{\EE_t}(\coker{\beta})$ belongs to the cokernels of the fibers of the two sheaves $F_{\EE_T}(\mathcal{H}')$ and   $F_{\EE_T}(\mathcal{H})$. According to the Kleiman criterion for coherent sheaves on a scheme the family $F_{\EE_t}(\coker{\beta})$ is bounded and in particular the number of Hilbert polynomials $P_{\EE_t}(\coker{\beta})$ is finite. Since $\mathcal{H}$ is bounded we can assume that there is a surjection $\mathcal{L}\coloneqq\EE_T^{\oplus N}\otimes\pi^\ast\OO_{X_T}(-m)\to\mathcal{H}$ so that we have an exact sequence:
\begin{displaymath}
  0\to \mathcal{C}_\beta\to\mathcal{L}_t\to \coker{\beta}\to 0
\end{displaymath}
The family $\mathcal{C}_\beta$ has just a finite number of different Hilbert polynomials since $\mathcal{L}$ is bounded and $\coker{\beta}$ has a finite number of polynomials. Moreover there is $(b)$ such that every sheaf $\mathcal{L}_t$ is a $(b)$-sheaf and since $\mathcal{C}_\beta$ are all subsheaves of $\mathcal{L}_t$ they are $(b)$-sheaves according to \ref{prop:sub-b-sheaf}. Using $(2)\Rightarrow(3)\Rightarrow(4)$ we deduce that there is a sheaf $\EE_T^{\oplus N'}\otimes\pi^\ast\OO_{X_T}(-m_1)$ such that the family $C_\beta$ is contained in the quotients of its fibers. This completes the first part. 
To complete the proof we take a finite stratification of $T$ so that the coherent sheaf $\HOM_{\OO_{\XX_T}}(\mathcal{H}',\mathcal{H})$ is flat on $T$ and $R^i{p_T}_\ast\HOM_{\OO_{\XX_T}}(\mathcal{H}',\mathcal{H})$ are locally free for every $i\geq 0$ (they are a finite number of sheaves according to the proof of \ref{prop:arbitrary-base-change}). By Proposition \ref{prop:arbitrary-base-change} we obtain that ${p_T}_\ast$ commutes with an arbitrary base-change. This implies that the representable functor $\mathbb{V}({p_T}_\ast\HOM_{\OO_{\XX_T}}(\mathcal{H}',\mathcal{H}))=\mathbb{V}$ is the same as a functor associating to a map $f\colon U\to T$ the group $\Gamma(U,{p_U}_\ast\widetilde{f}^\ast\HOM_{\OO_{\XX_T}}(\mathcal{H}',\mathcal{H}))$ where $\widetilde{f}\colon\XX_U\to\XX_T$. To conclude we observe that $\mathbb{V}$ is a vector bundle and the map $\mathbb{V}\to T$ is smooth and in particular flat so that the universal section is an object $\mathcal{U}\in \Gamma(\mathbb{V},{\pi_{\mathbb{V}}}_\ast\HOM_{\OO_{\XX_\mathbb{V}}}(\mathcal{H}'\vert_{\mathbb{V}},\mathcal{H}\vert_\mathbb{V}))$. Eventually we obtain a universal quotient:
\begin{displaymath}
  \mathcal{H}'\vert_{\mathbb{V}}\xrightarrow{\mathcal{U}}\mathcal{H}\vert_{\mathbb{V}}\to \mathcal{G}\to 0 
\end{displaymath}
where $\mathcal{G}$ bounds the family $\mathfrak{F}$. 
\end{proof}

We state here a useful lemma of Grothendieck about the boundedness of families of sheaves. The version of this lemma for schemes \cite[Lem 1.7.9]{MR1450870} does not require the Kleiman criterion, however in the case of stacks there is an easy way to pull-back the result from the moduli scheme using the Kleiman criterion for stacks.
\begin{lem}[Grothendieck]\label{lem:grothendieck-lemma}
Let $\XX$ be a projective stack over a field $k$ with moduli scheme $\pi\colon\XX\to X$. Let $P$ be an integral polynomial of degree $d=\dim{(X)}$ ($0\leq d\leq\dim{(X)}$) and $\rho$ an integer. There is a constant $C=C(P,\rho)$ such that if $\FF$ is  coherent sheaf of dimension $d$ on $\XX$ with $P_{\EE}(\FF)=P$ and $\reg_{\EE}(\FF)\leq\rho$, then for every $\FF'$ purely $d$-dimensional quotient $\hat{\mu}_{\EE}(\FF')\geq C$. Moreover, the family of purely $d$-dimensional quotients $\FF'_i,\, i\in I$ (for some set of indices $I$) with $\hat{\mu}_{\EE}(\FF'_i)$ bounded from above is bounded.
\end{lem}
\begin{proof}
  The first part of the lemma is just the original Grothendieck lemma applied to the moduli scheme, we have just to observe that if $\FF$ has dimension $d$ the sheaf $F_{\EE}(\FF)$ has the same dimension.
To prove the second part we observe that the lemma in the case of schemes provides us a coherent sheaf $\mathcal{G}$ on $X\times R$ for some finite type scheme $R\to\spec{k}$ bounding the family of quotients  $\pi_\ast(\FF'_i\otimes\EE^\vee)$. A bounded family of sheaves on a projective scheme has a finite number of Hilbert polynomial, so in particular the number of polynomials $P_\EE(\FF_i')$ is finite. We can pull back the problem to the stack using the functor $G_\EE$, and obtain that the family $\FF_i'$ is contained in the quotients of $\pi_R^\ast\mathcal{G}\otimes \EE_R$ and we write it as:
  \begin{displaymath}
    \xymatrix{
    \pi^\ast\mathcal{G}_i\otimes\EE_i \ar@{->>}[r] & \FF_i'
  }
  \end{displaymath}
The family of sheaves $ \pi_R^\ast\mathcal{G}\otimes_{\OO_{\XX\times R}}\EE_R$ is bounded and in particular it is a quotient of a sheaf  $\EE_R^{\oplus N}\otimes\pi_R^\ast\OO_{X\times R}(-m)$ for some $m$ and $N$; applying \ref{thm:kleiman-criterion} we deduce that $\FF_i'$ is a bounded family.
\end{proof}
\begin{rem}
  The same statement is obviously true if $\FF'$ is a subsheaf (a family of subsheaves) and the inequalities are all inverted.
\end{rem}

With this machinery we can prove that semistability and stability are open conditions.
\begin{prop}\label{prop:semistable-open}
Let $F$ be a flat family of $d$-dimensional coherent sheaves on $p\colon\XX\to S$  (a family of projective stacks again) and fixed modified Hilbert Polynomial $P$ of degree $d$. The set  $ \{\text{$\,s\in S\,\vert\; F_s\;$ is pure and semistable} \}$ is open in $S$. The same is true for stable sheaves and geometrically stable sheaves.
\end{prop}
\begin{proof}
  The same proof as in \cite[2.3.1]{MR1450870} but using the Grothendieck lemma for stacks and  projectivity of the $\Quot$-scheme for sheaves on stacks proved in \cite{MR2007396}. 
\end{proof}
\begin{cor}
  The stack of semistable sheaves on $\XX$ is an algebraic open substack of $\mathfrak{Coh}_{\XX /S}$.
\end{cor}
\begin{proof}
  It follows from the previous one and Corollary \ref{cor:alg-stack-of-coh-sh}.
\end{proof}
\begin{cor}\label{cor:corollary-for-boundedness}
  Let $\XX\to X\to S$ be a polarized stack satisfying hypothesis of Theorem \ref{thm:kleiman-criterion}, and $\mathfrak{F}$ is a set-theoretic family of coherent sheaves on its fibers. The family $\mathfrak{F}$ is bounded if and only if $F_{\EE}(\mathfrak{F})$ is bounded.  
\end{cor}
\begin{proof}
  If $\mathfrak{F}$ is bounded then there is $(b)$ such that $F_{\EE}(\mathfrak{F})$ are $(b)$-sheaves or equivalently there is an integer $m$ such that $F_{\EE}(\mathfrak{F})$ are $m$-regular. From the Kleiman criterion for schemes it follows that $F_{\EE}(\mathfrak{F})$ is a bounded family.
If $F_{\EE}(\mathfrak{F})$ is a bounded family from the Kleiman criterion for schemes it follows that there is $(b)$ or equivalently there is $m$ such that $F_{\EE}(\mathfrak{F})$ are $(b)$-sheaves or $m$-regular; from the Kleiman criterion for stacks this implies that $\mathfrak{F}$ is bounded.
\end{proof}

\subsection{A numerical criterion for boundedness}
With the last corollary we have reduced the problem of boundedness to a study of boundedness on the moduli scheme $X$ of the family of projective stacks. 
Working on the moduli scheme we have at disposal very fine results to establish whether  a family of sheaves is bounded or not: in characteristic zero we can use the well known theorem of Le Potier and Simpson \cite[3.3.1]{MR1450870} relying on the  Grauert M\"ulich theorem, in positive and mixed characteristic we can use a finer result of Langer in \cite[Thm 4.2]{MR2051393}.

Let $\mathfrak{F}$ be the family of pure semistable sheaves of dimension $d$ on  the fibers of $\XX\to S$ with fixed modified Hilbert polynomial $P$; as we have noticed before it is not true that $F_{\EE}(\mathfrak{F})$ are semistable, however we can study how much this family is destabilized and try to bound this loss of stability. 
Given $\FF$ in the family we can consider the Harder-Narasimhan filtration $0\subset F_n\subset\ldots\subset F_1\subset F_{\EE}(\FF)$ and look for some estimate of the maximal slope ($\hat{\mu}(F_n)$) depending only on the fixed Hilbert polynomial and possibly the sheaf $\EE$ and the geometry of $\XX$. 
The rest of this section is devoted to this problem.    
First of all we show a simple result of boundedness for smooth projective curves  whose proof is analogous to the one for schemes.
\begin{prop}\label{prop:sheaves-on-curves-bounded}
Let $\XX\to\spec{k}$ be a  smooth projective stack of dimension $1$ and  $k$ is an algebraically closed field. The family of torsion-free semistable sheaves on $\XX$ is bounded 
\end{prop}
\begin{proof}
  This is an application of Serre duality for stacky curves and Kleiman criterion for stacks as in \cite[Cor 1.7.7]{MR1450870}.
\end{proof}
There is also a very standard result of Maruyama and Yokogawa about the boundedness of parabolic bundles:
\begin{prop}
 Let $X\to\spec{k}$ be a smooth projective scheme, and consider a root construction on it $\pi\colon\XX\to X$. The family of semistable locally free sheaves on $\XX$ is bounded.  
\end{prop}
\begin{proof}
  This is a direct computation that can be found in the original paper \cite{MR1162674}. 
\end{proof}

What follows is devoted to the study of the problem in a greater generality.
Let $\FF$ a coherent sheaf on $\XX$ and $P$ a polynomial with integral coefficients. We will denote with $\Quot_{\XX/S}(\FF,P)$ the functor of quotients of $\FF$ with modified Hilbert polynomial $P$.  The natural transformation $F_\EE$ maps $\Quot_{\XX/S}(\FF,P)$ to the ordinary $\Quot$-scheme $\Quot_{X/S}(F_{\EE}(\FF),P)$ of quotient sheaves on $X$ with ordinary Hilbert polynomial $P$.
\begin{prop}\label{prop:closed-embedding-modified-vers}
  The natural transformation $F_{\EE}$ is relatively representable with schemes and actually a closed immersion. In particular $\Quot_{\XX/S}(\FF,P)$ is a projective scheme.
\end{prop}
\begin{proof}
  It is the same proof as in \cite[Lem 6.1]{MR2007396} and \cite[Prop 6.2]{MR2007396} but using Lemma \ref{lem:hilbert-poly-fibers} instead of \cite[Lem 4.3]{MR2007396}.
\end{proof}
Consider now $T$ an $S$-scheme, $\FF$ a coherent sheaf on $\XX_T$ and $P$ as before. We recall the definition of the natural transformation $\eta_T\colon\Quot_{X/S}(F_{\EE}(\FF),P)(T)\to\Quot_{\XX/S}(\FF)(T)$ from \cite{MR2007396}. Let $F_\EE(\FF)\xrightarrow{\rho}Q$ be a quotient sheaf in $\Quot_{X/S}(F_{\EE}(\FF),P)(T)$. First consider the kernel $K\xrightarrow{\sigma}F_{\EE}(\FF)$, apply $G_{\EE}$ and compose with the natural morphism $\theta_{\EE}$:
\begin{displaymath}
\xymatrix{
  G_{\EE}(K)\ar[r]^-{G_{\EE}(\sigma)} & G_{\EE}F_{\EE}(\FF) \ar@{->>}[r]^-{\theta_{\EE}(\FF)} & \FF
}
\end{displaymath}
Let $\mathcal{Q}$ be the cokernel of this composition so that we have defined a quotient $\FF\to\mathcal{Q}$ which is $\eta_T(\rho)$. 
\begin{lem}{\cite[Lem 6.1]{MR2007396}}\label{lem:left-inverse}
  Let $\mathcal{Q}$ be a coherent sheaf in $\Quot_{\XX/S}(\FF)(T)$, the composition of natural transformations $\eta_T(F_{\EE}(\mathcal{Q}))$ is the same sheaf $\mathcal{Q}$ moreover the association $T\mapsto \eta_T$ is functorial.
\end{lem}
\begin{rem}
  In this context there is no reason why $\eta_T(Q)$ should have modified Hilbert polynomial $P$, unless it is in the image of $F_{\EE}$.
\end{rem}

\begin{lem}
  Let $\FF$ be a quasicoherent sheaf on $\XX$. Let $Q$ be the following quotient:
  \begin{displaymath}
    \xymatrix{
0 \ar[r] &  K \ar[r]^-{\alpha} & F_{\EE}(\FF) \ar[r]^-{\beta} & Q \ar[r] & 0
}
  \end{displaymath}
Let $\FF\to\mathcal{Q}$ be the quotient associated to $Q$ by the transformation $\eta_S$ and let $\mathcal{K}$ be its kernel, then we have the following $9$-roman:
\begin{displaymath}
  \xymatrix{
     & 0 & & & & 0 \\
  0  \ar[r]  & A\ar[u]\ar[r] & Q\ar[rr]\ar[ul] & & F_{\EE}(\mathcal{Q})\ar[ur]\ar[r] & 0 \\
   & & & F_{\EE}(\FF)\ar[ur]^-{F_{\EE}(\beta)}\ar[ul] & & \\
  0\ & B\ar[l]\ar[uu]^{\wr} & F_{\EE}(\mathcal{K})\ar[ur]\ar[l] & & K\ar[ul]^-{\alpha}\ar[ll]^{\gamma} & 0\ar[l] \\
  & 0\ar[u]\ar[ur] & & & & 0\ar[ul] \\
  }
\end{displaymath}
Moreover the map $\gamma$ factorizes in the following way:
\begin{displaymath}
  \xymatrix{
    K\ar@{>->}[rr]^-{\varphi_{\EE}(K)} && F_{\EE}\circ G_{\EE}(K) \ar@{->>}[rr]^-{\widetilde{\gamma}} && F_{\EE}(\mathcal{K}) 
  }
\end{displaymath}
\end{lem}
\begin{proof}
  First of all we produce the sheaf $\QQ$ using the following diagram:
  \begin{displaymath}
    \xymatrix{
     & G_{\EE}(K) \ar@{->>}[d] \ar[r]^-{G_{\EE}(\alpha)}\ar[rd]^{\delta} & G_{\EE}\circ F_{\EE}(\FF) \ar[r]\ar@{->>}[d]^-{\theta_{\EE}(\FF)} & G_{\EE}(Q) \ar[r]\ar@{->>}[d] & 0 \\
    0 \ar[r] & \mathcal{K} \ar[r] & \FF \ar[r]^{\beta} & \QQ \ar[r] & 0 \\
    }
  \end{displaymath}
where $\QQ=\coker{\theta_\EE(\FF)\circ G_\EE(\alpha)}$ and $\mathcal{K}=\ker{\beta}$. Now we apply the exact functor $F_{\EE}$ and use the transformation $\varphi_{\EE}$ and formula (\ref{eq:2}) to obtain the following diagram:
 \begin{displaymath}
    \xymatrix{
   0 \ar[r] & K \ar[r]^-{\alpha}\ar@{>->}[d]^-{\varphi_{\EE}(K)} & F_{\EE}(\FF) \ar[r]\ar@{>->}[d]^-{\varphi_{\EE}(F_{\EE}(\FF))} & Q \ar@{>->}[d]^-{\varphi_{\EE}(Q)}\ar[r] & 0 \\
    & K\otimes\pi_\ast\END_{\OO_\XX}(\EE) \ar@{->>}[d]^-{\widetilde{\gamma}} \ar[r]^-{\alpha\otimes\Id}\ar[rd]_-{F_{\EE}(\delta)} &  F_{\EE}(\FF)\otimes\pi_\ast\END_{\OO_\XX}(\EE) \ar[r]\ar@{->>}[d]^-{F_{\EE}(\theta_{\EE}(\FF))} & Q\otimes\pi_\ast\END_{\OO_\XX}(\EE) \ar[r]\ar@{->>}[d] & 0 \\
    0 \ar[r] & F_{\EE}(\mathcal{K}) \ar[r]  & F_{\EE}(\FF) \ar[r]^-{F_{\EE}(\beta)} & F_{\EE}(\QQ) \ar[r] & 0 \\ }
\end{displaymath}
The middle column is the identity according to lemma \ref{lem:idenity-theta-phi} so that the left column is injective and the right column is surjective. It is immediate to produce the $9$-roman.
\end{proof}
\begin{prop}\label{prop:stima-max-destab}
Let $\XX$ be a projective polarized stack over an algebraically closed field $k$.  Let $\FF$ be a pure $\mu$-semistable sheaf  on $\XX$. Let $\overline{F}$ be the maximal destabilizing sheaf of $F_{\EE}(\FF)$. Take $\widetilde{m}$ an integer such that $\pi_\ast\END_{\OO_\XX}(\EE)(\widetilde{m})$ is generated by the global sections and denote with $N=h^0(X,\pi_\ast\END_{\OO_X}(\EE)(\widetilde{m}))$. The following inequality holds:
  \begin{equation}\label{eq:3}
  \hat{\mu}_{\max{}}(F_{\EE}(\FF))=\hat{\mu}(\overline{F})\leq \hat{\mu}_{\EE}(\FF) +\widetilde{m}\deg{(\OO_X(1))} 
  \end{equation}
\end{prop}
\begin{proof}
The coherent sheaf $\pi_\ast\END_{\OO_\XX}{\EE}$ is unstable in almost every example, however we have a surjection 
\begin{displaymath}
\xymatrix{
  \OO_X(-\widetilde{m})^{\oplus N} \ar@{->>}[r] & \pi_\ast\END_{\OO_\XX}{\EE}
}
\end{displaymath}
 with $\widetilde{m},N$ as in the hypothesis, which is given by the evaluation map. Since the sheaf $\overline{F}$ is semistable the sheaf $\overline{F}(-\widetilde{m})$ is again semistable. This is an immediate consequence of Riemann-Roch for projective schemes. Moreover we observe that $\hat{\mu}(\overline{F}(-\widetilde{m})^{\oplus N})=\hat{\mu}(\overline{F}(-\widetilde{m}))$ and in particular it doesn't depend on $N$. Therefor the sheaf $\overline{F}(-\widetilde{m})^{\oplus N}$ is semistable (for a proof see \cite[Lem 1.2.4.ii]{MR561910}). We have a surjection $\overline{F}\otimes\OO_X(-\widetilde{m})^{\oplus N}\to F_{\EE}(\overline{\FF}) $ where $\overline{\FF}$ is the sheaf associated to $\overline{F}$ by the transformation $\eta_k$. Since it is a subsheaf of $\FF$ it is pure. Using that $\overline{F}(-\widetilde{m})^{\oplus N}$ is $\mu$-semistable we obtain:  

\begin{displaymath}
\hat{\mu}(\overline{F}\otimes_X\OO_X(-\widetilde{m})^{\oplus N})\leq \hat{\mu}_{\EE}(\overline{\FF})\leq \hat{\mu}_{\EE}(\FF)  
\end{displaymath}
where the second inequality comes from the fact that $\FF$ is $\mu$-semistable and $\overline{\FF}$ is a sub-sheaf. 
The desired inequality follows from this one with a simple computation.
\end{proof}

\subsection{Two results  of Langer}
To complete the proof of  boundedness for semistable sheaves we have just to use the result of Langer about the boundedness of sheaves on projective schemes together with \ref{prop:stima-max-destab}.
We first recall the precise statement of \cite[Th 4.2]{MR2051393}
\begin{thm}[Langer]
Let $q\colon X\to S$ be a projective morphism of schemes of finite type over an algebraically closed field, let $\OO_X(1)$ be a $q$-very ample locally free sheaf on $X$. Let $P$ be an integral polynomial of degree $d$ and $\mu_0$ is a real number. The set-theoretic family of pure sheaves of dimension $d$ on the geometric fibers of $q$ with fixed Hilbert polynomial $P$ and maximal slope bounded by $\mu_0$ is bounded.
\end{thm}
In order to use \ref{prop:stima-max-destab} for a family of projective stacks $p\colon \XX\to S$ we need a homogeneous bound for $\widetilde{m}$ in the theorem for every fiber of $\EE$ and a bound for $\deg(\OO_{X_s}(1))$ for every geometric point $s$ of $S$.
\begin{lem}\label{lem:stima-omo-di-muhat}
  Let $p\colon \XX\to S$ be family of projective stacks polarized by $\EE,\OO_X(1)$. There is an integer $\widetilde{m}$ and a geometric point $\overline{s}$ of $S$ such that for every sheaf $\FF$ in the family of purely $d$-dimensional semistable sheaves on the fibers of $p$ with fixed modified Hilbert polynomial $P$ we have:
  \begin{equation}
    \label{eq:stima-omo-di-maxslope}
    \hat{\mu}_{\max{}}(F_{\EE_s}(\FF))\leq \hat{\mu}_{\EE}(\FF)+\widetilde{m}\deg(\OO_{X_{\overline{s}}}(1))
  \end{equation}
where $s$ is the point of $S$ on which $\FF$ is defined.
\end{lem}
\begin{proof}
  Let $\widetilde{m}$ be the integer such that $\pi_\ast\END_{\OO_\XX}(\EE)(\widetilde{m})$ is generated by the global sections. Since $k(s)$ is right exact for every point $s$ and it commutes with $F_{\EE}$ this $\widetilde{m}$ has the desired property on each fiber.
Choose a finite flat stratification of $S$ for $\OO_X(1)$. Since the Euler characteristic $\chi(X_s,\OO_{X_s}(1))$ is locally constant the function $s\mapsto \deg(\OO_{X_s}(1))$ assume only a finite number of values, in particular we can choose $\overline{s}$ such that $\deg(\OO_{X_{\overline{s}}}(1))$ is maximal.
\end{proof}
\begin{thm}
Let $p\colon\XX \to S$ be a family of projective stacks over an algebraically closed field, polarized by $\EE,\OO_X(1)$. Let $P$ be an integral polynomial of degree $d$ and $\mu_0$ a real number.
\begin{enumerate}
\item  Every set-theoretic family $\FF_i,\, i\in I$ ($I$ a set) of purely $d$-dimensional sheaves on the fibers of $p$ with fixed modified Hilbert polynomial and such that $\hat{\mu}_{\max{}}(F_{\EE}(\FF_i))\leq\mu_0$ is bounded. 
\item The family of semistable purely $d$-dimensional sheaves on the fibers of $q$ with fixed modified Hilbert polynomial $P$ is bounded.
\end{enumerate} 
\end{thm}

\begin{proof}\label{cor:boundedness-eventually}
$(1)$ It is  an immediate consequence of \cite[Thm 4.2]{MR2051393} that $F_{\EE}(\FF_i)$ form a bounded family, so according to Corollary \ref{cor:corollary-for-boundedness} the family $\FF_i$ is bounded too.

$(2)$ We choose $\widetilde{m}$ and $\overline{s}$ as in Lemma \ref{lem:stima-omo-di-muhat}; it follows from Corollary \ref{prop:stima-max-destab} that we have:
  \begin{displaymath}
    \hat{\mu}_{\max{}}(F_{\EE}(\FF_i)\leq \hat{\mu}_{\EE}(\FF_i) +\widetilde{m}\deg{(\OO_{X_{\overline{s}}}(1))}=\colon\mu_0
  \end{displaymath}
 From the previous point we have that $\FF_i$ is a bounded family.
\end{proof}
\begin{rem}
This result improves  boundedness for sheaves on curves  in \ref{prop:sheaves-on-curves-bounded} since we have no normality assumption on $\XX$. In particular we can study sheaves on nodal curves  as in \cite{MR2450211}. 
This result improves the boundedness result for parabolic bundles. Indeed the equivalence between parabolic sheaves and sheaves on stacks of roots is proven only for locally free-sheaves \cite{borne-2006} and \cite{borne-2007}, and we cannot use the result in \cite{MR1162674} to prove  boundedness of semistable sheaves. This generalizes also the result in \cite{MR2309155} about gerbes since we have no assumptions on the banding  of the gerbe, we just need the gerbe to be a projective stack.
\end{rem}

The second result of Langer that we need is estimate in \cite[Cor 3.4]{MR2085175}. In characteristic zero it is possible to bound the number of global sections of a family of semistable sheaves with fixed Hilbert polynomial  restricted to a general enough hyperplane or intersection of hyperplanes. This is known as Le Potier Simpson theorem \cite[3.3.1]{MR1450870}. In positive characteristic it is known that it is not possible to reproduce such a result (for a counterexample \cite[Ex. 3.1]{MR2085175}). However Langer was able to prove that it is possible to produce a bound for the number of global sections.
\begin{thm}[{\cite[Cor 3.4]{MR2085175}}]
 Let $X$ be a projective scheme with a very ample invertible sheaf $\OO_X(1)$. For any pure sheaf $F$ of dimension $d$ we have:
 \begin{displaymath}
   h^0(X,F)\leq \begin{cases} & r\binom{\hat{\mu}_{\max{}}(F)+r^2+f(r)+\frac{d-1}{2}}{d}\quad  \text{if $\hat{\mu}_{\max{}}(F)\geq \frac{d+1}{2}-r^2$} \\
& 0   \quad\quad\quad\quad\quad\quad\quad\quad\quad\;\;  \text{if $\hat{\mu}_{\max{}}(F)< \frac{d+1}{2}-r^2$} \\
\end{cases}
 \end{displaymath}
where $r$ is the multiplicity of $F$ and $f(r)=-1+\sum_{i=1}^r\frac{1}{i}$ is an approximation of $\ln{r}$. 
\end{thm}
From this we deduce a stacky version for semistable sheaves:
\begin{cor}
  Let $\XX$ be a projective stack with polarization $\EE,\OO_X(1)$. For any pure semistable sheaf $\FF$ on $\XX$ of dimension $d$ we have:
 \begin{equation}\label{eq:stima-molto-utile}
   h^0(\XX,\FF\otimes\EE^\vee)\leq \begin{cases} & r\binom{\hat{\mu}_{\EE}(\FF)+\widetilde{m}\deg(\OO_X(1))+r^2+f(r)+\frac{d-1}{2}}{d}\quad  \text{if $\hat{\mu}_{\max{}}(F_{\EE}(\FF))\geq \frac{d+1}{2}-r^2$} \\
& 0   \quad\quad\quad\quad\quad\quad\quad\quad\quad\quad\quad\quad\;\;  \text{if $\hat{\mu}_{\max{}}(F_\EE(\FF))< \frac{d+1}{2}-r^2$} \\
\end{cases}
 \end{equation}
where $r$ is the multiplicity of $F_{\EE}(\FF)$, the integer $\widetilde{m}$ is like in Proposition \ref{prop:stima-max-destab}.
\end{cor}
This corollary is a replacement for \cite[Cor 3.3.8]{MR1450870} and it will play a fundamental role in the study of the GIT quotient producing the moduli scheme of semistable sheaves.

\section{The stack of semistable sheaves}
In the previous section we have collected enough machinery to write the algebraic stack of semistable-sheaves on a projective Deligne-Mumford stack $\XX$ as a global quotient. Let $\pi\colon\XX\to X$ be a projective stack with moduli scheme $X$ over an algebraically closed field $k$. Fix a polarization  $\EE,\OO_X(1)$ and a polynomial $P$ with integer coefficients and degree  $d\leq\dim{X}$. Fix an integer $m$ such that every semistable sheaf on $\XX$ is $m$-regular, which exists since semistable sheaves are bounded and according to Kleiman criterion there is $m$ such that every sheaf in a bounded family is $m$-regular. Let $N$  be the positive integer $h^0(X,F_{\EE}(\FF)(m))=P_{\EE}(\FF,m)=P(m)$ and denote with $V$ the linear space $k^{\oplus N}$.
\begin{thm}\label{prop:stack-dei-moduli}
  There is an open subscheme $\QQ$ in $\Quot_{\XX/k}(V\otimes\EE\otimes\pi^\ast\OO_X(-m),P)$ such that $\mathcal{S}(\EE,\OO_X(1),P)$ the algebraic stack of semistable sheaves on $\XX$ with Hilbert polynomial $P$ is the global quotient:
  \begin{displaymath}
    \mathcal{S}(\EE,\OO_X(1),P)=[\QQ/GL_{N,k}]\subseteq [\Quot_{\XX/k}(V\otimes\EE\otimes\pi^\ast\OO_X(-m),P)/GL_{N,k}]
  \end{displaymath}
\end{thm}
\begin{proof}
First of all we consider the set of pairs $\FF,\rho$ where $\FF$ is a semistable sheaf and $\rho$ is an isomorphism $\rho\colon V\to H^0(X,F_{\EE}(\FF)(m))$. Every sheaf $F_{\EE}(\FF)$ with isomorphism $\rho$ can be written in a unique way as a quotient $\OO_{X}(-m)^{\oplus N}\xrightarrow{\widetilde{\rho}} F_{\EE}(\FF)$ such that the map induced by $\widetilde{\rho}$ between $V$ and $H^0(X,F_{\EE}(\FF)(m))$ is $\rho$. In particular the quotient $\widetilde{\rho}$ is unique and it is the following composition:
\begin{displaymath}
  V(-m)\xrightarrow{\rho\otimes\Id} H^0(X,F_{\EE}(\FF)(m))\otimes \OO_X(-m)\xrightarrow{\ev} F_{\EE}(\FF)
\end{displaymath}
where the second map is the evaluation and we have denoted with $V(-m)$ the tensor product $V\otimes\OO_X(-m)$ .
Now consider a quotient $G_{\EE}(V(-m))\xrightarrow{\widetilde{\sigma}}\FF$; apply the functor $F_{\EE}$ and compose on the left with the transformation $\varphi_{\EE}(V(-m))$:
\begin{displaymath}
\xymatrix{
  V(-m)\ar[rrr]^-{\varphi_{\EE}(V(-m))} &&& F_{\EE}\circ G_{\EE}(V(-m))\ar[rr]^-{F_{\EE}(\widetilde{\sigma})} && F_{\EE}(\FF) \\
}
\end{displaymath}
This induces a morphism in cohomology $V\xrightarrow{\sigma}H^0(X,F_{\EE}(\FF)(m))$. The subset of points of $\Quot_{\XX/k}(V\otimes\EE\otimes\pi^\ast\OO_X(-m),P)$ such that this procedure induces an isomorphism is an open (see the proof of \ref{prop:alg-stack-coherent}) and we denote it with $\QQ$. We claim that there is a bijection between points of $\QQ$ and couples $\FF,\rho$.
Given a couple $\FF,\rho$ we first associate to it a quotient $V(-m)\xrightarrow{\widetilde{\rho}} F_{\EE}(\FF)$ where $\widetilde{\rho}=\ev\circ (\rho\otimes\Id)$, then we produce the quotient $\theta_{\EE}(\FF)\circ G_{\EE}(\widetilde{\rho})\colon G_{\EE}(V(-m))\to \FF$. This quotient is clearly the same as $\widetilde{\ev}(N,m)\circ(\rho\otimes\Id)$ (defined in \ref{cor:alg-stack-of-coh-sh}). Given a quotient $\widetilde{\sigma}$ we associate to it the isomorphism induced in cohomology by $F_{\EE}(\widetilde{\sigma})\circ\varphi_{\EE}(V(-m))$. We show that these two maps of sets are inverse to each other. First we consider the following composition:
\begin{displaymath}
  \xymatrix{
    G_{\EE}(V(-m))\ar@{=}[rrrd]\ar[rrr]^-{G_{\EE}(\varphi_{\EE}(V(-m)))} &&& G_{\EE}\circ F_{\EE}\circ G_{\EE}(V(-m)) \ar[d]^-{\theta_{\EE}(G_{\EE}(V(-m)))}\ar[rr]^-{G_{\EE}\circ F_{\EE}(\widetilde{\sigma})} && G_{\EE}\circ F_{\EE}(\FF) \ar[r]^-{\theta_{\EE}(\FF)} & \FF \\
  &&& G_{\EE}(V(-m))\ar[rrru]_-{\widetilde{\sigma}} &&& \\
  }
\end{displaymath}
where the first triangle commutes because of lemma \ref{lem:idenity-theta-phi} and the second triangle commutes because $\theta_\EE$ is a natural transformation.
Now we consider the composition the other way round:
\begin{displaymath}
  \xymatrix{
    V(-m)\ar[rrrrrd]_-{\rho}\ar[rrr]^-{\varphi_{\EE}(V(-m))} &&& F_{\EE}\circ G_{\EE}(V(-m)) \ar[rr]^-{F_{\EE}\circ G_{\EE}(\rho)} && F_{\EE}\circ G_{\EE}\circ F_{\EE}(\FF) \ar[rr]^-{F_{\EE}(\theta_{\EE}(\FF))} && F_{\EE}(\FF) \\
  &&&&& F_{\EE}(\FF)\ar@{=}[rru]\ar[u]^-{\varphi_{\EE}(F_{\EE}(\FF))} && \\
  }
\end{displaymath}
The first triangle commutes because $\phi_\EE$ is a natural transformation and the second because of \ref{lem:idenity-theta-phi}.
\end{proof}
 It is clear that given a quotient $G_{\EE}(V(-m))=V\otimes\EE\otimes\pi^{\ast}\OO_X(-m)\xrightarrow{\widetilde{\sigma}}\FF$ there is an action of the group $\Aut{\EE}$ that is composing on the left with an automorphism $\alpha$ of $\EE$. A priori it is not clear if this is an action of $\Aut{\EE}$ on $\QQ$. Eventually it is.
\begin{lem}
  The open subscheme $\QQ$ is invariant by the action of $\Aut{\EE}$ on $\Quot_{\XX/k}(V\otimes\EE\otimes\pi^\ast\OO_X(-m),P)$.
\end{lem}
\begin{proof}
  We have just to prove that given a quotient $\widetilde{\sigma}$ inducing an isomorphism in cohomology and given $\alpha\in\Aut{\EE}$ the composition $\widetilde{\sigma}\circ\alpha$ induces an isomorphism in cohomology $\sigma_{\alpha}$. The morphism $\sigma_{\alpha}$ is induced by the following composition:
  \begin{displaymath}
    \xymatrix{
      V(-m)\ar[rrr]^-{\varphi_{\EE}(V(-m))} &&& F_{\EE}\circ G_{\EE}(V(-m))\ar[rr]^-{F_{\EE}(\alpha\otimes\Id)} && F_{\EE}\circ G_{\EE}(V(-m)) \ar[rr]^-{F_{\EE}(\widetilde{\sigma})} && F_{\EE}(F) \\
    }
  \end{displaymath}
The composition of the first two arrows on the left is the same as a map $ V(-m)\rightarrow F_{\EE}\circ G_{\EE}(V(-m))$ selecting the automorphism $\pi_\ast\alpha$ in $\pi_\ast\END_{\OO_\XX}(\EE)$. A simple computation in local coordinates shows that $\sigma_{\alpha}$ is the same as the composition of $\sigma$ and the endomorphism of $H^0(X,F_{\EE}(F)(m))$ given by the action of $\alpha$ on $\EE^\vee$. Since this is actually an automorphism we obtain that $\sigma_{\alpha}$ is an isomorphism.
\end{proof}
\begin{rem}[Psychological remark]
Despite the action of $\Aut{\EE}$ on $\QQ$ being well defined, there is no reason to quotient it in order to obtain the moduli stack of semistable sheaves, even if it could seem natural at a first sight. Actually we don't know what kind of moduli problem  could represent a quotient by this group action, for sure a much harder one from the GIT viewpoint.
\end{rem}
\begin{rem}
  As in the case of sheaves on schemes we can consider the subscheme $\QQ^s\subset\QQ$ of stable sheaves which is an open subscheme \ref{prop:semistable-open}.
\end{rem}
\begin{rem}
  The multiplicative group $\mathbb{G}_{m,k}$ is contained in the stabilizer of every point of $[\QQ/\GL_{N,k}]$ so that it is natural to consider the rigidification $[\QQ/\GL_{N,k}]\fs\Gm$ which is the stack $[\QQ/\PGL_{N,k}]$ where the action is induced by the exact sequence:
  \begin{displaymath}
    1\to \mathbb{G}_{m,k} \to \GL_{N,k} \to \PGL_{N,k} \to 1
  \end{displaymath}
In particular $[\QQ/\GL_{N,k}]$ is a $\mathbb{G}_{m,k}$-gerbe on $[\QQ/\PGL_{N,k}]$. In the same way we could consider the global quotient $[\QQ/\SL_{N,k}]$ where the action is induced by the inclusion $\SL_{N,k}\to \GL_{N,k}$; again we have that $[\QQ/SL_{N,k}]\fs\mu_{N,k}$ is isomorphic to $[\QQ/\PGL_{N,k}]$ and $[\QQ/\SL_{N,k}]$ is a $\mu_{N,k}$-gerbe on it. For this reason in the  GIT study of this global quotient it is equivalent to consider one of these three quotients.
\end{rem}
\begin{thm}
  The algebraic stack $[\QQ/GL_{N,k}]$ is an Artin stack of finite type on $k$.
\end{thm}
\begin{proof}
  Since the group scheme $GL_{N,k}$ is smooth separated and of finite presentation on $k$, the global quotient is an Artin stack and $\QQ$ is a smooth atlas \cite[4.6.1]{LMBca}. Since $\QQ\to\Spec{k}$ is a finite type morphism the stack itself is of finite type on $k$.
\end{proof}

\section{The moduli scheme of semistable sheaves}
The aim of this section is to prove that the global quotient $\mathcal{S}(\EE,\OO_X(1),P)$ of semistable sheaves  exists as a GIT quotient and it is a projective  scheme. As in the case of sheaves on schemes the GIT quotient is the moduli scheme for the stack $\mathcal{S}(\EE,\OO_X(1),P)$ only if there are no strictly semistable sheaves, otherwise it just parametrizes classes of $S$-equivalent sheaves. To prove this result we use the standard machinery of Simpson \cite{MR1307297} to compare semistability of sheaves to semistability for the GIT quotient $\QQ/SL_{N,k}$. The first section is devoted to the definition of the GIT problem, while the second contains the results.

\subsection{Closed embedding of  $\Quot_{\XX/S}$ in the projective space}
Let $\rho\colon \XX\to S$ a family of projective Deligne-Mumford stacks with moduli scheme $\XX\xrightarrow{\pi} X\xrightarrow{p} S$ and $\HH$ a locally free sheaf on $\XX$ and $P$ a polynomial with integral coefficients. In this subsection we write explicitly a class of very ample line bundles on $\Quot_{\XX/S}(\HH,P)$.
\begin{lem}\label{lem:dummy-repfunctor}
  Let $G$ and $H$ be two representable functors, represented by $\overline{G}$ and $\overline{H}$ respectively. Let $\UU_G$ and $\UU_H$ the two universal objects. Given $\iota\colon G\to H$ a natural transformation relatively represented by $\overline{\iota}$ the following square is cartesian:
  \begin{displaymath}
    \xymatrix{
      G \ar[r]^{\iota} & H \\
      \overline{G}\ar[u]^{\UU_G}\ar[r]^{\overline{\iota}} & \overline{H}\ar[u]_{\UU_H} \\
    }
  \end{displaymath}
and in particular $\iota(\UU_G)=\overline{\iota}^\ast\UU_{H}$.
\end{lem}
\begin{proof}
Almost the definition of representable functor.
\end{proof}
 Let $\widetilde{\UU}$ be the universal quotient sheaf of $\widetilde{\QQ}\coloneqq\Quot_{\XX/S}(\HH,P)$. Let $\UU$ be the universal quotient sheaf of $\QQ\coloneqq\Quot_{X/S}(F_{\EE}(\HH),P)$. First we recall that there is a class of closed embeddings  $j_l\colon\Quot_{X/S}(F_{\EE}(\HH),P)\to\Grass_S(p_\ast F_{\EE}(\HH)(l),P(l))$ where $l$ is a big enough integer, and it is given by the class  of very ample line bundles $\det{({p_{\QQ}}_\ast\UU(l))}$ where $p_{\QQ}\colon X_{\QQ}\to \QQ$ is the pull back of $p$ through the Pl\"ucker embedding $k_l$ in $\pp\coloneqq\Proj(\Lambda^{P(l)}W)$ of the grassmanian and the closed embedding $j_l$. The locally free sheaf $W$ is the universal quotient of $\Grass_S(p_\ast F_{\EE}(\HH)(l),P(l))$. 
According to \cite[Prop 6.2]{MR2007396} in its modified version \ref{prop:closed-embedding-modified-vers} there is a closed embedding $\iota\colon\widetilde{\QQ}\to\QQ$ representing the natural transformation $F_{\EE}$. We will denote with $\widetilde{\iota}$ the pull-back morphism $X_{\widetilde{\QQ}}\to X_{\QQ}$.
\begin{prop}\label{prop:very-ample-lb}
 Let $\pi_{\widetilde{\QQ}},p_{\widetilde{\QQ}}$ be the pull-back of $\pi,p$ through the Pl\"ucker embedding the closed embedding $j_l$ and $\iota$. The class of invertible sheaves $L_l\coloneqq\det{({p_{\widetilde{\QQ}}}_{\ast}(F_{\EE}(\widetilde{\UU})(l)))}$ is very ample for $l$ big enough as before.  
\end{prop}
\begin{proof}
  According to the previous recall about the $\Quot$ scheme of sheaves on a projective scheme we have that:
  \begin{displaymath}
    j_l^\ast k_l^\ast\OO(1)=\det{({p_{\QQ}}_\ast\UU(l))}
  \end{displaymath}
Moreover lemma \ref{lem:dummy-repfunctor} implies that $F_{\EE}(\widetilde{\UU})=\widetilde{\iota}^\ast\UU$. We observe that  $\iota^\ast {p_{\QQ}}_\ast \UU(l) \cong {p_{\widetilde{\QQ}}}_\ast\widetilde{\iota}^\ast\UU(l)$. This is an application of cohomology and base-change; in general we have this isomorphism for an arbitrary base change and a flat family $\GG$ on $X_{\QQ}$ whenever $H^1(\GG_q)=0$ for every closed point $q$ in $\QQ$ (this can be easily derived from cohomology and base change and for an explicit reference see \cite[0.5]{MR1304906}).  In this particular case the fiber $\UU_q(l)$ is just a sheaf of the $\Quot$-functor and $l$ in the hypothesis is big enough so that every quotient sheaf is $l$-regular. This implies $H^1(\UU_q(\overline{l}-1))=0$ for every $q$ and $\overline{l}\geq l$ and eventually the desired relation:
\begin{displaymath}
   \iota^\ast j_l^\ast k_l^\ast\OO(1)=\det{({p_{\widetilde{\QQ}}}_{\ast}(F_{\EE}(\widetilde{\UU})(l)))}
\end{displaymath}
\end{proof}

\begin{lem}\label{lem:equi-invertible}
  The class of very ample invertible sheaves $L_l$ of proposition \ref{prop:very-ample-lb} carries a natural $\GL_{N,k}$-linearization. 
\end{lem}
\begin{proof}
  The universal sheaf $\widetilde{\mathcal{U}}$ carries a natural $\GL_{N,k}$-linearization coming from the universal automorphism of $\GL_{N,k}$ (see \cite[pg. 90]{MR1450870} for the details). This linearization induces a linearization of $L_l$ because the formation of $L_l$ commutes with arbitrary base changes. This follows from \ref{prop:cohom-base-change-spcoarse}, the compatibility of $\EE$ with base change and the criterion in \ref{prop:very-ample-lb} .
\end{proof}
\begin{rem}
  The invertible sheaf $L_l$ together with the $GL_{N,k}$-linearization of the previous lemma defines an invertible sheaf on the global quotient  $[\QQ/\GL_{N,k}]$. We will denote this sheaf with $\mathcal{L}_l$.  
\end{rem}
With lemma \ref{lem:equi-invertible} we can define a notion of GIT semistable (stable) points in the projective scheme $\overline{\QQ}$, the closure of $\QQ$, with respect to the invertible sheaf $L_l$ and the action of $\GL_{N,k}$. Since a $\GL_{N,k}$-linearization induces an $\SL_{N,k}$-linearization we can consider the GIT problem with respect to the action of $\SL_{N,k}$. We will denote the subscheme of GIT semistable points with $\overline{\QQ}^{ss}(L_l)$ and the subscheme of GIT stable points with $\overline{\QQ}^{s}(L_l)$.  

\subsection{Two technical lemmas of Le Potier}
We collect here two technical results of Le Potier \cite{potier-ldmds1992} which are useful to compare the semistability of sheaves on $\XX$ and the GIT stability that we will study in the next section.
The first statement is a tool to relate the notion of semistability to the number of global sections of subsheaves or quotient sheaves.
\begin{thm}\label{thm:stima-utile-dilepotier}
  Let $\FF$ be a pure dimensional coherent sheaf on a projective stack $\XX$ with modified Hilbert polynomial $P_{\EE}(\FF)=P$, multiplicity $r$ and reduced Hilbert polynomial $p$ and let $m$ be a sufficiently large integer. The following conditions are equivalent:
  \begin{enumerate}
  \item $\FF$ is semistable (stable) \label{item:11} 
  \item $r\cdot p(m)\leq h^0(F_{\EE}(\FF)(m))$ and for every subsheaf $\FF'\subset \FF$ with multiplicity $r',\, 0<r'<r$ we have $h^0(F_{\EE}(\FF')(m))\leq r'\cdot p(m);\quad (<)$ \label{item:13}
  \item for every quotient sheaf $\FF\to \FF''$ of multiplicity $r'',0<r''<r$ we have $r''\cdot p(m)\leq h^0(F_{\EE}(\FF'')(m));\quad (<)$ \label{item:12}
  \end{enumerate}
moreover equality holds in \ref{item:13} and \ref{item:12} if and only if $\FF'$ and $\FF''$ respectively are semistable  and it holds for every $m$.  
\end{thm}
\begin{proof}
  The proof of this is just the same as in \cite[4.4.1]{MR1450870}. This is true since the proof relies only on the Grothendieck lemma \ref{lem:grothendieck-lemma}, the Kleiman criterion \ref{thm:kleiman-criterion} and above all  Langer's inequality (\ref{eq:stima-molto-utile}) (replacing \cite[Cor 3.3.8]{MR1450870}).  
\end{proof}
\begin{rem}
If $m$ is chosen such that every semistable sheaf on $\XX$ is $m$-regular then it is an $m$ big enough in the sense of the previous theorem.
\end{rem}

Before stating the second lemma we recall the definition of dual sheaf of a pure dimensional sheaf.
\begin{defn}
  Let $\XX$ be a smooth projective Deligne-Mumford stack and $\FF$ a coherent  sheaf of codimension $c$ on $\XX$. We define the dual of $\FF$ to be the coherent sheaf $\FF^D=\EXT_\XX^c(\FF,\omega_\XX)$. 
\end{defn}
If the stack $\XX$ is non smooth we could think of studying the dual choosing an embedding $i\colon\XX\to\mathcal{P}$ in a smooth projective ambient $\mathcal{P}$ and using $(i_\ast\FF)^D=\EXT^e_{\mathcal{P}}(i_\ast\FF,\omega_{\mathcal{P}})$ where $e$ is now the codimension of $\FF$ in $\mathcal{P}$. This is reasonable because of the following remark in \cite{MR1450870}:
\begin{lem}
  Let $\XX$ be a smooth projective stack and $i\colon\XX\to\mathcal{P}$ a closed embedding in a smooth projective stack $\mathcal{P}$. Let $c$ be the codimension of $\FF$ in $\XX$ and $e$ the codimension of $\FF$ in $\mathcal{P}$. We have the following isomorphism:
  \begin{displaymath}
    i_\ast\EXT_\XX^c(\FF,\omega_\XX)\cong\EXT_{\mathcal{P}}^e(i_\ast\FF,\omega_{\mathcal{P}})
  \end{displaymath}
\end{lem}
\begin{proof}
  It follows from the fact that $i_\ast$ is exact and an application of Grothendieck duality to the closed immersion $i$ \cite[Cor 1.41]{groth-duality-stack}.
\end{proof}
This lemma implies that our new definition of duality doesn't depend on the smooth ambient space.
\begin{lem}\label{lem:biduale-purity}
  Let $\FF$ be a  coherent sheaf on a projective stack $\XX$. There is a natural morphism:
  \begin{displaymath}
    \rho_{\FF}\colon \FF\to \FF^{DD}
  \end{displaymath}
and it is injective if and only if $\FF$ is pure dimensional.
\end{lem}
\begin{proof}
  Using Serre duality \cite[Cor 1.41]{groth-duality-stack} we can rewrite  \cite[Prop 1.1.6]{MR1450870}, \cite[Lem 1.1.8]{MR1450870} and \cite[Prop 1.1.10]{MR1450870} tensoring occasionally with a generating sheaf to achieve some vanishing. 
\end{proof}

With this machinery we can write the stacky version of a lemma of Le Potier \cite[4.4.2]{MR1450870}. Using this lemma we can  deal with possibly non pure dimensional sheaves that can be found in the closure $\overline{\QQ}$.
\begin{lem}\label{lem:lepotier-deform-sheaves}
  Let $\FF$ be a coherent sheaf on $\XX$ that can be deformed to a pure sheaf of the same dimension $d$. There is a pure $d$-dimensional sheaf $\GG$ on $\XX$ with a map $\FF\to\GG$ such that the kernel is $T_{d-1}(\FF)$  and $P_{\EE}(\FF)=P_{\EE}(\GG)$.
\end{lem}
\begin{proof}
  Since $\FF$ can be deformed to a pure sheaf there is a smooth connected curve $C\to\Spec{k}$ and a family $\mathfrak{F}$ of sheaves on $\XX_C$ flat on $C$ such that there is a point $0\in C$ satisfying $\mathfrak{F}_0\cong\FF$ and for every point $t\neq 0$ the fibers $\mathfrak{F}_t$ are pure dimensional. Using the technique in \cite[4.4.2]{MR1450870} together with Lemma \ref{lem:biduale-purity} we can find a flat family $\mathfrak{G}$ on $\XX_C$  and a map $0\to\mathfrak{F}\to\mathfrak{G}$ which induces isomorphisms between the fibers for every $t\neq 0$, and for the special fiber $t=0$ the map $\mathfrak{F}_0\to\mathfrak{G}_0$ has kernel the torsion of $T_{d-1}(\FF)$. The two sheaves have the same Hilbert polynomial $P=P(\FF)=P(\mathfrak{G}_0)$ since the family $\GG$ is flat on a connected scheme. We observe that the two families $F_{\EE_C}(\mathfrak{F})$ and $F_{\EE_C}(\mathfrak{G})$ are again flat over $C$ since $\EE$ is locally free and using corollary \ref{cor:tame-stack-2}~\ref{item:3}. Moreover picking a fiber of a family and the functor $F_{\EE}$ commute because of proposition \ref{prop:cohom-base-change-spcoarse} and the compatibility of $\EE$ with base change. We have a morphism $F_{\EE}(\FF)\to F_{\EE}(\GG_0)$ and the kernel is $F_{\EE}(T_{d-1}(\FF))$ which is the same as $T_{d-1}(F_{\EE}(\FF))$ because of Corollary \ref{cor:torsion-filtration-preserved}. Eventually $P(F_{\EE}(\FF))=P(F_{\EE}(\mathfrak{G}_0))$ since $F_{\EE}$ preserves flatness.  
\end{proof}

\subsection{GIT computations}
The GIT problem is well posed and we can compute the Hilbert Mumford criterion for points in $\overline{\QQ}$ and establish a numerical criterion for stability, which is actually the standard condition we have for stability of points in a grassmanian. 

Suppose we are given $\rho\colon V\otimes\EE\otimes\pi^\ast\OO_X(-m)\to \FF$ a closed point in $\overline{\QQ}$.  Let $\lambda\colon \mathbb{G}_{m,k}\to \SL_{N,k}$ be a group morphism. This representation splits $V$ into weight subspaces $V_n$ such that $\lambda(t)\cdot v=t^n\cdot v$ for every $n$, every $t\in\mathbb{G}_{m,k}$ and every $v\in V_n$. We construct an ascending filtration $V_{\leq n}=\oplus_{i\leq n} V_n$ of $V$. In general, given a subspace $W\subseteq V$, it induces a subsheaf of $\FF$ which is the image sheaf $\rho(W\otimes\EE\otimes\pi^\ast\OO_X(-m))$. In this case we can produce a filtration of $\FF$ with the subsheaves $\FF_{\leq n}=\rho(V_{\leq n}\otimes\EE\otimes\pi^\ast\OO_X(-m))$. We have an induced surjection $\rho_n\colon V_n\otimes\EE\otimes\pi^\ast\OO_X(-m)\to \FF_{\leq n}/F_{\leq n-1}\eqqcolon \FF_n$. Taking the sum over all the weights we obtain a new quotient sheaf:
\begin{displaymath}
  \overline{\rho}\coloneqq V\otimes\EE\otimes\pi^\ast\OO_X(-m)\to \bigoplus_{n} \FF_n\eqqcolon\overline{\FF}
\end{displaymath}
It is a very standard result that:
\begin{lem}
  The quotient $[\overline{\rho}]$ is the limit $\lim_{t\to 0} \lambda(t)\cdot [\rho]$ in the sense of the Hilbert Mumford criterion.
\end{lem}
\begin{proof}
  It is just the same proof as in the case of sheaves on a projective scheme. See for instance \cite[4.4.3]{MR1450870}.
\end{proof}
\begin{lem}
  The action of $\mathbb{G}_{m,k}$ via the representation $\lambda$ on the fiber of the invertible sheaf $L_l$ at the point $[\overline{\rho}]$ has weight 
  \begin{displaymath}
    \sum_{n}n\cdot P_{\EE}(\FF_n,l)
  \end{displaymath}
\end{lem}
\begin{proof}
  The action of $\mathbb{G}_{m,k}$ on $V_n$ induces an action on $\FF_n$ which is again multiplication by $t^n$ on the sections. The $k$-linear space $H^0(\FF_n\otimes\EE^\vee\otimes\pi^\ast\OO_{X}(-l))$ inherits the same action. We recall that $l$ is chosen big enough so that the $\Quot$-scheme is embedded in the grassmanian; in particular this means that every quotient of $V\otimes\EE\otimes\pi^\ast\OO_X(-m)$ with Hilbert polynomial $P=P_{\EE}(\FF)$ is $l$-regular. This means that $\overline{\FF}$ is $l$-regular and $P(l)=P_{\EE}(\overline{\FF},l)=h^0(X,F_{\EE}(\overline{\FF})(l))$. Now we take the fiber $[\overline{\rho}]$ in $L_l$:
  \begin{displaymath}
    L_l\vert_{[\overline{\rho}]}=\det{(H^0(X,F_{\EE}(\overline{\FF})(l)))}=\bigotimes_n\det{(H^0(X,F_{\EE}(\FF_n)(l)))}=\bigotimes_n\det{(H^\bullet(X,F_{\EE}(\FF_n)(l)))}
  \end{displaymath}
The last equality follows from the fact that $H^i(X,F_{\EE}(\overline{\FF})(l))=\bigoplus_{n}H^i(X,F_{\EE}(\FF_n)(l))$ for every $i\geq 0$ and it vanishes for $i>0$. This proves also that $h^0(X,F_{\EE}(\FF_n)(l))=P_{\EE}(\FF_n,l)$ so that the weight of the action of $\mathbb{G}_{m,k}$ on $L_{l}\vert_{[\overline{\rho}]}$ is the one we have stated. 
\end{proof}
An application of the Hilbert Mumford criterion translates in the following very standard lemma:
\begin{lem}
 A closed point $\rho\colon V\otimes\EE\otimes\pi^\ast\OO_X(-m)\to \FF$ is semistable (stable) if and only if for every non trivial subspace $V'\subset V$ the induced subsheaf $\FF'\subset\FF$ satisfies:
 \begin{equation}\label{eq:stima-con-l}
   \dim{(V)}\cdot P_{\EE}(\FF',l)\geq\dim{(V')}\cdot P_{\EE}(\FF,l);\quad (>)
 \end{equation}
\end{lem}
\begin{rem}
As we have already seen, given a quotient $\rho\colon V\otimes\EE\otimes\pi^\ast\OO_X(-m)\to \FF$ and a linear subspace $V'\subset V$ we can associate to it a subsheaf $\FF'\subset\FF$ which is $\rho(V'\otimes\EE\otimes\pi^\ast\OO_X(-m))$. Given a subsheaf $\FF'\subset\FF$ we can associate to it a subspace $V'$ of $V$. Consider the following cartesian square:
\begin{displaymath}
  \xymatrix{
    V \ar[r] & V\otimes H^0(X,F_{\EE}(\EE)(m))\ar[r] & H^0(X,F_{\EE}(\FF)(m)) \\
    V\cap H^0(X,F_{\EE}(\FF')(m)) \ar@{>->}[u] \ar[rr] && H^0(X,F_{\EE}(\FF')(m)) \ar@{>->}[u] \\
}
\end{displaymath}
where the first map on the top is induced by $\varphi_{\EE}(V(-m))$ and the second by $F_{\EE}(\rho)$. The linear space we associate to $\FF'$ is $ V\cap H^0(X,F_{\EE}(\FF')(m))$. If we take a linear subspace $V'$ and associate to it a subsheaf $\FF'$ and then we associate to $\FF'$ a linear space $V''$ with this procedure we obtain an inclusion $\xymatrix{V'' \ar@{>->}[r] & V'\\}$. On the contrary if we start from a subsheaf $\FF'$, associate to it a linear space $V'=V\cap H^0(X,F_{\EE}(\FF')(m))$ and we use $V'$ to generate a subsheaf $\FF''$ we obtain again a natural injection of sheaves $\xymatrix{\FF'' \ar@{>->}[r] & \FF' \\}$.
\end{rem}
From this observation follows the lemma:
\begin{lem}
  Let $\rho\colon V\otimes\EE\otimes\pi^\ast\OO_X(-m)\to\FF$ be a closed point in $\overline{\QQ}$. It is GIT semistable (stable) if and only if for any coherent subsheaf $\FF'$ of $\FF$ and denoted $V'=V\cap H^0(X,F_{\EE}(\FF')(l))$ we have the following inequality:
  \begin{equation}\label{eq:stima-senza-l}
    \dim{(V)}\cdot P_{\EE}(\FF') \geq \dim{(V')}\cdot P_{\EE}(\FF)
  \end{equation}
\end{lem} 
\begin{proof}
  We first observe that, fixed the point $[\rho]$, the family of subsheaves generated by a linear subspace of $V$ is bounded  because exact sequences of linear spaces split so that every subsheaf generated by a subspace has the same regularity as $\FF$. This implies also that this family has a finite number of Hilbert polynomials. If the number of polynomials is finite the inequality (\ref{eq:stima-senza-l}) is equivalent to (\ref{eq:stima-con-l}). The rest follows from the previous remark.
\end{proof}
\begin{lem}\label{lem:git-semist-injectivity}
Let $\rho\colon V\otimes\EE\otimes\pi^\ast\OO_X(-m)\to\FF$ be a closed point in $\overline{\QQ}$ which is GIT semistable then the induced morphism $V\to H^0(F_{\EE}(\FF)(m))$ is injective. 
\end{lem}
\begin{proof}
  Take the kernel $K$ of the morphism $V\to H^0(F_{\EE}(\FF)(m))$. It generates the zero subsheaf in $\FF$. Even if the zero sheaf has no global sections we have $0\cap V=K$ which is not zero. Inequality (\ref{eq:stima-senza-l}) reads $0\geq \dim{(K)}\cdot P_{\EE}(F) $ which is impossible unless $\dim{(K)}=0$.    
\end{proof}
\begin{prop}
 Let $m$ be a large enough integer (possibly larger than  the one we have used so far) and $l$ large enough in the usual sense.  The scheme $\QQ$ of semistable sheaves is equal to $\overline{\QQ}^{ss}(L_l)$ the scheme of GIT semistable points in $\overline{\QQ}$ with respect to $L_l$; moreover the stable points coincides $\QQ^s=\overline{\QQ}^{s}(L_l)$.
\end{prop}
\begin{proof}
The proof is just like the one in {\cite[4.3.3]{MR1450870}} with obvious modifications .
\end{proof}
To prove next theorem, which completes the GIT study, we need a result of semicontinuity for the $\Hom$ functor on a family of projective stacks. Since we were not able to retrieve this result from an analogous one on the moduli scheme of $\XX$, we prove it here and not in   the first section.
\begin{lem}\label{lem:semicont-for-hom}
  Let $p\colon\XX\to S$ be a family of projective stacks, let $\GG$ be a coherent sheaf on $\XX$ flat on $S$ and $\FF$ a coherent sheaf on $\XX$. Let $s$ be a point of $S$, the function $s\mapsto \hom_{\XX_s}(\FF_s,\GG_s)$ is upper semicontinuous. 
\end{lem}
\begin{proof}
  We prove this using Grothendieck original approach to the problem and we keep most of the notations in section $III.12$ of \cite{Hag}. Since the problem is local in the target we can assume that $S=\Spec{A}$ is affine. Let $T^0$ be the functor mapping an $A$-module $M$ to $\Hom_{\XX}(\FF,\GG\otimes_{A}M)$. Since $\XX$ is projective we can take a locally free resolution $\EE^{\oplus N_1}\otimes\pi^\ast\OO_X(-m_1)\to \EE^{\oplus N_2}\otimes\pi^\ast\OO_X(-m_2)\to \FF\to 0$ where $m_1,m_2$ are positive and big enough integers. We  produce the exact sequence:
  \begin{displaymath}
    \xymatrix{
      0\ar[r] & \Hom_{\XX}(\FF,\GG\otimes_A M) \ar[r] & H^0(X,F_{\EE}(\GG)^{\oplus N_2}(m_2))\otimes_A M \ar[r] & \ldots  \\
 \ldots \ar[r] & H^0(X,F_{\EE}(\GG)^{\oplus N_1}(m_1))\otimes_A M \\
}
  \end{displaymath}
The coherent sheaf $F_{\EE}(\GG)^{\oplus N_i}(m_i)$ is $A$-flat for $i=1,2$ and choosing $m_1,m_2$ even larger we can assume (Serre vanishing plus semicontinuity for cohomology) that \\ $H^1(X_y,F_{\EE}(\GG)^{\oplus N_i}(m_i)\otimes_A k(y))=0$ for every point $y$ in $S$. Denote with $q\colon X\to S$ the morphism from the moduli scheme to $S$. Using \cite[pag 19]{MR1304906} we can conclude that $p_\ast F_{\EE}(\GG)^{\oplus N_i}(m_i)$ is locally free, so that the module $L_i:=H^0(X,F_{\EE}(\GG)^{N_i}(m_i))$ is $A$-flat; it is also finitely generated since the morphism  $X\to S$ is projective. The $A$-module $L_i$ is flat and finitely generated so that it is a free $A$-module. Denote now with $W^1$ the cokernel of $L_0\to L_1$. The module $W_1$ is finitely generated and  according to \cite[Ex 12.7.2]{Hag} the function $y\mapsto \dim_{k(y)}W_1\otimes_A k(y)$ is upper semicontinuous; moreover since $L_i$ is a free module we can conclude that the function $y\mapsto T^0(k(y))$ is upper semicontinuous as in the proof of \cite[Prop 12.8]{Hag}. We are left to prove that $T^0(k(y))=\Hom_{\XX_y}(\FF_y,\GG_y)$. Since $\EE^{\oplus N_i}\otimes\pi^\ast\OO_X(-m_i)\otimes_A k(y)$ is still locally free we have the following exact diagram:
\begin{displaymath}
  \xymatrix{
 0\ar[d] & 0\ar[d] \\
  \Hom_{\XX}(\FF,\GG_y) \ar@{-->}[r] \ar[d] & \Hom_{\XX_y}(\FF_y,\GG_y)\ar[d] \\
  H^0(X,F_{\EE}(\GG)^{\oplus N_2}(m_2))\otimes_A k(y) \ar[r]^{\widetilde{}} \ar[d] &  H^0(X_y,F_{\EE}(\GG)^{\oplus N_2}(m_2)\otimes_A k(y)) \ar[d] \\
   H^0(X,F_{\EE}(\GG)^{\oplus N_1}(m_1))\otimes_A k(y)\ar[r]^{\widetilde{}} &  H^0(X_y,F_{\EE}(\GG)^{\oplus N_1}(m_1)\otimes_A k(y))\\
}
\end{displaymath}
The first two horizontal  arrows from the bottom are isomorphisms because of \cite[Cor 9.4]{Hag} and Proposition \ref{prop:cohom-base-change-spcoarse} so that the first horizontal arrow is also an isomorphism. 
\end{proof}
\begin{rem}
\begin{enumerate}
\item  The argument of the previous proof is very ad-hoc, even if we believe that a good result of semicontinuity for $\Ext$ functors should hold, we don't know about a general proof.
\item It is suggested by the proof of this lemma that the original Grothendieck's approach to cohomology and base change still holds for stacks, however it relies on Proposition \ref{prop:cohom-base-change-spcoarse}.
\end{enumerate}
\end{rem}

\begin{thm}\label{thm:git-polistabili}
  In the setup of the previous theorem, let $\rho\colon V\otimes\EE\otimes\pi^\ast\OO_X(-m)\to\FF$ be a semistable sheaf in $\QQ$. 
  \begin{enumerate}
  \item The polystable sheaf $\gr^{JH}(\FF)$ $S$-equivalent to $[\rho]$ belongs to the closure of the orbit of $[\rho]$.
  \item The orbit of $[\rho]$ is closed if and only if it is polystable.
  \item Given two semistable sheaves $[\rho]$ and $[\rho']$, their orbits intersect if and only if they are $S$-equivalent.
  \end{enumerate}
\end{thm}
\begin{proof}
Same  proof as in \cite[4.3.3]{MR1450870} with obvious modifications.
\end{proof}
\begin{thm}
  The stack of stable sheaves $\mathcal{S}(\EE,\OO_X(1),P)^s=[\QQ^s/\GL_{N,k}]$ is a $\mathbb{G}_{m,k}$-gerbe over its moduli space $M^s(\OO_X(1),\EE)\coloneqq \QQ^s/GL_{N,k}$ which is a quasi projective scheme.
\end{thm}
\begin{proof}
  Since the orbits of stable sheaves are closed it follows from the previous theorem and \cite[Thm 1.4.1.10]{MR1304906}.
\end{proof}

\begin{thm}
  Denote with $M^{ss}\coloneqq M^{ss}(\OO_X(1),\EE)$ the GIT quotient $\QQ/GL_{N,k}$. Let $\psi$ be the natural morphism $\psi\colon \mathcal{S}(\EE,\OO_X(1),P)=[\QQ/GL_{N,k}]\to M^{ss}(\OO_X(1),\EE)$. 
\begin{enumerate}
\item It has the following universal property: let $Z$ be an algebraic space and \\ $\phi\colon\mathcal{S}(\EE,\OO_X(1),P)\to Z$ a morphism, then there is only one morphism $\theta\colon M^{ss}\to Z$ making the following diagram commute:
  \begin{displaymath}
    \xymatrix{
      \mathcal{S}(\EE,\OO_X(1),P) \ar[dr]_{\phi}\ar[rr]^{\psi} && M^{ss}\ar@{-->}[dl]^{\exists !\theta} \\
       & Z & \\
    }
  \end{displaymath}
\item The natural morphism $\OO_{M^{ss}}\to\psi_\ast\OO_{\mathcal{S}(\EE,\OO_X(1),P)}$ is an isomorphism and the functor $\psi_\ast$ is exact; in different words $M^{ss}$ is a good moduli space in the sense of Alper \cite[Def 4.1]{alper-2008}. Moreover there is an invertible sheaf $\mathcal{M}$ on $M^{ss}$ and an integer $m$ such that denoted $\sigma\colon [\QQ/\SL_{N,k}]\to M^{ss}$ we have:
  \begin{displaymath}
    \sigma^\ast\mathcal{M}\cong\mathcal{L}_l^{\otimes m}
  \end{displaymath}
\item The algebraic stack $\mathcal{S}(\EE,\OO_X(1),P)$ has no moduli space or no tame moduli space in the sense of Alper \cite[Def 7.1]{alper-2008}.
\end{enumerate}
\end{thm}
\begin{proof}
$(1)$  According to \cite[Thm 1.4.1.10]{MR1304906} the map $\overline{\psi}\colon\QQ\to M^{ss}$ factorizing through the stack $[\QQ/\SL_{N,k}]$ and the morphism $[\QQ/\SL_{N,k}]\to[\QQ/\PGL_{N,k}]$ is a categorical quotient and this implies the universal property in the first point for the stack $[\QQ/\SL_{N,k}]$ and actually also for $[\QQ/\PGL_{N,k}]$. The map $\mathcal{S}(\EE,\OO_X(1),P)=[\QQ/\GL_{N,k}]\to [\QQ/\PGL_{N,k}]$ is a gerbe so that it has the universal property in \cite[IV Prop 2.3.18.ii]{Gcna-1971}; this implies that if $ [\QQ/\PGL_{N,k}]$ has the universal property in the statement then also $\mathcal{S}(\EE,\OO_X(1),P)$ has the same universal property.

$(2)$ It is just the stacky interpretation of \cite[1.4.1.10.ii]{MR1304906}.

$(3)$ Theorem \ref{thm:git-polistabili} implies that $M^{ss}$ is not a moduli scheme since its points are in bijection with $S$-equivalence classes, and semistable sheaves can be $S$-equivalent even if not isomorphic. So to speak the scheme $M^{ss}$ has not enough points to be a moduli scheme for $\mathcal{S}(\EE,\OO_X(1),P)$. However it satisfies the universal property in the first point (condition (C) in \cite{MR1432041}) and this implies that if a moduli space exists it is isomorphic to $M^{ss}$. 
\end{proof}




\section{Decorated sheaves.}

\subsection{Moduli of twisted sheaves.}

In this appendix we want to make a more precise comparison between our result on semistable sheaves on gerbes and analogous results in \cite{MR2309155} and \cite{MR2306170}.
In this section $\pi\colon\XX\to X$ is a $G$-banded gerbe over $X$ where $X$ is a projective scheme over an algebraically closed field $k$ and $G$ is a diagonalizable group scheme over $X$ (its Cartier dual is constant). With the word $G$-banded we mean banded by the trivial $G$-torsor on $X$. This assumption is very restrictive but fundamentel for the results that follow. The stack $\XX$ can have non finite inertia; the most interesting case, that is $G=\mathbb{G}_m$, has non finite inertia so that $\XX$ is not tame according to the definition in \cite{MR2427954}. However $\pi_\ast$ is exact and $X$ is the moduli space (a tame moduli space according to Alper) of $\XX$ so that all the construction of the moduli space of semistable sheaves still holds as far as  a generating sheaf exists. For sure we can say that there is an ample locally free sheaf (of finite rank) if the $\mathbb{G}_m$-banded gerbe is a torsion element in $H^2(X,\mathbb{G}_m)$  (see \cite[Thm 1.3.5]{cualduararu-dcotsocm2000}). The generating sheaf can be chosen to be a sum over the integers of the powers of the chosen ample locally free sheaf. 

Let $(\chi_1,\ldots,\chi_n)$ be the characters of $G$; in the following we will prove that the moduli space of semistable sheaves on $\XX$ is made of connected of components labelled by  characters and each of these is the moduli space of $\chi$-twisted sheaves on $X$.
\begin{lem}\label{lem:flatness-and-twisted}
  Let $\pi\colon\XX\to X$ be a $G$-banded gerbe and $q\colon X\to S$ a projective morphism of finite type schemes over a field. Fix $\EE,\OO_X(1)$ a polarization. Let $\FF$ be a coherent torsion-free sheaf on $\XX$ flat on $S$ and $\FF=\bigoplus_{\chi\in C(G)} \FF_\chi $ its  decomposition in eigensheaves. Let $P=P_\EE(\FF_s)$ be the modified Hilbert polynomial of the geometric fiber over $s$ a point of $S$. The polynomial $P$ splits as $P=\sum_{i=1}^n P_{\chi}$ where $P_{\chi}(m)=\chi(X_s,\pi_\ast(\FF_{\chi}\otimes\EE_{\chi}^\vee)(m)\vert_{X_s})$ is locally constant. 
\end{lem}
\begin{proof}
  Since $\FF$ is $S$-flat each summand $\FF_\chi$ is $S$-flat, moreover we observe that $\XX_s$ is again a $G$-banded gerbe and that the decomposition in eigensheaves is compatible with the restriction to the fiber. Applying To\"en-Riemann-Roch we have:
  \begin{displaymath}
    P_{\EE}(\FF\vert_{\XX_s},m)=\sum_{i=1}^n\chi(X_s,\pi_\ast(\FF_{\chi_i}\otimes\EE_{\chi_i}^\vee)(m)\vert_{X_s})=\sum_{i=1}^n P_{\chi_i}(m)
  \end{displaymath}
Each $P_{\chi_i}$ is locally constant according to  \cite[7.9.4]{MR0163911}.
\end{proof}
The family of semistable sheaves is bounded; this implies that there is a scheme $U\to S$ and a bounding family of sheaves parameterized by $U$. In general there is no reason why $U$ should be connected and the polynomial $P_\chi$ is just locally constant on $U$. Denote with $\mathfrak{P}$ the set of $n$-tuples  $\mathcal{P}_{\chi}=(P_{\chi_1},\ldots,P_{\chi_n})$ of polynomials labelled by $(\chi_1,\ldots,\chi_n)$ the irreducible representations of $G$ such that $\sum_{i=1}^n P_{\chi_i}=P$.  A priori we have components of $U$ for every such an $n$-tuple of polynomials.
Let $\GG$ be a locally free sheaf. Thanks to the lemma it makes sense to define the functor $\Quot_{\XX/k}(\mathcal{G}_{\chi},P_{\chi})$ of  quotients $\FF_{\chi}$ with modified Hilbert polynomial $P_{\chi}(m)=\chi(X,\pi_\ast(\FF_{\chi}\otimes\EE_{\chi}^\vee)(m))$. Every quotient must be a $\chi$-twisted sheaf or zero because the only morphism between sheaves twisted by different characters is the zero morphism. For every $n$-tuple $\mathcal{P}_{\chi}=(P_{\chi_1},\ldots,P_{\chi_n})$ we have a natural  transformation:
\begin{displaymath}
  \coprod_{i=1}^n\Quot_{\XX/k}(\mathcal{G}_{\chi_i},P_{\chi_i})\xrightarrow{\iota} \Quot_{\XX/k}(\mathcal{G},P)
\end{displaymath}
and they are monomorphisms (of sets) since there are evident sections.
\begin{lem}\label{lem:twist-dec-quot}
 Let $\GG$ be a locally free sheaf on $\XX$ and $\GG=\bigoplus_{i=1}^n \GG_{\chi_i} $ its  decomposition in eigensheaves . Fix $P$ a rational polynomial of degree $d=\dim{X}$. The natural transformation:
 \begin{displaymath}
   \coprod_{\mathcal{P}_{\chi}\in\mathfrak{P}} \coprod_{i=1}^n  \Quot_{\XX/k}(\mathcal{G}_{\chi_i},P_{\chi_i})\xrightarrow{\coprod\iota_{\chi}} \Quot_{\XX/k}(\mathcal{G},P)
 \end{displaymath}
is relatively representable, surjective and a closed immersion.
\end{lem}
\begin{proof}
Since there are no morphisms between sheaves twisted by different characters it is clear that the image of each natural transformation is disjoint from the others and the coproduct of all of them covers the target. We have just to prove that each $\iota_\chi$ is relatively representable and a closed immersion. To prove this we first observe that, while $\EE_\chi$ is not a generating sheaf, it is a generating sheaf for every $\chi$-twisted sheaf. Having this in mind, the result follows with the same proof as in \cite[Prop 6.2]{MR2007396} but using Lemma \ref{lem:flatness-and-twisted} instead of \cite[Lem 4.3]{MR2007396}.
\end{proof}
 Let $N,V,m$ be as in Proposition \ref{prop:stack-dei-moduli}. Let $\mathcal{Q}$ be the open subscheme of $\Quot_{\XX/k}(V\otimes\mathcal{E}\otimes\pi^\ast\OO_X(-m),P)$ defined in Proposition \ref{prop:stack-dei-moduli} and denote with $\mathcal{Q}_\chi$ its intersection with $\Quot_{\XX/k}(\mathcal{G}_{\chi},P_{\chi})$. We observe that when a sheaf $\FF$ is $m$-regular with $N=P_{\mathcal{E}}(\FF)(m)$ every summand $\FF_{\chi_i}$ is $m$-regular with $N_i=P_{\chi_i}(\FF_{\chi_i})(m)$ and the sum $\sum_{i=1}^n N_i$ is $N$. Since each $N_i$ is not greater than $N$ we can generate each summand $\FF_{\chi_i}$ using $N$ global sections. After this remark we can state the following result:
\begin{prop}
The moduli stack of torsion-free semistable sheaves on $\XX$ with fixed modified Hilbert polynomial $P$ is made of the following connected components:
\begin{displaymath}
  [\mathcal{Q}/GL_{N,k}]\cong \coprod_{\mathcal{P}_{\chi}\in\mathfrak{P}}\coprod_{i=1}^n[\mathcal{Q}_{\chi_i}/GL_{N,k}]
\end{displaymath}
In the same way the good moduli scheme of $[\mathcal{Q}/GL_{N,k}]$ decomposes in connected components:
\begin{displaymath}
  \mathcal{Q}/GL_{N,k}\cong \coprod_{\mathcal{P}_{\chi}\in\mathfrak{P}}\coprod_{i=1}^n\mathcal{Q}_{\chi_i}/GL_{N,k}
\end{displaymath}
and each of them is the good moduli scheme of $[\mathcal{Q}_{\chi_i}/GL_{N,k}]$
\end{prop}
\begin{proof}
  The first statement is an immediate consequence of Lemma \ref{lem:twist-dec-quot}. Since each $\mathcal{E}_\chi$ is a generating sheaf for the subcategory of quasicoherent $\chi$-twisted sheaves all the results in section $7.3$ can be reproduced for each quotient $\mathcal{Q}/GL_{N,k}$ and this implies the second statement.
\end{proof}
If the group scheme $G$ is $\mathbb{G}_m$ or $\mu_a$ for some integer $a$ each $\mathcal{Q}_\chi/GL_{N,k}$ is the moduli scheme of $\chi$-twisted sheaves constructed by Yoshioka for an evident choice of the generating sheaf $\mathcal{E}$. To obtain exactly the same moduli scheme produced by Lieblich it's enough to choose the generating sheaf $\EE$ such that each summand $\EE_\chi$ has trivial Chern classes $c_i(\EE_\chi)$ for $i=1,\ldots,\dim{\XX}$.

\subsection{Moduli of parabolic sheaves}

In this section we explicitly compare our construction with the moduli space of semistable sheaves produced by Maruyama and Yokogawa in \cite{MR1162674}. The fundamental ingredient in this comparison is the equivalence between coherent sheaves on a root stack $\XX$ and parabolic sheaves on its moduli scheme $X$ proved by Borne in \cite{borne-2006}. 

Let $S$ be a finite type scheme over an algebraically closed field and $f\colon X\to S$ an $S$-projective scheme with $\OO_X(1)$ an $f$-very ample invertible sheaf. Let $D$ be a relative ($\OO_S$-flat) effective Cartier divisor. We will not use the most general notion of parabolic sheaf; for the precise definition we use, see \cite[Def 2.1.1]{borne-2006}. Roughly speaking, let $d\geq 1$ an integer (coprime with the characteristic of every point of $S$) and $F_\bullet$ a coherent sheaf with a quasi-parabolic structure in the sense of Maruyama of length $d$, that is a coherent sheaf $F$ with a length $d$ filtration $F_d\subseteq F_{d-1}\subseteq \ldots \subseteq F_0$ where  $F_0=F$ and $F_d$ is the image of $F(-D)$ in $F$; the parabolic structure assigns weight $\frac{l}{d}$ to the sheaf $F_{l}$ in the filtration where $0\leq l\leq d$. We will denote with $\text{Par}_{\frac{1}{d}}(X,D)$ the category of coherent parabolic sheaves on $X$ with chosen divisor $D$ and multiplicity $d$.

Till now we have not assumed $D$ to be smooth since the construction doesn't rely on smoothness, but we want to recall that for non smooth divisors this notion of parabolic sheaves is not the natural one (see \cite{borne-2006} and \cite{borne-2007} for more details), beside not assuming smoothness introduces a few minor complications that we prefer to avoid, so from now on we will assume $D$ to be smooth.

Let $\XX$ be the root stack $\sqrt[d]{D/X}$ and $\mathcal{D}=(\pi^{-1}D)_{\text{red}}$ is the orbifold divisor. We fix as generating sheaf the locally free sheaf $\EE=\bigoplus_{l=0}^d\OO_\XX(l\mathcal{D})$. We recall now the definition of the functor $F_{\mathcal{D}}:\text{Coh}(\XX)\to \text{Par}_{\frac{1}{d}}(X,D)$ exhibiting the equivalence between the two categories and of its quasi-inverse as constructed by Borne.
\begin{displaymath}
\xymatrix@R=0pt{
  \text{Coh}_{\XX} \ar[r]^{F_{\mathcal{D}}} &  \text{Par}_{\frac{1}{d}}(X,D) \\ 
   \FF \ar@{|->}[r] & F_{\mathcal{D}}(\FF);\; F_{\mathcal{D}}(\FF)_l=\pi_\ast(\FF\otimes\OO_{\XX}(-l\mathcal{D})) 
}
\end{displaymath}
This is related to $F_\EE$ defined previously in the obvious way: $\bigoplus_{l=0}^d F_{\mathcal{D}}(\FF)_l=F_{\EE}(\FF)$.
We denote with $\mathbb{Z}$ the category of integers where maps are given by the natural ordering $(\geq)$. Given a parabolic sheaf $F_\bullet$ we define a functor:
\begin{displaymath}
\xymatrix@R=0pt{
 \mathbb{Z}^\circ\times\mathbb{Z}  \ar[r]^{g_{\mathcal{D}}(F_\bullet)} & \text{Coh}_{\XX} \\
  l,m \ar@{|->}[r] & \OO_{\XX}(l\mathcal{D})\otimes \pi^\ast F_m \\
  }
\end{displaymath}
We define now the functor:
\begin{displaymath}
  \xymatrix@R=0pt{
 \text{Par}_{\frac{1}{d}}(X,D) \ar[r]^{G_{\mathcal{D}}} &  \text{Coh}_{\XX} \\
 F_\bullet \ar@{|->}[r] & \int^\mathbb{Z} g_{\mathcal{D}}(F_\bullet)(l,l) \\
}
\end{displaymath}
We denote with $\int^\mathbb{Z} g_{\mathcal{D}}(F_\bullet)(l,l)$ the colimit of wedges:
\begin{displaymath}
  \xymatrix{
 g_{\mathcal{D}}(F_\bullet)(l,m) \ar[r]^-{f_{l,m}}\ar[d]^-{h_{l,m}} & g_{\mathcal{D}}(F_\bullet)(l,l)\ar[d]^-{\omega(l)} \\
 g_{\mathcal{D}}(F_\bullet)(m,m) \ar[r]_-{\omega(m)} & \mathcal{G} \\
}
\end{displaymath}
where $m\geq l$ is an arrow in $\mathbb{Z}$, the arrow $h_{l,m}$ is induced by the canonical section of the divisor, the arrow $f_{l,m}$ is induced by the filtration $\pi^\ast F_\bullet$, the arrow $\omega(l)$ is a dinatural transformation and $\mathcal{G}$ is a sheaf in  $\text{Coh}(\XX)$. Such a colimit is also called a \textit{coend}, and more details about it can be found in \cite[IX.5]{MR1712872}. In the computation of the colimit attention should be paid to the fact that $ g_{\mathcal{D}}(F_\bullet)(l,l)$ and $ g_{\mathcal{D}}(F_\bullet)(l+d,l+d)$ are identified (by the parabolic structure), so that the two arrows leaving from them $\omega(l)$ and $\omega(l+d)$ must be equal\footnotemark \footnotetext{The reader is probably familiar with the case of the self-gluing of a scheme along two non intersecting closed sub-schemes, even in that case we compute a colimit that differs from the push-out for the fact that the two arrows to the colimit must be equal.}. 

\begin{defn}\label{def:t-f-par}
  We define a parabolic sheaf $F_\bullet\in\text{Par}_{\frac{1}{d}}(X,D)$ to be  torsion free if $F_0$ is torsion free.
\end{defn}

The functor $F_{\mathcal{D}}$ clearly maps torsion free sheaves on the stack to torsion free parabolic sheaves (see Proposition \ref{prop:pure-sheaves}), the vice-versa is not completely obvious.
\begin{lem}
  The functor $G_{\mathcal{D}}$ maps torsion free sheaves on $X$ to torsion free sheaves on the corresponding root stack $\XX$. 
\end{lem}
\begin{proof}
  Let $F_\bullet $ be a torsion free parabolic sheaf and suppose that $\FF:=G_{\mathcal{D}}(F_\bullet)$ has some torsion subsheaf $\FF'$. We can apply to it the functor $F_{\mathcal{D}}$ and obtain injections $F_{\mathcal{D}}(\FF')_i \to F_i$. Each of the sheaves $F_{\mathcal{D}}(\FF')_i$ can be zero, however the direct sum $\bigoplus_{i=0}^d F_{\mathcal{D}}(\FF')_i$ cannot vanish according to Proposition \ref{prop:pure-sheaves}; this implies that one of the $F_i$ must have torsion, against the hypothesis. 
\end{proof}

The two functors $F_{\mathcal{D}}$ and $G_{\mathcal{D}}$ are inverse to each other when applied to torsion free sheaves. A proof of this will appear in a joint work of Borne and Vistoli (ref). If we consider non necessarily torsion free sheaves the functor $F_{\mathcal{D}}$ is again an equivalence of categories but its quasi-inverse is a more general functor than $G_{\mathcal{D}}$.
\begin{rem}
  It is clear from the construction of $F_{\mathcal{D}}$ that given a torsion free parabolic sheaf $F_\bullet$ all the cokernels $F_i\to F_j\to Q_{ij}$ with $i>j$ must be pure sheaves supported on the divisor or zero. This property has not been included in the definition of torsion free parabolic sheaf because it can be retrieved from our notion of torsion free parabolic sheaf. First of all we observe that in the following exact sequence:
  \begin{displaymath}
    \xymatrix{
      0 \ar[r] & F_i(-D) \ar[r] & F_{i} \ar[r] & F_i\vert_{D} \ar[r] & 0 \\
}
  \end{displaymath}
the cokernel $F_i\vert_{D}$ is pure and supported on the divisor. Assume now $d\geq i>j \geq 0$, we can draw the following exact and commutative diagram:
\begin{displaymath}
  \xymatrix{
    0 \ar[dr] & & & & 0\ar[dl] & \\
    0 \ar[r] & F_i \ar[dr] \ar[rr] & & F_j \ar[r]\ar[dl] & Q_{ij}\ar[r]\ar[dd]^-{\wr} & 0 \\
      & & F_i(D)\ar[dr]\ar[dl] & & & \\
    0 & R_{ij}\ar[l]\ar[dl] & & F_i(D)\vert_{D}\ar[ll]\ar[dr] & Q_{ij}\ar[l] & 0\ar[l] \\ 
    0 & & & & 0 & \\
}
\end{displaymath}
Since the sheaf $ F_i(D)\vert_{D}$ is pure and supported on the divisor also the sheaf $Q_{ij}$ is pure with the same support or zero.
\end{rem}
\begin{rem}
  Let $U\to S$ be a finite type scheme and denote with $X_U$ the product $X\times_S U$ with projection $p_1\colon X_U\to X$. We recall that the root construction is compatible with base change so that parabolic sheaves on $X_U$ are equivalent to coherent sheaves on the root construction obtained by base change $\XX_U$.  
\end{rem}

Since we are interested in parameterizing parabolic sheaves on the fibers of $f$ we need some notion of flatness, otherwise  fibers of a torsion free parabolic sheaf on $X$ are not parabolic sheaves in general.  
\begin{defn}
  We define a family of parabolic sheaves $F_\bullet\in\text{Par}_{\frac{1}{d}}(X_U,p_1^\ast D)$ to be $\OO_U$-flat if for every $l,m$ (where $m>l$)  every cokernel $F_l\to F_m\to Q_{l,m}$ is $\OO_U$-flat.
\end{defn}
This definition is enough to guarantee that fibers of a flat family of parabolic sheaves are parabolic sheaves.
\begin{lem}
  Let $f\colon X\to S$ be as before , let $D$ be a relative smooth effective Cartier divisor, $d\geq 1$ an integer. The functor $F_\mathcal{D}$ maps flat families of torsion free sheaves on $\XX :=\sqrt[d]{D/ X}$ to flat families of torsion free parabolic sheaves on $X$. The functor $G_\mathcal{D}$ maps flat families of torsion free parabolic sheaves on $X$ to flat families of torsion free sheaves on $\XX$.
\end{lem}
\begin{proof}
  The first statement is an immediate consequence of Corollary \ref{cor:tame-stack-2}~\ref{item:3} remembering that the divisor $D$ is $\OO_S$-flat and the morphism $\pi$ is flat too. To prove the second let $F_\bullet$ be a flat family of torsion free parabolic sheaves and compute $\FF=G_{\mathcal{D}}(F_\bullet)$. If this is not flat there is some ideal $I$ in $\OO_S$ corresponding to a closed subscheme $Z$ such that in the sequence:
  \begin{displaymath}
    \xymatrix{
 0\ar[r] & \Tor^{\OO_S}_1{(\FF,\OO_Z)} \ar[r] &  \FF\otimes_{\OO_S} I \ar[r] & \FF \ar[r] & \FF\otimes_{\OO_S} \OO_Z \ar[r] & 0 \\
}
  \end{displaymath}
the sheaf $\Tor^{\OO_S}_1{(\FF,\OO_Z)}$ is not zero. We can apply the functor $F_{\mathcal{D}}$ to this sequence to obtain exact sequences:
  \begin{displaymath}
    \xymatrix{
 0\ar[r] & F_{\mathcal{D}}(\Tor^{\OO_S}_1{(\FF,\OO_Z)})_i \ar[r] &  F_i\otimes_{\OO_S} I \ar[r] & F_i \ar[r] & F_i\otimes_{\OO_S} \OO_Z \ar[r] & 0 \\
}
  \end{displaymath}
where we have used the projection formula and Proposition \ref{prop:cohom-base-change-spcoarse} on the last term on the right. As we have observed before $F_{\mathcal{D}}(\Tor^{\OO_S}_1{(\FF,\OO_Z)})_i$ can vanish for some $i$ but not for every $i$. This implies that some $F_i$ is not $\OO_S$-flat against the hypothesis. 
\end{proof}
At this point we have proven that stability condition for sheaves on a root stack $\XX$ is equivalent to parabolic stability when the generating sheaf is chosen to be $\EE=\bigoplus_{i=0}^{d} \OO_\XX(i\mathcal{D})$, and flat families of torsion free sheaves on $\XX$ are equivalent to flat families of torsion free parabolic sheaves (in the sense of Maruyma and Yokogawa); we can conclude that the moduli space of sheaves on a root stack is equivalent to the moduli space of parabolic sheaves.
\begin{cor}
  Let $X$ be a smooth projective variety over an algebraically closed field $k$, $D$ an effective smooth Cartier divisor, $d\geq 1$ an integer and $\OO_X(1)$ a very ample invertible sheaf. Denote with $\XX$ the root stack associated to these data and $\EE=\bigoplus_{i=0}^{d} \OO_\XX(i\mathcal{D})$ is the chosen generating sheaf. Fix some polynomial $P\in\mathbb{Q}[m]$, the moduli space  $M^{ss}(\OO_X(1),\EE,P)$ of semistable sheaves with modified Hilbert polynomial $P$ is isomorphic to the coarse moduli space of the functor of semistable parabolic sheaves defined in \cite[Def 1.14]{MR1162674} and $M^{s}(\OO_X(1),\EE)$ is isomorphic to the moduli space of stable parabolic sheaves in \cite[Thm 3.6]{MR1162674}.     
\end{cor}

\bibliographystyle{amsalpha}
\bibliography{bibliostack}

\providecommand{\bysame}{\leavevmode\hbox to3em{\hrulefill}\thinspace}
\providecommand{\MR}{\relax\ifhmode\unskip\space\fi MR }
\providecommand{\MRhref}[2]{%
  \href{http://www.ams.org/mathscinet-getitem?mr=#1}{#2}
}
\providecommand{\href}[2]{#2}
\begin{thebibliography}{EHKV01}

\bibitem[AGV08]{MR2450211}
Dan Abramovich, Tom Graber, and Angelo Vistoli, \emph{Gromov-{W}itten theory of
  {D}eligne-{M}umford stacks}, Amer. J. Math. \textbf{130} (2008), no.~5,
  1337--1398. \MR{MR2450211 (2009k:14108)}

\bibitem[Alp08]{alper-2008}
Jarod Alper, \emph{Good moduli spaces for Artin stacks}, 2008.

\bibitem[AOV08]{MR2427954}
Dan Abramovich, Martin Olsson, and Angelo Vistoli, \emph{Tame stacks in
  positive characteristic}, Ann. Inst. Fourier (Grenoble) \textbf{58} (2008),
  no.~4, 1057--1091. \MR{MR2427954}

\bibitem[AV02]{MR1862797}
Dan Abramovich and Angelo Vistoli, \emph{Compactifying the space of stable
  maps}, J. Amer. Math. Soc. \textbf{15} (2002), no.~1, 27--75 (electronic).
  \MR{MR1862797 (2002i:14030)}

\bibitem[Bis97]{MR1455522}
Indranil Biswas, \emph{Parabolic bundles as orbifold bundles}, Duke Math. J.
  \textbf{88} (1997), no.~2, 305--325. \MR{MR1455522 (98m:14045)}

\bibitem[Bor06]{borne-2006}
Niels Borne, \emph{Fibr{\'e}s paraboliques et champ des racines}, 2006.

\bibitem[Bor07]{borne-2007}
\bysame, \emph{Sur les repr{\'e}sentations du groupe fondamental d'une
  vari{\'e}t{\'e} priv{\'e}e d'un diviseur {\`a} croisements normaux simples},
  2007.

\bibitem[Cad07]{Cstc-2007}
Charles Cadman, \emph{{Using stacks to impose tangency conditions on curves}},
  Amer. J. Math. \textbf{129} (2007), no.~2, 405--427.

\bibitem[C{\u{a}}l00]{cualduararu-dcotsocm2000}
Andrei C{\u{a}}ld{\u{a}}raru, \emph{Derived categories of twisted sheaves on
  Calabi-Yau manifolds}, 2000, Ph.D. Thesis, Cornell University.

\bibitem[DP03]{donagi-2003}
Ron Donagi and Tony Pantev, \emph{Torus fibrations, gerbes, and duality}, 2003.

\bibitem[EG95]{MR1355920}
Geir Ellingsrud and Lothar G{\"o}ttsche, \emph{Variation of moduli spaces and
  {D}onaldson invariants under change of polarization}, J. Reine Angew. Math.
  \textbf{467} (1995), 1--49. \MR{MR1355920 (96h:14009)}

\bibitem[EHKV01]{MR1844577}
Dan Edidin, Brendan Hassett, Andrew Kresch, and Angelo Vistoli, \emph{Brauer
  groups and quotient stacks}, Amer. J. Math. \textbf{123} (2001), no.~4,
  761--777. \MR{MR1844577 (2002f:14002)}

\bibitem[Gir71]{Gcna-1971}
Jean Giraud, \emph{Cohomologie non ab\'elienne}, Springer-Verlag, Berlin, 1971,
  Die Grundlehren der mathematischen Wissenschaften, Band 179.

\bibitem[EGAIII.1]{MR0217085}
A.~Grothendieck, \emph{{\'E}l\'ements de g\'eom\'etrie alg\'ebrique. iii.
  {\'e}tude cohomologique des faisceaux coh\'erents. i}, Inst. Hautes \'Etudes
  Sci. Publ. Math. (1961), no.~11, 167. \MR{MR0217085 (36 \#177c)}

\bibitem[EGAIII.2]{MR0163911}
\bysame, \emph{{\'E}l\'ements de g\'eom\'etrie alg\'ebrique. iii. {\'e}tude
  cohomologique des faisceaux coh\'erents. ii}, Inst. Hautes \'Etudes Sci.
  Publ. Math. (1963), no.~17, 91. \MR{MR0163911 (29 \#1210)}

\bibitem[EGAIV.2]{MR0199181}
\bysame, \emph{{\'E}l\'ements de g\'eom\'etrie alg\'ebrique. iv. {\'e}tude
  locale des sch\'emas et des morphismes de sch\'emas. ii}, Inst. Hautes
  \'Etudes Sci. Publ. Math. (1965), no.~24, 231. \MR{MR0199181 (33 \#7330)}

\bibitem[EGAIV.3]{MR0217086}
\bysame, \emph{{\'E}l\'ements de g\'eom\'etrie alg\'ebrique. iv. {\'e}tude
  locale des sch\'emas et des morphismes de sch\'emas. iii}, Inst. Hautes
  \'Etudes Sci. Publ. Math. (1966), no.~28, 255. \MR{MR0217086 (36 \#178)}

\bibitem[Har77]{Hag}
Robin Hartshorne, \emph{Algebraic geometry}, Springer-Verlag, New York, 1977,
  Graduate Texts in Mathematics, No. 52. \MR{57 \#3116}

\bibitem[HL97]{MR1450870}
Daniel Huybrechts and Manfred Lehn, \emph{The geometry of moduli spaces of
  sheaves}, Aspects of Mathematics, E31, Friedr. Vieweg \& Sohn, Braunschweig,
  1997. \MR{MR1450870 (98g:14012)}

\bibitem[KM97]{MR1432041}
Se{\'a}n Keel and Shigefumi Mori, \emph{Quotients by groupoids}, Ann. of Math.
  (2) \textbf{145} (1997), no.~1, 193--213. \MR{MR1432041 (97m:14014)}

\bibitem[Kre99]{MR1719823}
Andrew Kresch, \emph{Cycle groups for Artin stacks}, Invent. Math. \textbf{138}
  (1999), no.~3, 495--536. \MR{MR1719823 (2001a:14003)}

\bibitem[Kre06]{geomDM}
\bysame, \emph{On the geometry of Deligne-Mumford stacks}, 2006, to appear on
  the Seattle proceedings.

\bibitem[Lan04a]{MR2085175}
Adrian Langer, \emph{Moduli spaces of sheaves in mixed characteristic}, Duke
  Math. J. \textbf{124} (2004), no.~3, 571--586. \MR{MR2085175 (2005g:14082)}

\bibitem[Lan04b]{MR2051393}
\bysame, \emph{Semistable sheaves in positive characteristic}, Ann. of Math.
  (2) \textbf{159} (2004), no.~1, 251--276. \MR{MR2051393 (2005c:14021)}

\bibitem[Lie06]{MR2233719}
Max Lieblich, \emph{Remarks on the stack of coherent algebras}, Int. Math. Res.
  Not. (2006), Art. ID 75273, 12. \MR{MR2233719}

\bibitem[Lie07]{MR2309155}
\bysame, \emph{Moduli of twisted sheaves}, Duke Math. J. \textbf{138} (2007),
  no.~1, 23--118. \MR{MR2309155}

\bibitem[LMB00]{LMBca}
G{\'e}rard Laumon and Laurent Moret-Bailly, \emph{Champs alg\'ebriques},
  Ergebnisse der Mathematik und ihrer Grenzgebiete. 3. Folge. A Series of
  Modern Surveys in Mathematics [Results in Mathematics and Related Areas. 3rd
  Series. A Series of Modern Surveys in Mathematics], vol.~39, Springer-Verlag,
  Berlin, 2000.

\bibitem[Mat80]{MR575344}
Hideyuki Matsumura, \emph{Commutative algebra}, second ed., Mathematics Lecture
  Note Series, vol.~56, Benjamin/Cummings Publishing Co., Inc., Reading, Mass.,
  1980. \MR{MR575344 (82i:13003)}

\bibitem[Mat89]{MR1011461}
\bysame, \emph{Commutative ring theory}, second ed., Cambridge Studies in
  Advanced Mathematics, vol.~8, Cambridge University Press, Cambridge, 1989,
  Translated from the Japanese by M. Reid. \MR{MR1011461 (90i:13001)}

\bibitem[MFK94]{MR1304906}
D.~Mumford, J.~Fogarty, and F.~Kirwan, \emph{Geometric invariant theory}, third
  ed., Ergebnisse der Mathematik und ihrer Grenzgebiete (2) [Results in
  Mathematics and Related Areas (2)], vol.~34, Springer-Verlag, Berlin, 1994.
  \MR{MR1304906 (95m:14012)}

\bibitem[Mil80]{Met}
James~S. Milne, \emph{{\'E}tale cohomology}, Princeton Mathematical Series,
  vol.~33, Princeton University Press, Princeton, N.J., 1980. \MR{MR559531
  (81j:14002)}

\bibitem[ML98]{MR1712872}
Saunders Mac~Lane, \emph{Categories for the working mathematician}, second ed.,
  Graduate Texts in Mathematics, vol.~5, Springer-Verlag, New York, 1998.
  \MR{MR1712872 (2001j:18001)}

\bibitem[MR071]{MR0354655}
\emph{Th\'eorie des intersections et th\'eor\`eme de Riemann-Roch},
  Springer-Verlag, Berlin, 1971, S\'eminaire de G\'eom\'etrie Alg\'ebrique du
  Bois-Marie 1966--1967 (SGA 6), Dirig\'e par P. Berthelot, A. Grothendieck et
  L. Illusie. Avec la collaboration de D. Ferrand, J. P. Jouanolou, O. Jussila,
  S. Kleiman, M. Raynaud et J. P. Serre, Lecture Notes in Mathematics, Vol.
  225. \MR{MR0354655 (50 \#7133)}

\bibitem[MS80]{MR575939}
V.~B. Mehta and C.~S. Seshadri, \emph{Moduli of vector bundles on curves with
  parabolic structures}, Math. Ann. \textbf{248} (1980), no.~3, 205--239.
  \MR{MR575939 (81i:14010)}

\bibitem[Mum66]{MR0209285}
David Mumford, \emph{Lectures on curves on an algebraic surface}, With a
  section by G. M. Bergman. Annals of Mathematics Studies, No. 59, Princeton
  University Press, Princeton, N.J., 1966. \MR{MR0209285 (35 \#187)}

\bibitem[Mum70]{MR0282985}
\bysame, \emph{Abelian varieties}, Tata Institute of Fundamental Research
  Studies in Mathematics, No. 5, Published for the Tata Institute of
  Fundamental Research, Bombay, 1970. \MR{MR0282985 (44 \#219)}

\bibitem[MW97]{MR1433203}
Kenji Matsuki and Richard Wentworth, \emph{Mumford-Thaddeus principle on the
  moduli space of vector bundles on an algebraic surface}, Internat. J. Math.
  \textbf{8} (1997), no.~1, 97--148. \MR{MR1433203 (97m:14010)}

\bibitem[MY92]{MR1162674}
M.~Maruyama and K.~Yokogawa, \emph{Moduli of parabolic stable sheaves}, Math.
  Ann. \textbf{293} (1992), no.~1, 77--99. \MR{MR1162674 (93d:14022)}

\bibitem[Nir09]{groth-duality-stack}
Fabio Nironi, \emph{{Grothendieck duality for Deligne-Mumford stacks}},  arXiv:0811.1955v2 (2009).

\bibitem[Ols07]{Osas05}
Martin Olsson, \emph{{Sheaves on Artin stacks}}, J. Reine Angew. Math.
  \textbf{603} (2007), 55--112.

\bibitem[OS03]{MR2007396}
Martin Olsson and Jason Starr, \emph{Quot functors for Deligne-Mumford stacks},
  Comm. Algebra \textbf{31} (2003), no.~8, 4069--4096, Special issue in honor
  of Steven L. Kleiman. \MR{MR2007396 (2004i:14002)}

\bibitem[OSS80]{MR561910}
Christian Okonek, Michael Schneider, and Heinz Spindler, \emph{Vector bundles
  on complex projective spaces}, Progress in Mathematics, vol.~3, Birkh\"auser
  Boston, Mass., 1980. \MR{MR561910 (81b:14001)}

\bibitem[Pot92]{potier-ldmds1992}
J.~Le Potier, \emph{L'espace de modules de Simpson}, S{\'e}minaire de
  g{\'e}om{\'e}trie alg{\'e}brique, Jussieu (1992).

\bibitem[Sim94]{MR1307297}
Carlos~T. Simpson, \emph{Moduli of representations of the fundamental group of
  a smooth projective variety. i}, Inst. Hautes \'Etudes Sci. Publ. Math.
  (1994), no.~79, 47--129. \MR{MR1307297 (96e:14012)}

\bibitem[Toe99]{MR1710187}
B.~Toen, \emph{Th\'eor\`emes de Riemann-Roch pour les champs de
  Deligne-Mumford}, $K$-Theory \textbf{18} (1999), no.~1, 33--76. \MR{MR1710187
  (2000h:14010)}

\bibitem[Vis89]{Vitas-1989}
Angelo Vistoli, \emph{Intersection theory on algebraic stacks and on their
  moduli spaces}, Invent. Math. \textbf{97} (1989), no.~3, 613--670.

\bibitem[Yos03]{MR2019443}
K{\=o}ta Yoshioka, \emph{Twisted stability and Fourier-Mukai transform. i},
  Compositio Math. \textbf{138} (2003), no.~3, 261--288. \MR{MR2019443
  (2004j:14015)}

\bibitem[Yos06]{MR2306170}
\bysame, \emph{Moduli spaces of twisted sheaves on a projective variety},
  Moduli spaces and arithmetic geometry, Adv. Stud. Pure Math., vol.~45, Math.
  Soc. Japan, Tokyo, 2006, pp.~1--30. \MR{MR2306170}

\end{thebibliography}

\end{document}